\documentclass[12pt]{article}

\usepackage{amsmath}
\usepackage{amsthm}
\usepackage{amssymb}
\usepackage{amscd}
\usepackage{xspace}
\usepackage{verbatim}
\newtheorem{theorem}{Theorem}[section]
\newtheorem{corollary}{Corollary}[section]
\newtheorem{lemma}{Lemma}[section]
\newtheorem{proposition}{Proposition}[section]
\theoremstyle{definition}
\newtheorem{definition}{Definition}[section]

\theoremstyle{remark}
\newtheorem{remark}{Remark}[section]
\numberwithin{equation}{section}

\newcommand{\rth}[1]{Theorem~\ref{#1}}

\newcommand\blfootnote[1]{%
  \begingroup
  \renewcommand\thefootnote{}\footnote{#1}%
  \addtocounter{footnote}{-1}%
  \endgroup
}

\newcommand{\ov}{\overline}

\newcommand{\e}{\varepsilon}

\renewcommand{\O}{\Omega}

\renewcommand{\vec}[1]{\mathbf{#1}}
\newcommand{\field}[1]{\mathbb{#1}}
\newcommand{\C}{\field{C}}
\newcommand{\R}{\field{R}}

\newcommand{\er}{\eqref}

\DeclareMathOperator{\Div}{div} \DeclareMathOperator{\dist}{dist}
\DeclareMathOperator{\supp}{supp}

\renewcommand{\O}{\Omega}

\newcommand{\f}{\varphi}
\renewcommand{\vec}[1]{\boldsymbol{#1}}
\DeclareMathOperator{\essinf}{ess\,inf}

\date{}

\begin{document}
\title{Some remarks on a formula for Sobolev norms due to Brezis, Van
  Schaftingen and Yung}
\maketitle
\begin{center}
\textsc{Arkady Poliakovsky \footnote{E-mail:
poliakov@math.bgu.ac.il}
}\\[3mm]
Department of Mathematics, Ben Gurion University of the Negev,\\
P.O.B. 653, Be'er Sheva 84105, Israel
\\[2mm]
\end{center}

\begin{abstract}
    We provide answers to some questions raised in a recent work by H.\,Brezis, J.\,Van
    Schaftingen and Po-Lam\,Yung~\cite{hhh3,hhh2} concerning the  Gagliardo semi-norm $|u|_{W^{s,q}}$ computed at $s = 1$, when the strong $L^q$ is replaced by weak $L^q$.
     In particular, we address generalization of the results in \cite{hhh3,hhh2} for a general domain and non-smooth functions.
\end{abstract}

\section{Introduction}
The following two remarkable theorems were proved by H.~Brezis,
J.~Van Schaftingen and Po-Lam~Yung in \cite{hhh3,hhh2}:
\begin{theorem}\label{hjkgjkfhjffgggvggoopikhhhkjhhkjhjgjhhj}
Let $q\geq 1$. Then, for every dimension $N\geq 1$ there exist
constants $c_{N},C_{N}>0$  such that, for every $u\in
C_c^{\infty}(\R^N)$ we have
\begin{multline}\label{GMT'3jGHKKkkhjjhgzzZZzzZZzzbvq88nkhhjggjgjkpkjljluytytuguutloklljjgjgjhjklljjjkjkhjkkhkhhkhhjlkkhkjljljkjlkkhkllhjhjhhfyfppiooiououiuiuiuhjhjkhkhkjkhhkkhjjyhjggjuyyj}
c^q_{N}\,\int_{\R^N}\big|\nabla u(x)\big|^qdx\leq\\
\sup\limits_{s\in(0,+\infty)}\Bigg\{s\,\mathcal{L}^{2N}\Bigg(\bigg\{(x,y)\in
\R^N\times \R^N\;:\;\frac{
\big|u( y)-u(x)\big|^q}{|y-x|^{q+N}}> s\bigg\}\Bigg)\Bigg\}\\ \leq
C_{N}\,\int_{\R^N}\big|\nabla u(x)\big|^qdx \,.
\end{multline}
\end{theorem}
\begin{theorem}\label{hjkgjkfhjffgggvggoopikhhhkjhhkjhjgjhhjggg}
Let $q\geq 1$. Then, for every $u\in C_c^{\infty}(\R^N)$ we have
\begin{multline}\label{GMT'3jGHKKkkhjjhgzzZZzzZZzzbvq88nkhhjggjgjkpkjljluytytuguutloklljjgjgjhjklljjjkjkhjkkhkhhkhhjlkkhkjljljkjlkkhkllhjhjhhfyfppiooiououiuiuiuhjhjkhkhkjkhhkkhjjyhjggjuyyjghfh}
\lim\limits_{s\to+\infty}\Bigg\{s\,\mathcal{L}^{2N}\Bigg(\bigg\{(x,y)\in
\R^N\times \R^N\;:\;\frac{
\big|u( y)-u(x)\big|^q}{|y-x|^{q+N}}>
s\bigg\}\Bigg)\Bigg\}\\=\frac{\int_{S^{N-1}}|z_1|^qd\mathcal{H}^{N-1}(z)}{N}\int_{\R^N}\big|\nabla
u(x)\big|^qdx \,.
\end{multline}
\end{theorem}
These results shed light on what happens when one replaces the
strong $L^q$ by weak $L^q$ in the expression for the Gagliardo
semi-norm $|u|_{W^{s,q}}$, computed at $s=1$.  Several interesting
open problems, related to Theorems
\ref{hjkgjkfhjffgggvggoopikhhhkjhhkjhjgjhhj} and
\ref{hjkgjkfhjffgggvggoopikhhhkjhhkjhjgjhhjggg}, were raised  in
\cite{hhh2}:
\begin{itemize}
\item[{\bf(i)}] If  $u\in L^q$ for some $q\ge1$ satisfies
\begin{equation*}
\sup\limits_{s\in(0,+\infty)}\Bigg\{s\,\mathcal{L}^{2N}\Bigg(\bigg\{(x,y)\in
\R^N\times \R^N\;:\;\frac{
\big|u( y)-u(x)\big|^q}{|y-x|^{q+N}}>
s\bigg\}\Bigg)\Bigg\}<+\infty\,,
\end{equation*}
does it imply that $u\in W^{1,q}$ (when $q>1$) or $u\in BV$ (when
$q=1$)?
\item[{\bf (ii)}] Does Theorem
\ref{hjkgjkfhjffgggvggoopikhhhkjhhkjhjgjhhj} hold in the cases $u\in
W^{1,q}$ (when $q>1$) and $u\in BV$ (when $q=1$)?

\item[{\bf (iii)}] Same question as in (ii), but for Theorem
\ref{hjkgjkfhjffgggvggoopikhhhkjhhkjhjgjhhjggg}.

\item[{\bf (iv)}] Given $u\in L^q$, does
\begin{equation*}
\lim\limits_{s\to+\infty}\Bigg\{s\,\mathcal{L}^{2N}\Bigg(\bigg\{(x,y)\in
\R^N\times \R^N\;:\;\frac{
\big|u( y)-u(x)\big|^q}{|y-x|^{q+N}}> s\bigg\}\Bigg)\Bigg\}=0\,,
\end{equation*}
imply that $u$ necessarily equals a constant a.e.\,in $\R^N$?

\item[\bf {(v)}] For $r\in(0,q)$  characterize the class of functions $u\in L^q$,
satisfying
\begin{equation}\label{hjgjgjggjg}
\sup\limits_{s\in(0,\infty)}\Bigg\{s\,\mathcal{L}^{2N}\Bigg(\bigg\{(x,y)\in
\R^N\times \R^N\;:\;\frac{
\big|u( y)-u(x)\big|^q}{|y-x|^{r+N}}>
s\bigg\}\Bigg)\Bigg\}<+\infty\,.
\end{equation}
In particular, determine how this class is related to an appropriate
Besov space.
\end{itemize}
In the current paper we give full affirmative answers to questions
{\bf (i)}, {\bf (ii)} and  {\bf (iv)}, see Theorem
\ref{hjkgjkfhjffgggvggoopikhhhkjh} and Corollary \ref{gbjgjgjgj}
bellow. Moreover, we give a partial answer to questions {\bf (iii)},
see Corollary \ref{hjkgjkfhjffgggvggoopikhhhkjhhkjhjgjhhjgggghhdf}
bellow (in particular, we completely resolve this question in the
case $q>1$). Concerning question {\bf (v)}, we give only some
partial information about the quantity appearing in
\eqref{hjgjgjggjg} that could be obtained by combining Theorem
\ref{hjkgjkfhjff}, treating the quantities
\begin{multline}\label{ghfghfhdf}
\limsup\limits_{s\to+\infty}\Bigg\{s\,\mathcal{L}^{2N}\Bigg(\bigg\{(x,y)\in
\R^N\times \R^N\;:\;\frac{
\big|u( y)-u(x)\big|^q}{|y-x|^{r+N}}> s\bigg\}\Bigg)\Bigg\}\\
\text{and}\quad\quad
\liminf\limits_{s\to+\infty}\Bigg\{s\,\mathcal{L}^{2N}\Bigg(\bigg\{(x,y)\in
\R^N\times \R^N\;:\;\frac{
\big|u( y)-u(x)\big|^q}{|y-x|^{r+N}}> s\bigg\}\Bigg)\Bigg\}\,,
\end{multline}
for general $r$, together with Proposition \ref{gjggjfhhffhfhjgghf}.

Our first main result, answering Questions {\bf (i)} and {\bf (ii)},
is:
\begin{theorem}\label{hjkgjkfhjffgggvggoopikhhhkjh}
Let $\Omega\subset\R^N$ be an open domain with Lipschitz boundary
and let $q\geq 1$. Then there exist constants $C_{\Omega}>0$ and
${\widetilde C}_{N}>0$ satisfying $C_{\Omega}=1$ if $\Omega=\R^N$,
such that for every $u\in L^q(\Omega,\R^m)$ we have:
\begin{enumerate}
    \item [(i)]
 When $q>1$,
\begin{multline}\label{GMT'3jGHKKkkhjjhgzzZZzzZZzzbvq88nkhhjggjgjkpkjljluytytuguutloklljjgjgjhjklljjjkjkhjkkhkhhkhhjlkkhkjljljkjlkkhkllhjhjhhfyfppiooiououiuiuiuhjhjkhkhkjkhhkkhj}
\frac{\int_{S^{N-1}}|z_1|^qd\mathcal{H}^{N-1}(z)}{(N+q)}\,\int_\Omega\big|\nabla
u(x)\big|^qdx \\
\leq\limsup\limits_{s\to+\infty}\Bigg\{s\,\mathcal{L}^{2N}\Bigg(\bigg\{(x,y)\in
\Omega\times \Omega\;:\;\frac{
\big|u( y)-u(x)\big|^q}{|y-x|^{q+N}}> s\bigg\}\Bigg)\Bigg\}\leq\\
\sup\limits_{s\in(0,+\infty)}\Bigg\{s\,\mathcal{L}^{2N}\Bigg(\bigg\{(x,y)\in
\Omega\times \Omega\;:\;\frac{
\big|u( y)-u(x)\big|^q}{|y-x|^{q+N}}> s\bigg\}\Bigg)\Bigg\} \leq
C^q_{\Omega}{\widetilde C}_{N}\,\int_\Omega\big|\nabla u(x)\big|^qdx
\,,
\end{multline}
with the convention that $\int_\Omega\big|\nabla
u(x)\big|^qdx=+\infty$ if $u\notin W^{1,q}(\Omega,\R^m)$.
\item[(ii)] When $q=1$,
\begin{multline}\label{GMT'3jGHKKkkhjjhgzzZZzzZZzzbvq88nkhhjggjgjkpkjljluytytuguutloklljjgjgjhjklljjjkjkhjkkhkhhkhhjlkkhkjljljkjlkkhkllhjhjhhfyfppiooiououiuiuiuhjhjkhkhkjkjjkkhjkhkhjhjhjljkjk}
\frac{\int_{S^{N-1}}|z_1|d\mathcal{H}^{N-1}(z)}{(N+1)}\,\|Du\|(\Omega)
\leq\\
\limsup\limits_{s\to+\infty}\Bigg\{s\,\mathcal{L}^{2N}\Bigg(\bigg\{(x,y)\in
\Omega\times \Omega\;:\;\frac{
\big|u( y)-u(x)\big|}{|y-x|^{1+N}}> s\bigg\}\Bigg)\Bigg\}\leq\\
\sup\limits_{s\in(0,+\infty)}\Bigg\{s\,\mathcal{L}^{2N}\Bigg(\bigg\{(x,y)\in
\Omega\times \Omega\;:\;\frac{
\big|u( y)-u(x)\big|}{|y-x|^{1+N}}> s\bigg\}\Bigg)\Bigg\} \leq
C_{\Omega}{\widetilde C}_{N}\,\|Du\|(\Omega)\,,
\end{multline}
with the convention $\|Du\|(\Omega)=+\infty$ if $u\notin
BV(\Omega,\R^m)$.
\end{enumerate}
\end{theorem}
\begin{remark}
    Setting ${\tilde c}_N=\frac{\int_{S^{N-1}}|z_1|d\mathcal{H}^{N-1}(z)}{(N+1)\big(\mathcal{H}^{N-1}(S^{N-1})+1\big)}$ it is easy to deduce by H\"older's inequality that
$$
{\tilde c}_N^q\le
\frac{\int_{S^{N-1}}|z_1|^qd\mathcal{H}^{N-1}(z)}{(N+q)}\,.
$$
Therefore, the lower-bound in inequality
\eqref{GMT'3jGHKKkkhjjhgzzZZzzZZzzbvq88nkhhjggjgjkpkjljluytytuguutloklljjgjgjhjklljjjkjkhjkkhkhhkhhjlkkhkjljljkjlkkhkllhjhjhhfyfppiooiououiuiuiuhjhjkhkhkjkhhkkhj}
can be also written, analogously to
\eqref{GMT'3jGHKKkkhjjhgzzZZzzZZzzbvq88nkhhjggjgjkpkjljluytytuguutloklljjgjgjhjklljjjkjkhjkkhkhhkhhjlkkhkjljljkjlkkhkllhjhjhhfyfppiooiououiuiuiuhjhjkhkhkjkhhkkhjjyhjggjuyyj},
as
$$
{\tilde c}_N^q\int_\Omega\big|\nabla u(x)\big|^qdx\le
\limsup\limits_{s\to+\infty}\Bigg\{s\,\mathcal{L}^{2N}\Bigg(\bigg\{(x,y)\in
\Omega\times \Omega\;:\;\frac{
    \big|u( y)-u(x)\big|^q}{|y-x|^{q+N}}> s\bigg\}\Bigg)\Bigg\}\,.
$$
\end{remark}
\begin{remark}\label{hkhjggjhgh}
Only the lower bound in Theorem \ref{hjkgjkfhjffgggvggoopikhhhkjh}
requires a non-trivial proof. Indeed, the upper bound in this
Theorem follows quite easily from Theorem
\ref{hjkgjkfhjffgggvggoopikhhhkjhhkjhjgjhhj},  by an extension of
$u\in W^{1,q}$ or $u\in BV$ from $\Omega$ to $\R^N$, followed by a
standard approximation of its gradient seminorm by smooth functions,
see the technical Lemma \ref{fhgolghoi} for details.
\end{remark}

The proof of Theorem\,\ref{hjkgjkfhjffgggvggoopikhhhkjh} is given in
Section\,\ref{sec:r=q} below. From Theorem
\ref{hjkgjkfhjffgggvggoopikhhhkjh} we deduce the next corollary,
that provides a positive answer to  Question {\bf (iv)} (see
\cite[Open Problem\,1]{hhh2}):
\begin{corollary}\label{gbjgjgjgj}
Let $\Omega\subset\R^N$ be an open domain, $q\geq 1$ and $u\in
L^q(\Omega,\R^m)$. If
\begin{equation}\label{GMT'3jGHKKkkhjjhgzzZZzzZZzzbvq88nkhhjggjgjkpkjljluytytuguutloklljjgjgjhjklljjjkjkhjkkhkhhkhhjlkkhkjljljkjlkkhkllhjhjhhfyfppiooiououiuiuiuhjhjkhkhkjkhhkkhjjkhkjkk}
\lim\limits_{s\to+\infty}\Bigg\{s\,\mathcal{L}^{2N}\Bigg(\bigg\{(x,y)\in
\Omega\times \Omega\;:\;\frac{
\big|u( y)-u(x)\big|^q}{|y-x|^{q+N}}> s\bigg\}\Bigg)\Bigg\}=0\,,
\end{equation}
then $u(x)$ necessarily equals a constant a.e. in $\Omega$.
\end{corollary}

Regarding  Question {\bf (iii)}, the following result
 provides a  positive answer to it for
$q>1$, and in the case $q=1$ under the additional assumption, $u\in
W^{1,1}\subsetneq BV$.
\begin{corollary}\label{hjkgjkfhjffgggvggoopikhhhkjhhkjhjgjhhjgggghhdf}
Let $q\geq 1$ and let $\Omega\subset\R^N$ be an open set. Then, for
every $u\in W^{1,q}(\R^N,\R^m)$ we have
\begin{multline}\label{GMT'3jGHKKkkhjjhgzzZZzzZZzzbvq88nkhhjggjgjkpkjljluytytuguutloklljjgjgjhjklljjjkjkhjkkhkhhkhhjlkkhkjljljkjlkkhkllhjhjhhfyfppiooiououiuiuiuhjhjkhkhkjkhhkkhjjyhjggjuyyjghfhjhjghhhzz}
\lim\limits_{s\to+\infty}\Bigg\{s\,\mathcal{L}^{2N}\Bigg(\bigg\{(x,y)\in
\Omega\times \Omega\;:\;\frac{
\big|u( y)-u(x)\big|^q}{|y-x|^{q+N}}>
s\bigg\}\Bigg)\Bigg\}\\=\frac{\int_{S^{N-1}}|z_1|^qd\mathcal{H}^{N-1}(z)}{N}\int_{\Omega}\big|\nabla
u(x)\big|^qdx \,.
\end{multline}
\end{corollary}

Actually,
Corollary\,\ref{hjkgjkfhjffgggvggoopikhhhkjhhkjhjgjhhjgggghhdf} is
just a special case of the following more general result in which we
replace the quantity appearing on the L.H.S.\,of
\eqref{GMT'3jGHKKkkhjjhgzzZZzzZZzzbvq88nkhhjggjgjkpkjljluytytuguutloklljjgjgjhjklljjjkjkhjkkhkhhkhhjlkkhkjljljkjlkkhkllhjhjhhfyfppiooiououiuiuiuhjhjkhkhkjkhhkkhjjyhjggjuyyjghfhjhjghhhzz}
by a more general one; the case appearing in
Corollary\,\ref{hjkgjkfhjffgggvggoopikhhhkjhhkjhjgjhhjgggghhdf}
 corresponds
to the special case $F(a,y,x):=|a|^q$:
\begin{theorem}\label{hjkgjkfhjffgggvggoopikhhhkjhhkjhjgjhhjgggghhdf11}
Let $q\geq 1$ and let $\Omega\subset\R^N$ be an open set. Let
$F:\R\times\R^N\times\R^N\to[0,+\infty)$ be a continuous function,
such that there exists $C>0$ satisfying $0\leq F(a,y,x)\leq C|a|^q$
for every $a\in\R$ and every $x,y\in\R^N$. Moreover, assume that
$F(a,y,x)$ is  non-decreasing in the $a$--variable on $[0,+\infty)$
for every fixed $(y,x)\in\R^N\times\R^N$. Then, for every $u\in
W^{1,q}(\R^N,\R^m)$ we have
\begin{multline}\label{GMT'3jGHKKkkhjjhgzzZZzzZZzzbvq88nkhhjggjgjkpkjljluytytuguutloklljjgjgjhjklljjjkjkhjkkhkhhkhhjlkkhkjljljkjlkkhkllhjhjhhfyfppiooiououiuiuiuhjhjkhkhkjkhhkkhjjyhjggjuyyjghfhjhjghhhzz11}
\lim\limits_{s\to +\infty}s\mathcal{L}^{2N}\Bigg(\Bigg\{(x,y)\in
\Omega\times\Omega\;:\;
F\bigg(\frac{\big|u(y)-u(x)\big|}{|y-x|},y,x\bigg)\,\frac{1}{|y-x|^{N}}>
s\Bigg\}\Bigg)\\
=\frac{1}{N}\int\limits_{\Omega}\Bigg(\int\limits_{S^{N-1}}F\bigg(\big|\nabla
u(x)\big||z_1|,x,x\bigg)d\mathcal{H}^{N-1}(z)\Bigg)dx\,.
\end{multline}
\end{theorem}
The proof of
Theorem\,\ref{hjkgjkfhjffgggvggoopikhhhkjhhkjhjgjhhjgggghhdf11} is
given in Section\,\ref{limtem} below.

In the proof of the lower bound in Theorem
\ref{hjkgjkfhjffgggvggoopikhhhkjh} we essentially use the so called
``BBM formula'' due to J.~Bourgain, H.~Brezis,
P.~Mironescu~\cite{hhh1}  for $q>1$ (and under some limitations for
$q=1$). For $q=1$ the formula in the general case of BV functions
is due to J. D\'{a}vila~\cite{hhh5}. This formula states, in
particular, that given an open domain with Lipschitz boundary
$\Omega\subset\R^N$, a family of radial mollifiers
$\rho_\e\big(|z|\big):\R^N\to[0,+\infty)$, satisfying
$\int_{\R^N}\rho_\e\big(|z|\big)dz=1$ and such that for every $r>0$
there exists $\delta:=\delta_r>0$, satisfying
$\supp{(\rho_\e)}\subset B_r(0)$ for every $\e\in(0,\delta_r)$, the
following holds true:
\begin{enumerate}
    \item[(i)]
For any $q>1$ and any  $u\in L^q(\Omega,\R^m)$ we have
\begin{equation}
\label{eq:1jjjint} \lim_{\e\to 0^+} \int_{\Omega}\int_{\Omega}
\frac{|u(x)-u(y)|^q}{|x-y|^q}\,\rho_\e\big(|x-y|\big)\,dx\,dy=K_{q,N}\int_{\Omega}\big|\nabla
u(x)\big|^qdx\,,
\end{equation}
with the convention that $\int_{\Omega}\big|\nabla
u(x)\big|^qdx=+\infty$ if $u\notin W^{1,q}(\Omega,\R^m)$ and with
$K_{q,N}$ given by
\begin{equation}\label{fgyufghfghjgghgjkhkkGHGHKKggkhhjoozzbvqkkint}
K_{q,N}:=\frac{1}{\mathcal{H}^{N-1}(S^{N-1})}\int_{S^{N-1}}|z_1|^qd\mathcal{H}^{N-1}(z)\quad\quad\forall
q\geq 1\,.
\end{equation}
\item[(ii)]  In the case $q=1$, for any $u\in L^1(\Omega,\R^m)$
we have
\begin{equation}
\label{eq:1jjj1int} \lim_{\e\to 0^+} \int_{\Omega}\int_{\Omega}
\frac{|u(x)-u(y)|}{|x-y|}\,\rho_\e\big(|x-y|\big)\,dx\,dy=K_{1,N}\,\|Du\|(\Omega)\,,
\end{equation}
with the convention that  $\|Du\|(\Omega)=+\infty$ if $u\notin
BV(\Omega,\R^m)$.
\end{enumerate}In particular, taking
$$\rho_\e\big(|z|\big):=\frac{1}{2\sigma_\e\mathcal{H}^{N-1}(S^{N-1})|z|^{N-1}}\,\chi_{[\e-\sigma_\e,\e+\sigma_\e]}(|z|)\quad\quad\forall z\in\R^N$$
with sufficiently small $0<\sigma_\e\ll\e$, we deduce the following
variant of the ``BBM formula'':
\begin{multline}\label{GMT'3jGHKKkkhjjhgzzZZzzZZzzbvq88nkhhjggjgjkpkjljluytytl;klljkljojkojjo;k;kklklklkiljluikljjhkjhjhjhjh}
\frac{1}{\mathcal{H}^{N-1}(S^{N-1})}\,\lim\limits_{\e\to
0^+}\Bigg(\int\limits_{S^{N-1}}\int\limits_{\Omega}\chi_{\Omega}(x+\e\vec
n)\frac{\big|u( x+\e\vec
n)-u(x)\big|^q}{\e^q}\,dx\,d\mathcal{H}^{N-1}(\vec n)\Bigg)
\\=
K_{q,N}\int_{\Omega}\big|\nabla
u(x)\big|^qdx\quad\quad\quad\quad\text{for}\quad q>1\,,
\end{multline}
and
\begin{multline}\label{GMT'3jGHKKkkhjjhgzzZZzzZZzzbvq88nkhhjggjgjkpkjljluytytl;klljkljojkojjo;k;kklklklkiljluikljjhkjhjhjhjhjhgjghhf}
\frac{1}{\mathcal{H}^{N-1}(S^{N-1})}\,\lim\limits_{\e\to
0^+}\Bigg(\int\limits_{S^{N-1}}\int\limits_{\Omega}\chi_{\Omega}(x+\e\vec
n)\frac{\big|u( x+\e\vec
n)-u(x)\big|}{\e}\,dx\,d\mathcal{H}^{N-1}(\vec n)\Bigg) \\=
K_{1,N}\,\|Du\|(\Omega)\quad\quad\quad\quad\text{for}\quad q=1\,,
\end{multline}
where we denote
\begin{equation}\label{higuffykljk}
\chi_{\Omega}(z):=\begin{cases} 1\quad\quad z\in\Omega\,,\\
0\quad\quad z\in\R^N\setminus\Omega\,.
\end{cases}
\end{equation}
In the spirit of
\er{GMT'3jGHKKkkhjjhgzzZZzzZZzzbvq88nkhhjggjgjkpkjljluytytl;klljkljojkojjo;k;kklklklkiljluikljjhkjhjhjhjh}
and
\er{GMT'3jGHKKkkhjjhgzzZZzzZZzzbvq88nkhhjggjgjkpkjljluytytl;klljkljojkojjo;k;kklklklkiljluikljjhkjhjhjhjhjhgjghhf}
we prove the following Theorem. The special case  $r=q$ provides the
{\em key ingredient} in the proof of Theorem
\ref{hjkgjkfhjffgggvggoopikhhhkjh}:
\begin{theorem}\label{hjkgjkfhjff} Let
$\Omega\subset\R^N$ be a bounded domain, $q\geq 1$, $r\geq 0$ and
$u\in L^\infty(\Omega,\R^m)
$. Then,
\begin{multline}\label{GMT'3jGHKKkkhjjhgzzZZzzZZzzbvq88nkhhjggjgjkpkjljluytytuguutloklljjgjgjhjklljjjkjkhjkkhkhhkhhjlkkhkjljljkjlkkhkllhjhj1}
\liminf\limits_{\e\to
0^+}\Bigg(\int\limits_{S^{N-1}}\int\limits_{\Omega}\chi_{\Omega}(x+\e\vec
n)\frac{\big|u( x+\e\vec
n)-u(x)\big|^q}{\e^r}\,dx\,d\mathcal{H}^{N-1}(\vec n)\Bigg)\\
\leq
(N+r)\limsup\limits_{s\to+\infty}\Bigg\{s\,\mathcal{L}^{2N}\Bigg(\bigg\{(x,y)\in
\Omega\times\Omega\;:\;\frac{
\big|u( y)-u(x)\big|^q}{|y-x|^{r+N}}> s\bigg\}\Bigg)\Bigg\},
\end{multline}
and
\begin{multline}\label{GMT'3jGHKKkkhjjhgzzZZzzZZzzbvq88nkhhjggjgjkpkjljluytytuguutloklljjgjgjhjklljjjkjkhhiigjhjgjjljklkhhkghfhk;llkljklljjhjbjjbljjhuyuhyulkjjkhhggh1}
\limsup\limits_{\e\to
0^+}\Bigg(\int\limits_{S^{N-1}}\int\limits_{\Omega}\chi_{\Omega}(x+\e\vec
n)\frac{\big|u( x+\e\vec
n)-u(x)\big|^q}{\e^r}\,dx\,d\mathcal{H}^{N-1}(\vec n)\Bigg)\\
\geq N\,
\liminf\limits_{s\to+\infty}\Bigg\{s\,\mathcal{L}^{2N}\Bigg(\bigg\{(x,y)\in
\Omega\times\Omega\;:\;\frac{
\big|u( y)-u(x)\big|^q}{|y-x|^{r+N}}> s\bigg\}\Bigg)\Bigg\}.
\end{multline}
\end{theorem}
We refer the reader to  Lemma \ref{gjyfyfuyyfifgyify} in Appendix
for the significance  of the quantity
$$\lim\limits_{\e\to
0^+}\Bigg(\int\limits_{S^{N-1}}\int\limits_{\Omega}\chi_{\Omega}(x+\e\vec
n)\frac{\big|u( x+\e\vec
n)-u(x)\big|^q}{\e^r}\,dx\,d\mathcal{H}^{N-1}(\vec n)\Bigg)\,,$$
appearing in Theorem \ref{hjkgjkfhjff} for general $r$.
\begin{remark}\label{hjhjghgh}
Although we stated Theorem \ref{hjkgjkfhjff} for every $r\geq 0$, it
is useful only for $r\in(0,q]$, since in the case $r>q$ we have
\begin{multline}\label{GMT'3jGHKKkkhjjhgzzZZzzZZzzbvq88nkhhjggjgjkpkjljluytytuguutloklljjgjgjhjklljjjkjkhjkkhkhhkhhjlkkhkjljljkjlkkhkllhjhjkyjh}
\liminf\limits_{\e\to
0^+}\Bigg(\int\limits_{S^{N-1}}\int\limits_{\Omega}\chi_{\Omega}(x+\e\vec
n)\frac{\big|u( x+\e\vec
n)-u(x)\big|^q}{\e^r}\,dx\,d\mathcal{H}^{N-1}(\vec
n)\Bigg)<+\infty\quad\quad\text{impplies}\\  \liminf\limits_{\e\to
0^+}\Bigg(\int\limits_{S^{N-1}}\int\limits_{\Omega}\chi_{\Omega}(x+\e\vec
n)\frac{\big|u( x+\e\vec
n)-u(x)\big|^q}{\e^q}\,dx\,d\mathcal{H}^{N-1}(\vec n)\Bigg)=0\,,
\end{multline}
and thus, by the ``BBM formula", $u$ must be a constant. On the
other hand, in the case $r=0$ we obviously have
\begin{equation}\label{GMT'3jGHKKkkhjjhgzzZZzzZZzzbvq88nkhhjggjgjkpkjljluytytuguutloklljjgjgjhjklljjjkjkhjkkhkhhkhhjlkkhkjljljkjlkkhkllhjhjkyjhkuhh}
\lim\limits_{\e\to
0^+}\Bigg(\int\limits_{S^{N-1}}\int\limits_{\Omega}\chi_{\Omega}(x+\e\vec
n)\big|u( x+\e\vec n)-u(x)\big|^q\,dx\,d\mathcal{H}^{N-1}(\vec
n)\Bigg)=0\,.
\end{equation}
\end{remark}
Next we recall the definition of the Besov Spaces $B_{q,\infty}^s$
with $s\in(0,1)$:
\begin{definition}\label{gjghghghjgghGHKKhjhjhjhhzzbvqkkl,.,.}
Given $q\geq 1$ and $s\in(0,1)$, we say that $u\in
L^q(\mathbb{R}^N,\R^m)$ belongs to the Besov space
$B_{q,\infty}^s(\mathbb{R}^N,\R^m)$ if
\begin{equation}\label{gjhgjhfghffh}
\sup\limits_{\rho\in(0,\infty)}\Bigg(\sup_{|h|\leq\rho}\int_{\mathbb{R}^N}
\frac{|u(x+h)-u(x)\big|^q}{\rho^{sq}}dx\Bigg)<+\infty.
\end{equation}
Moreover, for every open $\Omega\subset\R^N$ we say that $u\in
L^q_{loc}(\Omega,\R^m)$ belongs to Besov space
$\big(B_{q,\infty}^s\big)_{loc}(\Omega,\R^m)$ if for every compact
$K\subset\subset\Omega$ there exists $u_K\in
B_{q,\infty}^s(\mathbb{R}^N,\R^d)$ such that $u_K(x)= u(x)$ for
every $x\in K$.
\end{definition}
The following technical proposition makes the connection between
Besov spaces and the quantities appearing in the statement of
Theorem\,\ref{hjkgjkfhjff}. This proposition is a direct consequence
of Corollary \ref{gjggjfhhffhfh} and Lemma \ref{hjgjg}, see in the
Appendix, whose proofs are based on similar arguments to those used
in  \cite{jmp}.
\begin{proposition}\label{gjggjfhhffhfhjgghf}
If $q\geq 1$, $r\in(0,q)$ and $u\in L^q(\R^N,\R^m)$, then,
$u\in\big(B_{q,\infty}^{r/q}\big)(\R^N,\R^m)$ if and only if we have
\begin{equation}\label{gghgjhfgggjfgfhughGHGHKKzzjkjkyuyuybvqjhgfhfhgjgjjlhkhkh}
\limsup\limits_{\e\to
0^+}\Bigg\{\int_{S^{N-1}}\int_{\R^N}\frac{\big|u( x+\e\vec
n)-u(x)\big|^q}{\e^r}\,dx\,d\mathcal{H}^{N-1}(\vec
n)\Bigg\}<+\infty\,.
\end{equation}
Moreover, if $\Omega\subset\R^N$ be an open set, $q\geq 1$,
$r\in(0,q)$ and $u\in L^q_{loc}(\Omega,\R^m)$, then,
$u\in\big(B_{q,\infty}^{r/q}\big)_{loc}(\Omega,\R^m)$ if and only if
for every open set $G\subset\subset\Omega$ we have
\begin{equation}\label{gghgjhfgggjfgfhughGHGHKKzzjkjkyuyuybvqjhgfhfhgjgjjlhk}
\limsup\limits_{\e\to 0^+}\Bigg\{\int_{S^{N-1}}\int_{G}\chi_G(
x+\e\vec n)\frac{\big|u( x+\e\vec
n)-u(x)\big|^q}{\e^r}\,dx\,d\mathcal{H}^{N-1}(\vec
n)\Bigg\}<+\infty\,.
\end{equation}
\end{proposition}
Combining Theorem \ref{hjkgjkfhjff} in the case $r\in(0,q)$ with
Proposition\,\ref{gjggjfhhffhfhjgghf} might be  a first step towards
an  answer to open question {\bf (v)}.

Our last result links the quantities in \er{ghfghfhdf} for $r=1,q>1$
and $u\in BV\cap L^\infty$ with the ``jump in the power $q$'' of
$u$:
\begin{theorem}\label{hjkgjkfhjffjhmgg7}
Let $\Omega$ be an open set with bounded Lipschitz boundary, $q>1$
and $u\in BV(\Omega,\R^m)\cap L^\infty(\Omega,\R^m)$. Then,
\begin{multline}\label{GMT'3jGHKKkkhjjhgzzZZzzZZzzbvq88nkhhjggjgjkpkjljluytytuguutloklljjgjgjhjklljjjkjkhjkkhkhhkhhjlkkhkjljljkjlkkhkllhjhj7}
N\,
\liminf\limits_{s\to+\infty}\Bigg\{s\,\mathcal{L}^{2N}\Bigg(\bigg\{(x,y)\in
\Omega\times\Omega\;:\;\frac{
\big|u( y)-u(x)\big|^q}{|y-x|^{N+1}}> s\bigg\}\Bigg)\Bigg\}\leq
\\
\bigg(\int_{S^{N-1}}|z_1|d\mathcal{H}^{N-1}(z)\bigg)\Bigg(\int_{J_u\cap
\Omega}\Big|u^+(x)-u^-(x)\Big|^qd\mathcal{H}^{N-1}(x)\Bigg)\\
\leq
(N+r)\limsup\limits_{s\to+\infty}\Bigg\{s\,\mathcal{L}^{2N}\Bigg(\bigg\{(x,y)\in
\Omega\times\Omega\;:\;\frac{
\big|u( y)-u(x)\big|^q}{|y-x|^{N+1}}> s\bigg\}\Bigg)\Bigg\}\,.
\end{multline}
Here $J_u$ denotes the jump set of $u\in BV$  and $u^+,u^-$ are the
approximate one-side limits of $u$.
\end{theorem}
To conclude, we list some interesting open problems for  future
research:
\begin{itemize}
\item[ {\bf (a)}] Does a \underline{complete} version of Theorem
\ref{hjkgjkfhjff} hold, where we replace $\liminf$ by $\limsup$ in
\er{GMT'3jGHKKkkhjjhgzzZZzzZZzzbvq88nkhhjggjgjkpkjljluytytuguutloklljjgjgjhjklljjjkjkhjkkhkhhkhhjlkkhkjljljkjlkkhkllhjhj}
and
\er{GMT'3jGHKKkkhjjhgzzZZzzZZzzbvq88nkhhjggjgjkpkjljluytytuguutloklljjgjgjhjklljjjkjkhhiigjhjgjjljklkhhkghfhk;llkljklljjhjbjjbljjhuyuhyulkjjkhhggh}.
In particular, following Proposition \ref{gjggjfhhffhfhjgghf}, this
 would provide a full answer to Question {\bf (v)} in the case $u\in
L^\infty$.

\item[ {\bf (b)}] In the spirit of Theorem
\ref{hjkgjkfhjffgggvggoopikhhhkjhhkjhjgjhhjggg}, does
\er{GMT'3jGHKKkkhjjhgzzZZzzZZzzbvq88nkhhjggjgjkpkjljluytytuguutloklljjgjgjhjklljjjkjkhjkkhkhhkhhjlkkhkjljljkjlkkhkllhjhj}
in Theorem \ref{hjkgjkfhjff} hold with the constant $N$, instead of
$(N+r)$?

\item[ {\bf (c)}]
Does Corollary \ref{hjkgjkfhjffgggvggoopikhhhkjhhkjhjgjhhjgggghhdf}
hold for $q=1$ and $u\in BV\setminus W^{1,1}$?
\end{itemize}

\subsubsection*{Acknowledgments.} I am indebted to Prof.\,Haim Brezis
for providing me the preprints \cite{hhh3,hhh2} and for suggesting
to me the research directions that served as a basis for the study
carried out in the present manuscript.

\section{Proof of Theorem \ref{hjkgjkfhjff}}
This section is devoted to the proof of Theorem \ref{hjkgjkfhjff}.
Its special case $r=q$  is essential for the proof of the main
results Theorem\,\ref{hjkgjkfhjffgggvggoopikhhhkjh} and
Corollary\,\ref{gbjgjgjgj}. Theorem \ref{hjkgjkfhjff} is a
particular case of the following more general statment.
\begin{proposition}\label{hjkgjkfhjff1} Let
$\Omega\subset\R^N$ be a bounded domain,
$r\geq 0$ and
$F\in L^\infty\Big(\Omega\times\Omega\,,\,[0,+\infty)\Big)
$. Then,
\begin{multline}\label{GMT'3jGHKKkkhjjhgzzZZzzZZzzbvq88nkhhjggjgjkpkjljluytytuguutloklljjgjgjhjklljjjkjkhjkkhkhhkhhjlkkhkjljljkjlkkhkllhjhj}
\liminf\limits_{\e\to
0^+}\Bigg(\int\limits_{S^{N-1}}\int\limits_{\Omega}\chi_{\Omega}(x+\e\vec
n)\frac{F\big(x+\e\vec n,x\big)
}{\e^r}\,dx\,d\mathcal{H}^{N-1}(\vec n)\Bigg)\\
\leq
(N+r)\limsup\limits_{s\to+\infty}\Bigg\{s\,\mathcal{L}^{2N}\Bigg(\bigg\{(x,y)\in
\Omega\times\Omega\;:\;\frac{
F(y,x)}{|y-x|^{r+N}}> s\bigg\}\Bigg)\Bigg\},
\end{multline}
and
\begin{multline}\label{GMT'3jGHKKkkhjjhgzzZZzzZZzzbvq88nkhhjggjgjkpkjljluytytuguutloklljjgjgjhjklljjjkjkhhiigjhjgjjljklkhhkghfhk;llkljklljjhjbjjbljjhuyuhyulkjjkhhggh}
\limsup\limits_{\e\to
0^+}\Bigg(\int\limits_{S^{N-1}}\int\limits_{\Omega}\chi_{\Omega}(x+\e\vec
n)\frac{F\big(x+\e\vec n,x\big)}{\e^r}\,dx\,d\mathcal{H}^{N-1}(\vec n)\Bigg)\\
\geq N\,
\liminf\limits_{s\to+\infty}\Bigg\{s\,\mathcal{L}^{2N}\Bigg(\bigg\{(x,y)\in
\Omega\times\Omega\;:\;\frac{
F(y,x)}{|y-x|^{r+N}}> s\bigg\}\Bigg)\Bigg\}.
\end{multline}
\end{proposition}
\begin{proof}
Given $\alpha>0$, consider
\begin{equation}\label{GMT'3jGHKKkkhjjhgzzZZzzZZzzbvq88jkhkhhk88jhkhhjhhhjhiyhijkkjkkhkhhhuhhjhjjhiihhjlhkhkjohkhk1}
\eta_\e\big(t\big):=\frac{1}{\alpha\big|\ln{\e}\big||t|^N}\,\chi_{[\e^{1+\alpha},\e]}(t)\quad\quad\forall
t\in \mathbb{R}.
\end{equation}
In particular, we have
\begin{equation}\label{GMT'3jGHKKkkhjjhgzzZZzzZZzzbvq88jkhkhhk88jhkhhjhhhjhiyhijkkjkkhkhhhuhhjhjjhiihh1}
\eta_\e\big(t\big)t^N=\frac{1}{\alpha\big|\ln{\e}\big|}\chi_{[\e^{1+\alpha},\e]}(t).
\end{equation}
Then, for every $z\in\R^N$ every $h\geq 0$ and every $s\geq 0$
considering
\begin{multline}\label{jghjgghghhg}
K_{\e,u}\big(z,s,h\big):=\mathcal{L}^N\Bigg(\bigg\{x\in
\Omega\;:\;\frac{\chi_{\Omega}(x+z)F\big(
x+z,x\big)}{|z|^h}>s\bigg\}\Bigg)\\=K_{\e,u}\bigg(z,\frac{s}{|z|^l},h+l\bigg)\quad\quad\forall
l\geq 0\,,
\end{multline}
by Fubini Theorem we deduce:
\begin{multline}\label{GMT'3jGHKKkkhjjhgzzZZzzZZzzbvq88nkhhjggjgjkpkjljluytyt1}
\int\limits_{\Omega}\int\limits_{\Omega}\eta_\e\Big(|y-x|\Big)\frac{F(y,x)}{|y-x|^r}dydx=\int\limits_{\mathbb{R}^N}\int\limits_{\Omega}\eta_\e\Big(|z|\Big)
\chi_{\Omega}(x+z)\frac{F\big( x+z,x\big)}{|z|^r}dxdz\\=
\int\limits_{\mathbb{R}^+}\int\limits_{\mathbb{R}^N}\eta_\e\Big(|z|\Big)K_{\e,u}\big(z,s,r\big)
dzds=
\int\limits_{S^{N-1}}\int\limits_{\mathbb{R}^+}\int\limits_{\mathbb{R}^+}\eta_\e\big(t\big)t^{N-1}K_{\e,u}\big(t\vec
n,s,r\big) dtdsd\mathcal{H}^{N-1}(\vec n)
\\=
\int\limits_{S^{N-1}}\int\limits_{\mathbb{R}^+}\int\limits_{\mathbb{R}^+}\eta_\e\big(t\big)t^{N-1}K_{\e,u}\bigg(t\vec
n,\frac{s}{t^N},r+N\bigg) dtdsd\mathcal{H}^{N-1}(\vec n)
 \\=
\int\limits_{S^{N-1}}\int\limits_{\mathbb{R}^+}\int\limits_{\mathbb{R}^+}\eta_\e\big(t\big)t^N\,t^{N
-1}K_{\e,u}\big(t\vec n,s,r+N\big) dtdsd\mathcal{H}^{N-1}(\vec n).
\end{multline}
%
%
%
%
%
%
Thus, by
\er{GMT'3jGHKKkkhjjhgzzZZzzZZzzbvq88nkhhjggjgjkpkjljluytyt1} and
\er{GMT'3jGHKKkkhjjhgzzZZzzZZzzbvq88jkhkhhk88jhkhhjhhhjhiyhijkkjkkhkhhhuhhjhjjhiihh1}
we have
\begin{equation}\label{GMT'3jGHKKkkhjjhgzzZZzzZZzzbvq88nkhhjggjgjkpkjljluytytuguutloklljjgjgjhj1}
\int\limits_{\Omega}\int\limits_{\Omega}\eta_\e\Big(|y-x|\Big)\frac{F(y,x)}{|y-x|^r}dydx=I_{\e,u,\alpha}\Big([0,+\infty]\Big).
\end{equation}
where, for every $0\leq a\leq b\leq +\infty$ we denote:
\begin{multline}\label{GMT'3jGHKKkkhjjhgzzZZzzZZzzbvq88nkhhjggjgjkpkjljluytytuguutloklljjgjgjhjklljjjkjkh1jkhjhjj}
I_{\e,u,\alpha}\Big([a,b]\Big):=
\int\limits_{\mathbb{R}}\int\limits_{\mathbb{R}}\int\limits_{S^{N-1}}\frac{1}{\alpha\big|\ln{\e}\big|s}\,\chi_{[\e^{1+\alpha},\e]}(t)\,
\chi_{[a,b]}(s)t^{N -1}sK_{\e,u}\big(t\vec
n,s,r+N\big)d\mathcal{H}^{N-1}(\vec n) dtds
\\
=\int\limits_{\mathbb{R}}\int\limits_{\mathbb{R}}\int\limits_{S^{N-1}}\frac{1}{\alpha\big|\ln{\e}\big|s}\,\chi_{[\e^{1+\alpha},\e]}(t)\,
\chi_{[a,b]}(s)\,
\times\\ \times\,t^{N -1}s\mathcal{L}^N\Bigg(\bigg\{x\in
\Omega\;:\;\frac{\chi_{\Omega}(x+t\vec n)F\big(x+\e\vec
n,x\big)}{|t|^{r+N}}> s\bigg\}\Bigg)d\mathcal{H}^{N-1}(\vec n) dtds.
\end{multline}
So, for every $d>0$ and $\gamma>0$ we have
\begin{multline}\label{GMT'3jGHKKkkhjjhgzzZZzzZZzzbvq88nkhhjggjgjkpkjljluytytuguutloklljjgjgjhjklljjjkjkh1}
\int\limits_{\Omega}\int\limits_{\Omega}\eta_\e\Big(|y-x|\Big)\frac{F(y,x)}{|y-x|^r}dydx\\=I_{\e,u,\alpha}\Bigg(\bigg[\frac{1}{\e^d},\frac{1}{\e^{d+\gamma}}\bigg]\Bigg)+I_{\e,u,\alpha}\Bigg(\bigg[0,\frac{1}{\e^d}\bigg]\Bigg)+I_{\e,u,\alpha}\Bigg(\bigg[\frac{1}{\e^{d+\gamma}},+\infty\bigg]\Bigg).
\end{multline}
Furthermore, since $\Omega$ is bonded, by
\er{GMT'3jGHKKkkhjjhgzzZZzzZZzzbvq88nkhhjggjgjkpkjljluytytuguutloklljjgjgjhjklljjjkjkh1jkhjhjj}
we can obtain
\begin{equation}\label{GMT'3jGHKKkkhjjhgzzZZzzZZzzbvq88nkhhjggjgjkpkjljluytytuguutloklljjgjgjhjklljjjkjkhjhkhgjljjhkhkhjggghgnkhk}
I_{\e,u,\alpha}\Bigg(\bigg[0,\frac{1}{\e^d}\bigg]\Bigg)\leq
\frac{C}{\alpha\big|\ln{\e}\big|}\,\e^{N-d},
\end{equation}
and
\begin{multline}\label{GMT'3jGHKKkkhjjhgzzZZzzZZzzbvq88nkhhjggjgjkpkjljluytytuguutloklljjgjgjhjklljjjkjkhkhhgjnkhhkjohkjioojjojjiojjojl}
I_{\e,u,\alpha}\Bigg(\bigg[\frac{1}{\e^{d+\gamma}},+\infty\bigg]\Bigg)
=\int\limits_{\mathbb{R}}\int\limits_{\mathbb{R}}\int\limits_{S^{N-1}}\frac{1}{\alpha\big|\ln{\e}\big|}\,\frac{1}{
t^{r+1}}\,\chi_{[\e^{1+\alpha},\e]}(t)\,
\chi_{[1/\e^{d+\gamma},+\infty)}(s)\,
\times\\ \times\,t^{r+N}K_{\e,u}\big(t\vec n,s
t^{r+N},0\big)d\mathcal{H}^{N-1}(\vec n) dtds\\
=\int\limits_{\mathbb{R}}\int\limits_{\mathbb{R}}\int\limits_{S^{N-1}}\frac{1}{\alpha\big|\ln{\e}\big|}\,\frac{1}{
t^{r+1}}\,\chi_{[\e^{1+\alpha},\e]}(t)\,
\chi_{[t^{r+N}/\e^{d+\gamma},+\infty)}(\tau)K_{\e,u}\big(t\vec
n,\tau,0\big)d\mathcal{H}^{N-1}(\vec n) dtd\tau\\
\leq
\int\limits_{\frac{1}{\e^{(d+\gamma)-(1+\alpha)(r+N)}}}^{+\infty}\int\limits_{\mathbb{R}}\int\limits_{S^{N-1}}\frac{1}{\alpha\big|\ln{\e}\big|}\,\frac{1}{
t^{r+1}}\,\chi_{[\e^{1+\alpha},\e]}(t)\, K_{\e,u}\big(t\vec
n,\tau,0\big)d\mathcal{H}^{N-1}(\vec n) dtd\tau\leq
\\
\int\limits_{\frac{1}{\e^{(d+\gamma)-(1+\alpha)(r+N)}}}^{+\infty}\int\limits_{\mathbb{R}}\int\limits_{S^{N-1}}\frac{1}{\alpha\big|\ln{\e}\big|}\,\frac{1}{
t^{r+1}}\,\chi_{[\e^{1+\alpha},\e]}(t)K_{\e,u}\bigg(t\vec
n,\frac{1}{\e^{(d+\gamma)-(1+\alpha)(r+N)}},0\bigg)d\mathcal{H}^{N-1}(\vec
n) dtd\tau\,.
\end{multline}
Thus, since $F\in L^\infty$, in the case $d\leq N$ and
$\gamma>(1+\alpha)(r+N)-d
$ for sufficiently small
$\e>0$ we have
\begin{multline}\label{ghhvufioufufljhkh}
K_{\e,u}\bigg(t\vec
n,\frac{1}{\e^{(d+\gamma)-(1+\alpha)(r+N)}},0\bigg):=\\ \bigg\{x\in
\Omega\;:\;\chi_{\Omega}(x+t\vec n)F\big(x+\e\vec n,x\big)>
1/\e^{((d+\gamma)-(1+\alpha)(r+N))}\bigg\}=\emptyset,
\end{multline}
and thus by
\er{GMT'3jGHKKkkhjjhgzzZZzzZZzzbvq88nkhhjggjgjkpkjljluytytuguutloklljjgjgjhjklljjjkjkhjhkhgjljjhkhkhjggghgnkhk}
and
\er{GMT'3jGHKKkkhjjhgzzZZzzZZzzbvq88nkhhjggjgjkpkjljluytytuguutloklljjgjgjhjklljjjkjkhkhhgjnkhhkjohkjioojjojjiojjojl}
we rewrite
\er{GMT'3jGHKKkkhjjhgzzZZzzZZzzbvq88nkhhjggjgjkpkjljluytytuguutloklljjgjgjhjklljjjkjkh1},
in the case $d\leq N$, $\gamma>(1+\alpha)(r+N)-d
$ and
sufficiently small $\e>0$, as:
\begin{multline}\label{GMT'3jGHKKkkhjjhgzzZZzzZZzzbvq88nkhhjggjgjkpkjljluytytuguutloklljjgjgjhjklljjjkjkhjkkhkh}
\int\limits_{\Omega}\int\limits_{\Omega}\eta_\e\Big(|y-x|\Big)\frac{
F(y,x)}{|y-x|^r}dydx=O\bigg(\frac{1}{\alpha\big|\ln{\e}\big|}\bigg)+I_{\e,u,\alpha}\Bigg(\bigg[\frac{1}{\e^d},\frac{1}{\e^{d+\gamma}}\bigg]\Bigg)\\
\leq O\bigg(\frac{1}{\alpha\big|\ln{\e}\big|}\bigg)+
\int\limits_{\mathbb{R}}\int\limits_{\mathbb{R}}\int\limits_{S^{N-1}}\frac{1}{\alpha\big|\ln{\e}\big|s}\,\chi_{(0,+\infty)}(t)\,
\chi_{[1/\e^d,1/\e^{d+\gamma}]}(s)\,
\times\\ \times\,t^{N -1}s\mathcal{L}^N\Bigg(\bigg\{x\in
\Omega\;:\;\frac{\chi_{\Omega}(x+t\vec n)F\big(x+\e\vec
n,x\big)}{|t|^{r+N}}> s\bigg\}\Bigg)d\mathcal{H}^{N-1}(\vec n) dtds=
\\
O\bigg(\frac{1}{\alpha\big|\ln{\e}\big|}\bigg)+
\int\limits_{\mathbb{R}}\frac{1}{\alpha\big|\ln{\e}\big|s}\,
\chi_{[1/\e^d,1/\e^{d+\gamma}]}(s)\,
s\mathcal{L}^{2N}\Bigg(\bigg\{(x,y)\in \Omega\times\Omega\;:\;\frac{
F(y,x)}{|y-x|^{r+N}}> s\bigg\}\Bigg)ds.
\end{multline}
On the other hand, since, we have
\er{GMT'3jGHKKkkhjjhgzzZZzzZZzzbvq88jkhkhhk88jhkhhjhhhjhiyhijkkjkkhkhhhuhhjhjjhiihh1},
then we deduce:
\begin{multline}\label{GMT'3jGHKKkkhjjhgzzZZzzZZzzbvq88nkhhjggjgjkpkjljluytytl;klljkljojkojjo1}
\int\limits_{\Omega}\int\limits_{\Omega}\eta_\e\Big(|y-x|\Big)\frac{F(y,x)}{|y-x|^r}dydx=\int\limits_{\mathbb{R}^N}\int\limits_{\Omega}\eta_\e\Big(|z|\Big)\chi_{\Omega}(x+z)\frac{F\big(
x+z,x\big)}{|z|^r}dxdz\\=
\int\limits_{S^{N-1}}\int\limits_{\mathbb{R}^+}\int\limits_{\Omega}t^{N-1}\eta_\e\big(t\big)\chi_{\Omega}(x+t\vec
n)\frac{F\big(x+\e\vec n,x\big)}{t^r}dxdtd\mathcal{H}^{N-1}(\vec n)
\\=
\int\limits_{\e^{1+\alpha}}^{\e}\frac{1}{\alpha\big|\ln{\e}\big|t}\Bigg(\int\limits_{S^{N-1}}\int\limits_{\Omega}\chi_{\Omega}(x+t\vec
n)\frac{F\big(x+\e\vec n,x\big)}{t^r}dxd\mathcal{H}^{N-1}(\vec
n)\Bigg)dt.
\end{multline}
Thus, by
\er{GMT'3jGHKKkkhjjhgzzZZzzZZzzbvq88nkhhjggjgjkpkjljluytytl;klljkljojkojjo1}
and
\er{GMT'3jGHKKkkhjjhgzzZZzzZZzzbvq88nkhhjggjgjkpkjljluytytuguutloklljjgjgjhjklljjjkjkhjkkhkh}
for $d=N$, $\gamma>r+\alpha (N+r)$ and
sufficiently small $\e>0$ we have
\begin{multline}\label{GMT'3jGHKKkkhjjhgzzZZzzZZzzbvq88nkhhjggjgjkpkjljluytytuguutloklljjgjgjhjklljjjkjkhjkkhkhhkhhjlkkhk}
\int\limits_{\e^{1+\alpha}}^{\e}\frac{1}{\alpha\big|\ln{\e}\big|t}\Bigg(\int\limits_{S^{N-1}}\int\limits_{\Omega}\chi_{\Omega}(x+t\vec
n)\frac{F\big(x+\e\vec n,x\big)}{t^r}dxd\mathcal{H}^{N-1}(\vec
n)\Bigg)dt\\=\int\limits_{\Omega}\int\limits_{\Omega}\eta_\e\Big(|y-x|\Big)\frac{F(y,x)}{|y-x|^r}dydx
\leq
O\bigg(\frac{1}{\alpha\big|\ln{\e}\big|}\bigg)+\\
\int\limits_{\mathbb{R}}\,\frac{1}{\alpha\big|\ln{\e}\big|s}\,
\chi_{[1/\e^N,1/\e^{N+\gamma}]}(s)\,
s\mathcal{L}^{2N}\Bigg(\bigg\{(x,y)\in \Omega\times\Omega\;:\;\frac{
F(y,x)}{|y-x|^{r+N}}> s\bigg\}\Bigg)ds.
\end{multline}
On the other hand,
we have
\begin{multline}\label{GMT'3jGHKKkkhjjhgzzZZzzZZzzbvq88nkhhjggjgjkpkjljluytytl;klljkljojkojjo;k;kklklklkiljluikljhkjhjhmhhgn}
\inf\limits_{t\in(0,\e)}\Bigg(\int\limits_{S^{N-1}}\int\limits_{\Omega}\chi_{\Omega}(x+t\vec
n)\frac{F\big(x+\e\vec n,x\big)}{t^r}dxd\mathcal{H}^{N-1}(\vec
n)\Bigg)
\leq\int\limits_{\Omega}\int\limits_{\Omega}\eta_\e\Big(|y-x|\Big)\frac{F(y,x)}{|y-x|^r}dydx\\=
\int\limits_{\e^{1+\alpha}}^{\e}\frac{1}{\alpha\big|\ln{\e}\big|t}\Bigg(\int\limits_{S^{N-1}}\int\limits_{\Omega}\chi_{\Omega}(x+t\vec
n)\frac{F\big(x+\e\vec n,x\big)}{t^r}dxd\mathcal{H}^{N-1}(\vec
n)\Bigg)dt
\\
\leq\sup\limits_{t\in(0,\e)}\Bigg(\int\limits_{S^{N-1}}\int\limits_{\Omega}\chi_{\Omega}(x+t\vec
n)\frac{F\big(x+\e\vec n,x\big)}{t^r}dxd\mathcal{H}^{N-1}(\vec
n)\Bigg),
\end{multline}
and
\begin{multline}\label{GMT'3jGHKKkkhjjhgzzZZzzZZzzbvq88nkhhjggjgjkpkjljluytytuguutloklljjgjgjhjklljjjkjkhjkkhkhhkhhjlkkhkjljjljjhhk}
\inf\limits_{s>(1/\e^N)}\Bigg\{s\,\mathcal{L}^{2N}\Bigg(\bigg\{(x,y)\in
\Omega\times\Omega\;:\;\frac{
F(y,x)}{|y-x|^{r+N}}> s\bigg\}\Bigg)\Bigg\}\leq\\
\int\limits_{\mathbb{R}}\,\frac{1}{\gamma\big|\ln{\e}\big|s}\,
\chi_{[1/\e^N,1/\e^{N+\gamma}]}(s)\,
s\mathcal{L}^{2N}\Bigg(\bigg\{(x,y)\in \Omega\times\Omega\;:\;\frac{
F(y,x)}{|y-x|^{r+N}}> s\bigg\}\Bigg)ds\\
\leq
\sup\limits_{s>(1/\e^N)}\Bigg\{s\,\mathcal{L}^{2N}\Bigg(\bigg\{(x,y)\in
\Omega\times\Omega\;:\;\frac{
F(y,x)}{|y-x|^{r+N}}> s\bigg\}\Bigg)\Bigg\}.
\end{multline}
Therefore, inserting these into
\er{GMT'3jGHKKkkhjjhgzzZZzzZZzzbvq88nkhhjggjgjkpkjljluytytuguutloklljjgjgjhjklljjjkjkhjkkhkhhkhhjlkkhk}
gives that  for $d=N$, $\gamma>r+\alpha (N+r)$ and sufficiently
small $\e>0$ we have
\begin{multline}\label{GMT'3jGHKKkkhjjhgzzZZzzZZzzbvq88nkhhjggjgjkpkjljluytytuguutloklljjgjgjhjklljjjkjkhjkkhkhhkhhjlkkhkjljljkjlkkh}
\inf\limits_{t\in(0,\e)}\Bigg(\int\limits_{S^{N-1}}\int\limits_{\Omega}\chi_{\Omega}(x+t\vec
n)\frac{F\big(x+\e\vec n,x\big)}{t^r}dxd\mathcal{H}^{N-1}(\vec
n)\Bigg) \leq
O\bigg(\frac{1}{\alpha\big|\ln{\e}\big|}\bigg)+\\
\frac{\gamma}{\alpha}\,\sup\limits_{s>(1/\e^N)}\Bigg\{s\,\mathcal{L}^{2N}\Bigg(\bigg\{(x,y)\in
\Omega\times\Omega\;:\;\frac{
F(y,x)}{|y-x|^{r+N}}> s\bigg\}\Bigg)\Bigg\}.
\end{multline}
Thus, letting $\e\to 0^+$ in
\er{GMT'3jGHKKkkhjjhgzzZZzzZZzzbvq88nkhhjggjgjkpkjljluytytuguutloklljjgjgjhjklljjjkjkhjkkhkhhkhhjlkkhkjljljkjlkkh}
gives that for $d=N$ and $\gamma>r+\alpha (N+r)$ we have
\begin{multline}\label{GMT'3jGHKKkkhjjhgzzZZzzZZzzbvq88nkhhjggjgjkpkjljluytytuguutloklljjgjgjhjklljjjkjkhhiigjhjgjjljklkhhkghfhk;llkljklljjhjbjjbljjhuyuhyulkjjkrrjkhkhklkljk}
\liminf\limits_{\e\to
0^+}\Bigg(\int\limits_{S^{N-1}}\int\limits_{\Omega}\chi_{\Omega}(x+\e\vec
n)\frac{F\big(x+\e\vec n,x\big)}{\e^r}dxd\mathcal{H}^{N-1}(\vec n)\Bigg)\\
\leq \frac{\gamma}{\alpha}\,
\limsup\limits_{s\to+\infty}\Bigg\{s\,\mathcal{L}^{2N}\Bigg(\bigg\{(x,y)\in
\Omega\times\Omega\;:\;\frac{
F(y,x)}{|y-x|^{r+N}}> s\bigg\}\Bigg)\Bigg\}.
\end{multline}
Therefore, letting $\gamma\to\big(r+\alpha (N+r)\big)^+$ in
\er{GMT'3jGHKKkkhjjhgzzZZzzZZzzbvq88nkhhjggjgjkpkjljluytytuguutloklljjgjgjhjklljjjkjkhhiigjhjgjjljklkhhkghfhk;llkljklljjhjbjjbljjhuyuhyulkjjkrrjkhkhklkljk}
we deduce for $d=N$ and $\gamma=r+\alpha (N+r)$:
\begin{multline}\label{GMT'3jGHKKkkhjjhgzzZZzzZZzzbvq88nkhhjggjgjkpkjljluytytuguutloklljjgjgjhjklljjjkjkhhiigjhjgjjljklkhhkghfhk;llkljklljjhjbjjbljjhuyuhyulkjjkrr}
\liminf\limits_{\e\to
0^+}\Bigg(\int\limits_{S^{N-1}}\int\limits_{\Omega}\chi_{\Omega}(x+\e\vec
n)\frac{F\big(x+\e\vec n,x\big)}{\e^r}dxd\mathcal{H}^{N-1}(\vec n)\Bigg)\\
\leq \Bigg(\frac{r}{\alpha}+(N+r)\Bigg)\,
\limsup\limits_{s\to+\infty}\Bigg\{s\,\mathcal{L}^{2N}\Bigg(\bigg\{(x,y)\in
\Omega\times\Omega\;:\;\frac{
F(y,x)}{|y-x|^{r+N}}> s\bigg\}\Bigg)\Bigg\}.
\end{multline}
Letting $\alpha\to+\infty$ in
\er{GMT'3jGHKKkkhjjhgzzZZzzZZzzbvq88nkhhjggjgjkpkjljluytytuguutloklljjgjgjhjklljjjkjkhhiigjhjgjjljklkhhkghfhk;llkljklljjhjbjjbljjhuyuhyulkjjkrr}
we finally deduce
\er{GMT'3jGHKKkkhjjhgzzZZzzZZzzbvq88nkhhjggjgjkpkjljluytytuguutloklljjgjgjhjklljjjkjkhjkkhkhhkhhjlkkhkjljljkjlkkhkllhjhj}.
Next, by
\er{GMT'3jGHKKkkhjjhgzzZZzzZZzzbvq88nkhhjggjgjkpkjljluytytuguutloklljjgjgjhjklljjjkjkh1}
for every $d>0$ and $\gamma>0$ we have
\begin{equation}\label{GMT'3jGHKKkkhjjhgzzZZzzZZzzbvq88nkhhjggjgjkpkjljluytytuguutloklljjgjgjhjklljjjkjkhhiigj}
\int\limits_{\Omega}\int\limits_{\Omega}\eta_\e\Big(|y-x|\Big)\frac{F(y,x)}{|y-x|^r}dydx\geq
I_{\e,u,\alpha}\Bigg(\bigg[\frac{1}{\e^d},\frac{1}{\e^{d+\gamma}}\bigg]\Bigg).
\end{equation}
Furthermore, by \er{jghjgghghhg} we have
\begin{multline}\label{GMT'3jGHKKkkhjjhgzzZZzzZZzzbvq88nkhhjggjgjkpkjljluytytuguutloklljjgjgjhjklljjjkjkhkhhkjggjkjhhkhohkhjklj}
\int\limits_{\mathbb{R}}\int\limits_{\mathbb{R}}\int\limits_{S^{N-1}}\frac{1}{\alpha\big|\ln{\e}\big|s}\,\chi_{[\e,+\infty)}(t)\,
\chi_{[1/\e^d,1/\e^{d+\gamma}]}(s)t^{N -1}sK_{\e,u}\big(t\vec n,s
,r+N\big)d\mathcal{H}^{N-1}(\vec n)
dtds=\\
\int\limits_{\mathbb{R}}\int\limits_{\mathbb{R}}\int\limits_{S^{N-1}}\frac{1}{\alpha\big|\ln{\e}\big|}\,\chi_{[\e,+\infty)}(t)\,
\chi_{[1/\e^d,1/\e^{d+\gamma}]}(s)\frac{1}{t^{r+1}}\,t^{N
+r}K_{\e,u}\big(t\vec n,s t^{r+N},0\big)d\mathcal{H}^{N-1}(\vec n)
dtds
\\=
\int\limits_{\mathbb{R}}\int\limits_{\mathbb{R}}\int\limits_{S^{N-1}}\frac{1}{\alpha\big|\ln{\e}\big|}\,\chi_{[\e,+\infty)}(t)\,
\chi_{[t^{r+N}/\e^d,t^{r+N}/\e^{d+\gamma}]}(\tau)\frac{1}{t^{r+1}}K_{\e,u}\big(t\vec
n,\tau,0\big)d\mathcal{H}^{N-1}(\vec n) dtd\tau\\ \leq
\int\limits_{\mathbb{R}}\int\limits_{\mathbb{R}}\int\limits_{S^{N-1}}\frac{1}{\alpha\big|\ln{\e}\big|}\,\chi_{[\e,+\infty)}(t)\,
\chi_{[1/\e^{d-(r+N)},t^{r+N}/\e^{d+\gamma}]}(\tau)\frac{1}{t^{r+1}}K_{\e,u}\big(t\vec
n,\tau,0\big)d\mathcal{H}^{N-1}(\vec n) dtd\tau\leq\\
\int\limits_{\mathbb{R}}\int\limits_{\mathbb{R}}\int\limits_{S^{N-1}}\frac{1}{\alpha\big|\ln{\e}\big|}\,\chi_{[\e,+\infty)}(t)\,
\chi_{[1/\e^{d-(r+N)},t^{r+N}/\e^{d+\gamma}]}(\tau)\frac{1}{t^{r+1}}K_{\e,u}\bigg(t\vec
n,\frac{1}{\e^{d-(r+N)}},0\bigg)d\mathcal{H}^{N-1}(\vec n) dtd\tau.
\end{multline}
On the other hand, since $F\in L^\infty$, in the case $d>(N+r)$ for
sufficiently small $\e>0$ we have
\begin{equation}\label{ghhvufioufufljhkhbhggg}
K_{\e,u}\bigg(t\vec n,\frac{1}{\e^{d-(r+N)}},0\bigg)=\bigg\{x\in
\Omega\;:\;\chi_{\Omega}(x+t\vec n)F\big(x+\e\vec n,x\big)>
1/\e^{d-(r+N)}\bigg\}=\emptyset,
\end{equation}
and thus by
\er{GMT'3jGHKKkkhjjhgzzZZzzZZzzbvq88nkhhjggjgjkpkjljluytytuguutloklljjgjgjhjklljjjkjkhkhhkjggjkjhhkhohkhjklj},
in the case $d>(N+r)$ for sufficiently small $\e>0$ we have
\begin{multline}\label{GMT'3jGHKKkkhjjhgzzZZzzZZzzbvq88nkhhjggjgjkpkjljluytytuguutloklljjgjgjhjklljjjkjkhkhhkjggjkjhhkhohkhjkljjkhhkhjljkjj}
\int\limits_{\mathbb{R}}\int\limits_{\mathbb{R}}\int\limits_{S^{N-1}}\frac{1}{\alpha\big|\ln{\e}\big|s}\,\chi_{[\e,+\infty)}(t)\,
\chi_{[1/\e^d,1/\e^{d+\gamma}]}(s)\,
\times\\ \times\,t^{N -1}s\mathcal{L}^N\Bigg(\bigg\{x\in
\Omega\;:\;\frac{\chi_{\Omega}(x+t\vec n)F\big(x+\e\vec
n,x\big)}{|t|^{r+N}}> s\bigg\}\Bigg)d\mathcal{H}^{N-1}(\vec n)
dtds=0.
\end{multline}
In particular, by
\er{GMT'3jGHKKkkhjjhgzzZZzzZZzzbvq88nkhhjggjgjkpkjljluytytuguutloklljjgjgjhjklljjjkjkhhiigj}
and
\er{GMT'3jGHKKkkhjjhgzzZZzzZZzzbvq88nkhhjggjgjkpkjljluytytuguutloklljjgjgjhjklljjjkjkhkhhkjggjkjhhkhohkhjkljjkhhkhjljkjj}
in the case $d>(N+r)$ for sufficiently small $\e>0$ we have
\begin{multline}\label{GMT'3jGHKKkkhjjhgzzZZzzZZzzbvq88nkhhjggjgjkpkjljluytytuguutloklljjgjgjhjklljjjkjkhhiigjhjgjjljklkhhk}
\int\limits_{\Omega}\int\limits_{\Omega}\eta_\e\Big(|y-x|\Big)\frac{F(y,x)}{|y-x|^r}dydx\geq\\
\int\limits_{\mathbb{R}}\int\limits_{\mathbb{R}}\int\limits_{S^{N-1}}\frac{1}{\alpha\big|\ln{\e}\big|s}\,\chi_{[\e^{1+\alpha},+\infty)}(t)\,
\chi_{[1/\e^d,1/\e^{d+\gamma}]}(s)\,
\times\\ \times\,t^{N -1}s\mathcal{L}^N\Bigg(\bigg\{x\in
\Omega\;:\;\frac{\chi_{\Omega}(x+t\vec n)F\big(x+\e\vec
n,x\big)}{|t|^{r+N}}> s\bigg\}\Bigg)d\mathcal{H}^{N-1}(\vec n) dtds.
\end{multline}
On the other hand, since $\Omega$ is bonded, we have
\begin{multline}\label{GMT'3jGHKKkkhjjhgzzZZzzZZzzbvq88nkhhjggjgjkpkjljluytytuguutloklljjgjgjhjklljjjkjkhhiigjhjgjjljklkhhklhkhkgjgg}
\int\limits_{\mathbb{R}}\int\limits_{\mathbb{R}}\int\limits_{S^{N-1}}\frac{1}{\alpha\big|\ln{\e}\big|s}\,\chi_{[0,\e^{1+\alpha}]}(t)\,
\chi_{[1/\e^d,1/\e^{d+\gamma}]}(s)\,
\times\\ \times\,t^{N -1}s\mathcal{L}^N\Bigg(\bigg\{x\in
\Omega\;:\;\frac{\chi_{\Omega}(x+t\vec n)F\big(x+\e\vec
n,x\big)}{|t|^{r+N}}> s\bigg\}\Bigg)d\mathcal{H}^{N-1}(\vec n)
dtds\\
\leq
\int\limits_{\mathbb{R}}\int\limits_{\mathbb{R}}\frac{1}{\alpha\big|\ln{\e}\big|}\,t^{N
-1}\,\chi_{[0,\e^{1+\alpha}]}(t)\,
\chi_{[1/\e^d,1/\e^{d+\gamma}]}(s)\,
\,C dtds\leq
C\frac{\e^{N(1+\alpha)-(d+\gamma)}}{N\alpha\big|\ln{\e}\big|}.
\end{multline}
Moreover, by
\er{GMT'3jGHKKkkhjjhgzzZZzzZZzzbvq88jkhkhhk88jhkhhjhhhjhiyhijkkjkkhkhhhuhhjhjjhiihh1}
we deduce:
\begin{multline}\label{GMT'3jGHKKkkhjjhgzzZZzzZZzzbvq88nkhhjggjgjkpkjljluytytl;klljkljojkojjo}
\int\limits_{\Omega}\int\limits_{\Omega}\eta_\e\Big(|y-x|\Big)\frac{F(y,x)}{|y-x|^r}dydx=\int\limits_{\mathbb{R}^N}\int\limits_{\Omega}\eta_\e\Big(|z|\Big)\chi_{\Omega}(x+z)
\frac{F\big( x+z,x\big)}{|z|^r}dxdz\\=
\int\limits_{S^{N-1}}\int\limits_{\mathbb{R}^+}\int\limits_{\Omega}t^{N-1}\eta_\e\big(t\big)\chi_{\Omega}(x+t\vec
n)\frac{F\big(x+\e\vec n,x\big)}{t^r}dxdtd\mathcal{H}^{N-1}(\vec n)
\\=
\int\limits_{\e^{1+\alpha}}^{\e}\frac{1}{\alpha\big|\ln{\e}\big|t}\Bigg(\int\limits_{S^{N-1}}\int\limits_{\Omega}\chi_{\Omega}(x+t\vec
n)\frac{F\big(x+\e\vec n,x\big)}{t^r}dxd\mathcal{H}^{N-1}(\vec
n)\Bigg)dt.
\end{multline}
Thus, by
\er{GMT'3jGHKKkkhjjhgzzZZzzZZzzbvq88nkhhjggjgjkpkjljluytytuguutloklljjgjgjhjklljjjkjkhhiigjhjgjjljklkhhk},
\er{GMT'3jGHKKkkhjjhgzzZZzzZZzzbvq88nkhhjggjgjkpkjljluytytuguutloklljjgjgjhjklljjjkjkhhiigjhjgjjljklkhhklhkhkgjgg}
and
\er{GMT'3jGHKKkkhjjhgzzZZzzZZzzbvq88nkhhjggjgjkpkjljluytytl;klljkljojkojjo},
in the case $d>(N+r)$ and $N(1+\alpha)-(d+\gamma)>0$ for
sufficiently small $\e>0$ we deduce
\begin{multline}\label{GMT'3jGHKKkkhjjhgzzZZzzZZzzbvq88nkhhjggjgjkpkjljluytytuguutloklljjgjgjhjklljjjkjkhhiigjhjgjjljklkhhkghfh}
\int\limits_{\e^{1+\alpha}}^{\e}\frac{1}{\alpha\big|\ln{\e}\big|t}\Bigg(\int\limits_{S^{N-1}}\int\limits_{\Omega}\chi_{\Omega}(x+t\vec
n)\frac{F\big(x+\e\vec n,x\big)}{t^r}dxd\mathcal{H}^{N-1}(\vec
n)\Bigg)dt\\=
\int\limits_{\Omega}\int\limits_{\Omega}\eta_\e\Big(|y-x|\Big)\frac{F(y,x)}{|y-x|^r}dydx\geq
O\bigg(\frac{1}{\alpha\big|\ln{\e}\big|}\bigg)\\+
\int\limits_{\mathbb{R}}\int\limits_{\mathbb{R}}\int\limits_{S^{N-1}}\frac{\gamma}{\alpha}\,\frac{1}{\gamma\big|\ln{\e}\big|s}\,\chi_{[0,+\infty)}(t)\,
\chi_{[1/\e^d,1/\e^{d+\gamma}]}(s)t^{N -1}sK_{\e,u}\big(t\vec
n,s,r+N\big)d\mathcal{H}^{N-1}(\vec n) dtds=\\
O\bigg(\frac{1}{\alpha\big|\ln{\e}\big|}\bigg)+
\int\limits_{\frac{1}{\e^d}}^{\frac{1}{\e^{d+\gamma}}}\frac{\gamma}{\alpha}\,\frac{1}{\gamma\big|\ln{\e}\big|}\,
\mathcal{L}^{2N}\Bigg(\bigg\{(x,z)\in
\Omega\times\R^N\,:\,\frac{\chi_{\Omega}(x+z)F\big(
x+z,x\big)}{|z|^{r+N}}> s\bigg\}\Bigg)ds
\\=
O\bigg(\frac{1}{\alpha\big|\ln{\e}\big|}\bigg)+
\int\limits_{\frac{1}{\e^d}}^{\frac{1}{\e^{d+\gamma}}}\frac{\gamma}{\alpha}\,\frac{1}{\gamma\big|\ln{\e}\big|}\,
\mathcal{L}^{2N}\Bigg(\bigg\{(x,y)\in
\Omega\times\Omega\,:\,\frac{F(y,x)}{|y-x|^{r+N}}> s\bigg\}\Bigg)ds
.
\end{multline}
Then, by
\er{GMT'3jGHKKkkhjjhgzzZZzzZZzzbvq88nkhhjggjgjkpkjljluytytuguutloklljjgjgjhjklljjjkjkhhiigjhjgjjljklkhhkghfh}
we infer
\begin{multline}\label{GMT'3jGHKKkkhjjhgzzZZzzZZzzbvq88nkhhjggjgjkpkjljluytytuguutloklljjgjgjhjklljjjkjkhhiigjhjgjjljklkhhkghfhk;llkljklljj}
\int\limits_{\e^{1+\alpha}}^{\e}\frac{1}{\alpha\big|\ln{\e}\big|t}\Bigg(\int\limits_{S^{N-1}}\int\limits_{\Omega}\chi_{\Omega}(x+t\vec
n)\frac{F\big(x+\e\vec n,x\big)}{t^r}dxd\mathcal{H}^{N-1}(\vec
n)\Bigg)dt\\=
\int\limits_{\Omega}\int\limits_{\Omega}\eta_\e\Big(|y-x|\Big)\frac{F(y,x)}{|y-x|^r}dydx\geq O\bigg(\frac{1}{\alpha\big|\ln{\e}\big|}\bigg)+\\
\int\limits_{\mathbb{R}}\frac{\gamma}{\alpha}\,\frac{1}{\gamma
s\big|\ln{\e}\big|}\, \chi_{[1/\e^d,1/\e^{d+\gamma}]}(s)\,
s\mathcal{L}^{2N}\Bigg(\bigg\{(x,y)\in
\Omega\times\Omega\;:\;\frac{F(y,x)}{|y-x|^{r+N}}> s\bigg\}\Bigg)ds.
\end{multline}
On the other hand, we have
\begin{multline}\label{GMT'3jGHKKkkhjjhgzzZZzzZZzzbvq88nkhhjggjgjkpkjljluytytl;klljkljojkojjo;k;kklklklkiljluikljhkjhjhmhhgnuhjhjggg}
\inf\limits_{t\in(0,\e)}\Bigg(\int\limits_{S^{N-1}}\int\limits_{\Omega}\chi_{\Omega}(x+t\vec
n)\frac{F\big(x+\e\vec n,x\big)}{t^r}dxd\mathcal{H}^{N-1}(\vec
n)\Bigg)
\leq\int\limits_{\Omega}\int\limits_{\Omega}\eta_\e\Big(|y-x|\Big)\frac{F(y,x)}{|y-x|^r}dydx\\=
\int\limits_{\e^{1+\alpha}}^{\e}\frac{1}{\alpha\big|\ln{\e}\big|t}\Bigg(\int\limits_{S^{N-1}}\int\limits_{\Omega}\chi_{\Omega}(x+t\vec
n)\frac{F\big(x+\e\vec n,x\big)}{t^r}dxd\mathcal{H}^{N-1}(\vec
n)\Bigg)dt
\\
\leq\sup\limits_{t\in(0,\e)}\Bigg(\int\limits_{S^{N-1}}\int\limits_{\Omega}\chi_{\Omega}(x+t\vec
n)\frac{F\big(x+\e\vec n,x\big)}{t^r}dxd\mathcal{H}^{N-1}(\vec
n)\Bigg),
\end{multline}
and
\begin{multline}\label{GMT'3jGHKKkkhjjhgzzZZzzZZzzbvq88nkhhjggjgjkpkjljluytytuguutloklljjgjgjhjklljjjkjkhjkkhkhhkhhjlkkhkjljjljjhhkhihhjhh}
\inf\limits_{s>(1/\e^d)}\Bigg\{s\,\mathcal{L}^{2N}\Bigg(\bigg\{(x,y)\in
\Omega\times\Omega\;:\;\frac{
F(y,x)}{|y-x|^{r+N}}> s\bigg\}\Bigg)\Bigg\}\leq\\
\int\limits_{\mathbb{R}}\,\frac{1}{\gamma\big|\ln{\e}\big|s}\,
\chi_{[1/\e^d,1/\e^{d+\gamma}]}(s)\,
s\mathcal{L}^{2N}\Bigg(\bigg\{(x,y)\in \Omega\times\Omega\;:\;\frac{
F(y,x)}{|y-x|^{r+N}}> s\bigg\}\Bigg)ds\\
\leq
\sup\limits_{s>(1/\e^d)}\Bigg\{s\,\mathcal{L}^{2N}\Bigg(\bigg\{(x,y)\in
\Omega\times\Omega\;:\;\frac{
F(y,x)}{|y-x|^{r+N}}> s\bigg\}\Bigg)\Bigg\}.
\end{multline}
Therefore, inserting these into
\er{GMT'3jGHKKkkhjjhgzzZZzzZZzzbvq88nkhhjggjgjkpkjljluytytuguutloklljjgjgjhjklljjjkjkhhiigjhjgjjljklkhhkghfhk;llkljklljj}
gives, that in the case $d>(N+r)$ and $N(1+\alpha)-(d+\gamma)> 0$
for sufficiently small $\e>0$ we have:
\begin{multline}\label{GMT'3jGHKKkkhjjhgzzZZzzZZzzbvq88nkhhjggjgjkpkjljluytytuguutloklljjgjgjhjklljjjkjkhhiigjhjgjjljklkhhkghfhk;llkljklljjhjbjjbljj}
\sup\limits_{t\in(0,\e)}\Bigg(\int\limits_{S^{N-1}}\int\limits_{\Omega}\chi_{\Omega}(x+t\vec
n)\frac{F\big(x+\e\vec n,x\big)}{t^r}dxd\mathcal{H}^{N-1}(\vec n)\Bigg)\\
\geq O\bigg(\frac{1}{\alpha\big|\ln{\e}\big|}\bigg)+
\frac{\gamma}{\alpha}\inf\limits_{s>(1/\e^d)}\Bigg\{s\,\mathcal{L}^{2N}\Bigg(\bigg\{(x,y)\in
\Omega\times\Omega\;:\;\frac{
F(y,x)}{|y-x|^{r+N}}> s\bigg\}\Bigg)\Bigg\}.
\end{multline}
Thus, letting $\e\to 0^+$ in
\er{GMT'3jGHKKkkhjjhgzzZZzzZZzzbvq88nkhhjggjgjkpkjljluytytuguutloklljjgjgjhjklljjjkjkhhiigjhjgjjljklkhhkghfhk;llkljklljjhjbjjbljj}
gives in the case $d>(N+r)$ and $N(1+\alpha)-(d+\gamma)>0$:
\begin{multline}\label{GMT'3jGHKKkkhjjhgzzZZzzZZzzbvq88nkhhjggjgjkpkjljluytytuguutloklljjgjgjhjklljjjkjkhhiigjhjgjjljklkhhkghfhk;llkljklljjhjbjjbljjhuyuhyu}
\limsup\limits_{\e\to
0^+}\Bigg(\int\limits_{S^{N-1}}\int\limits_{\Omega}\chi_{\Omega}(x+\e\vec
n)\frac{F\big(x+\e\vec n,x\big)}{\e^r}dxd\mathcal{H}^{N-1}(\vec n)\Bigg)\\
\geq \frac{\gamma}{\alpha}\,
\liminf\limits_{s\to+\infty}\Bigg\{s\,\mathcal{L}^{2N}\Bigg(\bigg\{(x,y)\in
\Omega\times\Omega\;:\;\frac{
F(y,x)}{|y-x|^{r+N}}> s\bigg\}\Bigg)\Bigg\}.
\end{multline}
In particular, given $\delta>0$ sufficiently small,
\er{GMT'3jGHKKkkhjjhgzzZZzzZZzzbvq88nkhhjggjgjkpkjljluytytuguutloklljjgjgjhjklljjjkjkhhiigjhjgjjljklkhhkghfhk;llkljklljjhjbjjbljjhuyuhyu}
holds for $\gamma= N(1+\alpha)-d-\delta$ and  $d=(N+r)+\delta$ in
the case of sufficiently large $\alpha>0$. Thus, letting $\delta\to
0^+$,
we deduce by
\er{GMT'3jGHKKkkhjjhgzzZZzzZZzzbvq88nkhhjggjgjkpkjljluytytuguutloklljjgjgjhjklljjjkjkhhiigjhjgjjljklkhhkghfhk;llkljklljjhjbjjbljjhuyuhyu}:
\begin{multline}\label{GMT'3jGHKKkkhjjhgzzZZzzZZzzbvq88nkhhjggjgjkpkjljluytytuguutloklljjgjgjhjklljjjkjkhhiigjhjgjjljklkhhkghfhk;llkljklljjhjbjjbljjhuyuhyulkjjk}
\limsup\limits_{\e\to
0^+}\Bigg(\int\limits_{S^{N-1}}\int\limits_{\Omega}\chi_{\Omega}(x+\e\vec
n)\frac{F\big(x+\e\vec n,x\big)}{\e^r}dxd\mathcal{H}^{N-1}(\vec n)\Bigg)\\
\geq \frac{N\alpha-r}{\alpha}\,
\liminf\limits_{s\to+\infty}\Bigg\{s\,\mathcal{L}^{2N}\Bigg(\bigg\{(x,y)\in
\Omega\times\Omega\;:\;\frac{
F(y,x)}{|y-x|^{r+N}}> s\bigg\}\Bigg)\Bigg\}.
\end{multline}
Then letting $\alpha\to +\infty$ in
\er{GMT'3jGHKKkkhjjhgzzZZzzZZzzbvq88nkhhjggjgjkpkjljluytytuguutloklljjgjgjhjklljjjkjkhhiigjhjgjjljklkhhkghfhk;llkljklljjhjbjjbljjhuyuhyulkjjk}
we infer
\er{GMT'3jGHKKkkhjjhgzzZZzzZZzzbvq88nkhhjggjgjkpkjljluytytuguutloklljjgjgjhjklljjjkjkhhiigjhjgjjljklkhhkghfhk;llkljklljjhjbjjbljjhuyuhyulkjjkhhggh}.
\end{proof}

\section{Some consequences of Theorem \ref{hjkgjkfhjff} in the case $r=q$}
\label{sec:r=q} This section is devoted to the proof of
\rth{hjkgjkfhjffgggvggoopikhhhkjh}, that will follow
 from Corollary \ref{hjkgjkfhjffgggvggoopikhh} and Lemma
\ref{fhgolghoi} after proving some intermediate results. We start
with the following Proposition:
\begin{proposition}\label{hjkgjkfhjffgggvggoop}
Let $\Omega\subset\R^N$ be an open set, $q\geq 1$ and $u\in
L^q(\Omega,\R^m)
$. Furthermore, let $G\subset\Omega$ be an open subset, such that
$\mathcal{L}^N(\partial G)=0$ and either $G$ is convex or $\ov
G\subset\subset\Omega$. Then,
\begin{multline}\label{GMT'3jGHKKkkhjjhgzzZZzzZZzzbvq88nkhhjggjgjkpkjljluytytuguutloklljjgjgjhjklljjjkjkhjkkhkhhkhhjlkkhkjljljkjlkkhkllhjhjhhfyfp2}
\limsup\limits_{\e\to
0^+}\Bigg(\int\limits_{S^{N-1}}\int\limits_{G}\frac{\big|u( x+\e\vec
n)-u(x)\big|^q}{\e^q}\chi_{G}(x+\e\vec
n)dxd\mathcal{H}^{N-1}(\vec n)\Bigg)\\
\leq\sup\limits_{\e\in(0,h)}\Bigg(\int\limits_{S^{N-1}}\int\limits_{G}\frac{\big|u(
x+\e\vec n)-u(x)\big|^q}{\e^q}\chi_{G}(x+\e\vec
n)dxd\mathcal{H}^{N-1}(\vec n)\Bigg)\\
\leq
(N+q)\limsup\limits_{s\to+\infty}\Bigg\{s\,\mathcal{L}^{2N}\Bigg(\bigg\{(x,y)\in
\Omega\times\Omega\;:\;\frac{
\big|u( y)-u(x)\big|^q}{|y-x|^{q+N}}> s\bigg\}\Bigg)\Bigg\},
\end{multline}
where
\begin{equation}\label{fhfuyfyufuyurressgstr2}
h:=\begin{cases} +\infty\quad\quad\text{if $G$ is convex}\,,\\
\dist(G,\R^N\setminus\Omega)\quad\quad\text{otherwise}\,.
\end{cases}
\end{equation}
\end{proposition}
\begin{proof}
In the case of bounded $\Omega$ and $u\in L^\infty(\Omega,\R^m)$,
\er{GMT'3jGHKKkkhjjhgzzZZzzZZzzbvq88nkhhjggjgjkpkjljluytytuguutloklljjgjgjhjklljjjkjkhjkkhkhhkhhjlkkhkjljljkjlkkhkllhjhjhhfyfp2}
follow from Theorem \ref{hjkgjkfhjff}
together with either Corollary
\ref{gughfgfhfgdgddffddfKKzzbvqhigygygtyuu2} or Corollary
\ref{gughfgfhfgdgddffddfKKzzbvqhigygygtyuu} from the Appendix. So it
remains to prove
\er{GMT'3jGHKKkkhjjhgzzZZzzZZzzbvq88nkhhjggjgjkpkjljluytytuguutloklljjgjgjhjklljjjkjkhjkkhkhhkhhjlkkhkjljljkjlkkhkllhjhjhhfyfp2}
in the case of unbounded $\Omega$ and/or $u\notin
L^\infty(\Omega,\R^m)$. Thus for every $k\in \mathbb{N}$ consider a
bounded open sets $G_k\subset\Omega_k\subset\Omega$ with
$\mathcal{L}^N(\partial G_k)=0$, defined by
\begin{equation}\label{hkjkgkfjfkkkghjggfhfhfhjgjghlkhigiukpoihghgfhihj2}
G_k:=G\cap B_k(0)\quad\text{and}\quad\Omega_k:=\Omega\cap B_k(0)\,,
\end{equation}
and consider
$u^{(k)}(x):=\big(u^{(k)}_1(x),\ldots,u^{(k)}_m(x)\big)\in
L^\infty(\Omega,\R^m)\cap L^q(\Omega,\R^m)$, defined by
\begin{equation}\label{hkjkgkfjfkkkghjggfhfhfhjgjghlkhigiukpoihghgfhihjghffhkhgh2}
u^{(k)}_j(x):=\begin{cases}-k\quad\text{if}\quad u(x)\leq -k,\\
u(x)\quad\text{if}\quad u(x)\in [-k,k],\\
k\quad\text{if}\quad u(x)\geq k,
\end{cases}\quad\quad\quad\forall
x\in\Omega\quad\quad\forall\, j\in\{1,\ldots, m\}\,.
\end{equation}
Then obviously,
\begin{equation}\label{hkjkgkfjfkkkghjggfhfhfhjgjghlkhigiukpoihghgfhihjghffhkhghihhjjb2}
\Big|u^{(k)}(y)-u^{(k)}(x)\Big|^q\leq
\Big|u(y)-u(x)\Big|^q\quad\quad\forall (x,y)\in
\Omega\times\Omega\,,\quad\quad\forall\, k\in \mathbb{N}\,.
\end{equation}
and
\begin{equation}\label{hkjkgkfjfkkkghjggfhfhfhjgjghlkhigiukpoihghgfhihjghffhkhghihhjjbhiihhhh2}
\lim\limits_{k\to+\infty}u^{(k)}(x)=u(x)\quad\quad\forall x\in
\Omega\,.
\end{equation}
Moreover, if $G$ is convex then $G_k$ is also convex. Otherwise,
$G_k=G$ for sufficiently large $k$ and
$\dist(G_k,\R^N\setminus\Omega_k)=\dist(G,\R^N\setminus\Omega)$ for
sufficiently large $k$. Thus, by
\er{GMT'3jGHKKkkhjjhgzzZZzzZZzzbvq88nkhhjggjgjkpkjljluytytuguutloklljjgjgjhjklljjjkjkhjkkhkhhkhhjlkkhkjljljkjlkkhkllhjhjhhfyfp2}
with $\Omega_k$ instead of $\Omega$ and $u^{(k)}$ instead of $u$,
for sufficiently large $k$ we have
\begin{multline}\label{GMT'3jGHKKkkhjjhgzzZZzzZZzzbvq88nkhhjggjgjkpkjljluytytuguutloklljjgjgjhjklljjjkjkhjkkhkhhkhhjlkkhkjljljkjlkkhkllhjhjhhfyfphiih2}
\sup\limits_{\e\in(0,h)}\Bigg(\int\limits_{S^{N-1}}\int\limits_{G_k}\frac{\Big|u^{(k)}(
x+\e\vec n)-u^{(k)}(x)\Big|^q}{\e^q}\chi_{G_k}(x+\e\vec
n)dxd\mathcal{H}^{N-1}(\vec n)\Bigg)\\ \leq
(N+q)\limsup\limits_{s\to+\infty}\Bigg\{s\,\mathcal{L}^{2N}\Bigg(\bigg\{(x,y)\in
\Omega_k\times\Omega_k\;:\;\frac{
\big|u^{(k)}( y)-u^{(k)}(x)\big|^q}{|y-x|^{q+N}}>
s\bigg\}\Bigg)\Bigg\}\quad\quad\forall k\in\mathbb{N}.
\end{multline}
Thus, since $\Omega_k\subset\Omega$ by
\er{hkjkgkfjfkkkghjggfhfhfhjgjghlkhigiukpoihghgfhihjghffhkhghihhjjb2}
we deduce from
\er{GMT'3jGHKKkkhjjhgzzZZzzZZzzbvq88nkhhjggjgjkpkjljluytytuguutloklljjgjgjhjklljjjkjkhjkkhkhhkhhjlkkhkjljljkjlkkhkllhjhjhhfyfphiih2}
that, for sufficiently large $k$ we have
\begin{multline}\label{GMT'3jGHKKkkhjjhgzzZZzzZZzzbvq88nkhhjggjgjkpkjljluytytuguutloklljjgjgjhjklljjjkjkhjkkhkhhkhhjlkkhkjljljkjlkkhkllhjhjhhfyfphiihgfggffhjhg2}
\int\limits_{S^{N-1}}\int\limits_{G_k}\frac{\Big|u^{(k)}( x+\e\vec
n)-u^{(k)}(x)\Big|^q}{\e^q}\chi_{G_k}(x+\e\vec
n)dxd\mathcal{H}^{N-1}(\vec n)\leq \\
(N+q)\limsup\limits_{s\to+\infty}\Bigg\{s\,\mathcal{L}^{2N}\Bigg(\bigg\{(x,y)\in
\Omega\times\Omega\;:\;\frac{
\big|u( y)-u(x)\big|^q}{|y-x|^{q+N}}>
s\bigg\}\Bigg)\Bigg\}\quad\quad\forall\,\e\in(0,h)\,,\quad\quad\forall
k\in\mathbb{N}\,.
\end{multline}
Then, letting $k\to+\infty$ in
\er{GMT'3jGHKKkkhjjhgzzZZzzZZzzbvq88nkhhjggjgjkpkjljluytytuguutloklljjgjgjhjklljjjkjkhjkkhkhhkhhjlkkhkjljljkjlkkhkllhjhjhhfyfphiihgfggffhjhg2},
using
\er{hkjkgkfjfkkkghjggfhfhfhjgjghlkhigiukpoihghgfhihjghffhkhghihhjjb2},
\er{hkjkgkfjfkkkghjggfhfhfhjgjghlkhigiukpoihghgfhihjghffhkhghihhjjbhiihhhh2}
and the Dominated Convergence Theorem, gives
\begin{multline}\label{GMT'3jGHKKkkhjjhgzzZZzzZZzzbvq88nkhhjggjgjkpkjljluytytuguutloklljjgjgjhjklljjjkjkhjkkhkhhkhhjlkkhkjljljkjlkkhkllhjhjhhfyfphiihgfggffhjhghjj2}
\int\limits_{S^{N-1}}\int\limits_{G}\frac{\Big|u( x+\e\vec
n)-u(x)\Big|^q}{\e^q}\chi_{G}(x+\e\vec
n)dxd\mathcal{H}^{N-1}(\vec n)\leq \\
(N+q)\limsup\limits_{s\to+\infty}\Bigg\{s\,\mathcal{L}^{2N}\Bigg(\bigg\{(x,y)\in
\Omega\times\Omega\;:\;\frac{
\big|u( y)-u(x)\big|^q}{|y-x|^{q+N}}>
s\bigg\}\Bigg)\Bigg\}\quad\quad\quad\quad\forall\,\e\in(0,h)\,.
\end{multline}
In particular,
\begin{multline}\label{GMT'3jGHKKkkhjjhgzzZZzzZZzzbvq88nkhhjggjgjkpkjljluytytuguutloklljjgjgjhjklljjjkjkhjkkhkhhkhhjlkkhkjljljkjlkkhkllhjhjhhfyfphgg2}
\sup\limits_{\e\in(0,h)}\Bigg(\int\limits_{S^{N-1}}\int\limits_{G}\frac{\big|u(
x+\e\vec n)-u(x)\big|^q}{\e^q}\chi_{G}(x+\e\vec
n)dxd\mathcal{H}^{N-1}(\vec n)\Bigg)\\ \leq
(N+q)\limsup\limits_{s\to+\infty}\Bigg\{s\,\mathcal{L}^{2N}\Bigg(\bigg\{(x,y)\in
\Omega\times\Omega\;:\;\frac{
\big|u( y)-u(x)\big|^q}{|y-x|^{q+N}}> s\bigg\}\Bigg)\Bigg\}.
\end{multline}
Thus, by
\er{GMT'3jGHKKkkhjjhgzzZZzzZZzzbvq88nkhhjggjgjkpkjljluytytuguutloklljjgjgjhjklljjjkjkhjkkhkhhkhhjlkkhkjljljkjlkkhkllhjhjhhfyfphgg2}
we finally deduce
\er{GMT'3jGHKKkkhjjhgzzZZzzZZzzbvq88nkhhjggjgjkpkjljluytytuguutloklljjgjgjhjklljjjkjkhjkkhkhhkhhjlkkhkjljljkjlkkhkllhjhjhhfyfp2}.
\end{proof}

\begin{corollary}\label{hjkgjkfhjffgggvggoop2}
Let $\Omega\subset\R^N$ be a convex open domain, such that
$\mathcal{L}^N(\partial\Omega)=0$, $q\geq 1$ and $u\in
L^q(\Omega,\R^m)
$. Then,
\begin{multline}\label{GMT'3jGHKKkkhjjhgzzZZzzZZzzbvq88nkhhjggjgjkpkjljluytytuguutloklljjgjgjhjklljjjkjkhjkkhkhhkhhjlkkhkjljljkjlkkhkllhjhjhhfyfp}
\sup\limits_{\e\in(0,+\infty)}\Bigg(\int\limits_{S^{N-1}}\int\limits_{\Omega}\frac{\big|u(
x+\e\vec n)-u(x)\big|^q}{\e^q}\chi_{\Omega}(x+\e\vec
n)dxd\mathcal{H}^{N-1}(\vec n)\Bigg)\\=\lim\limits_{\e\to
0^+}\Bigg(\int\limits_{S^{N-1}}\int\limits_{\Omega}\chi_{\Omega}(x+\e\vec
n)\frac{\big|u( x+\e\vec
n)-u(x)\big|^q}{\e^q}dxd\mathcal{H}^{N-1}(\vec n)\Bigg)\\
\leq
(N+q)\limsup\limits_{s\to+\infty}\Bigg\{s\,\mathcal{L}^{2N}\Bigg(\bigg\{(x,y)\in
\Omega\times\Omega\;:\;\frac{
\big|u( y)-u(x)\big|^q}{|y-x|^{q+N}}> s\bigg\}\Bigg)\Bigg\}.
\end{multline}
Moreover, in the case of bounded $\Omega$ and $u\in
L^\infty(\Omega,\R^m)$ we also have
\begin{multline}\label{GMT'3jGHKKkkhjjhgzzZZzzZZzzbvq88nkhhjggjgjkpkjljluytytuguutloklljjgjgjhjklljjjkjkhhiigjhjgjjljklkhhkghfhk;llkljklljjhjbjjbljjhuyuhyulkjjkhhgghijjffgyfghjhopp}
\sup\limits_{\e\in(0,+\infty)}\Bigg(\int\limits_{S^{N-1}}\int\limits_{\Omega}\frac{\big|u(
x+\e\vec n)-u(x)\big|^q}{\e^q}\chi_{\Omega}(x+\e\vec
n)dxd\mathcal{H}^{N-1}(\vec n)\Bigg)\\=\lim\limits_{\e\to
0^+}\Bigg(\int\limits_{S^{N-1}}\int\limits_{\Omega}\chi_{\Omega}(x+\e\vec
n)\frac{\big|u( x+\e\vec
n)-u(x)\big|^q}{\e^q}dxd\mathcal{H}^{N-1}(\vec n)\Bigg)\\
\geq N\,
\liminf\limits_{s\to+\infty}\Bigg\{s\,\mathcal{L}^{2N}\Bigg(\bigg\{(x,y)\in
\Omega\times\Omega\;:\;\frac{
\big|u( y)-u(x)\big|^q}{|y-x|^{q+N}}> s\bigg\}\Bigg)\Bigg\}.
\end{multline}
\end{corollary}
\begin{proof}
In the case of bounded $\Omega$ and $u\in L^\infty(\Omega,\R^m)$,
\er{GMT'3jGHKKkkhjjhgzzZZzzZZzzbvq88nkhhjggjgjkpkjljluytytuguutloklljjgjgjhjklljjjkjkhjkkhkhhkhhjlkkhkjljljkjlkkhkllhjhjhhfyfp}
and
\er{GMT'3jGHKKkkhjjhgzzZZzzZZzzbvq88nkhhjggjgjkpkjljluytytuguutloklljjgjgjhjklljjjkjkhhiigjhjgjjljklkhhkghfhk;llkljklljjhjbjjbljjhuyuhyulkjjkhhgghijjffgyfghjhopp}
follow from Theorem \ref{hjkgjkfhjff}
together with Corollary \ref{gughfgfhfgdgddffddfKKzzbvqhigygygtyuu}
from the Appendix. On the other hand, in in the case of unbounded
$\Omega$ and/or $u\notin L^\infty(\Omega,\R^m)$, in order to prove
\er{GMT'3jGHKKkkhjjhgzzZZzzZZzzbvq88nkhhjggjgjkpkjljluytytuguutloklljjgjgjhjklljjjkjkhjkkhkhhkhhjlkkhkjljljkjlkkhkllhjhjhhfyfp}
we use Proposition \ref{hjkgjkfhjffgggvggoop} with $G=\Omega$,
together with Corollary \ref{gughfgfhfgdgddffddfKKzzbvqhigygygtyuu}
from the Appendix.
\end{proof}

\begin{corollary}\label{hjkgjkfhjffgggvggoopikhh}
Let $\Omega\subset\R^N$ be an open domain, $q\geq 1$ and $u\in
L^q(\Omega,\R^m)
$. Then, in the case $q>1$ we have
\begin{multline}\label{GMT'3jGHKKkkhjjhgzzZZzzZZzzbvq88nkhhjggjgjkpkjljluytytuguutloklljjgjgjhjklljjjkjkhjkkhkhhkhhjlkkhkjljljkjlkkhkllhjhjhhfyfppiooiououiuiuiuhjhjkhkhkjkhhk}
\frac{\int_{S^{N-1}}|z_1|^qd\mathcal{H}^{N-1}(z)}{(N+q)}\,\int_\Omega\big|\nabla
u(x)\big|^qdx\leq \\
\limsup\limits_{s\to+\infty}\Bigg\{s\,\mathcal{L}^{2N}\Bigg(\bigg\{(x,y)\in
\Omega\times \Omega\;:\;\frac{
\big|u( y)-u(x)\big|^q}{|y-x|^{q+N}}> s\bigg\}\Bigg)\Bigg\}\,,
\end{multline}
with the convention that $\int_\Omega\big|\nabla
u(x)\big|^qdx=+\infty$ if $u\notin W^{1,q}(\Omega,\R^m)$. On the
other hand, in the case $q=1$ we have:
\begin{multline}\label{GMT'3jGHKKkkhjjhgzzZZzzZZzzbvq88nkhhjggjgjkpkjljluytytuguutloklljjgjgjhjklljjjkjkhjkkhkhhkhhjlkkhkjljljkjlkkhkllhjhjhhfyfppiooiououiuiuiuhjhjkhkhkjkjjkkhjkhkhjhjh}
\frac{\int_{S^{N-1}}|z_1|d\mathcal{H}^{N-1}(z)}{(N+1)}\,\|Du\|(\Omega)
\leq\\
\limsup\limits_{s\to+\infty}\Bigg\{s\,\mathcal{L}^{2N}\Bigg(\bigg\{(x,y)\in
\Omega\times \Omega\;:\;\frac{
\big|u( y)-u(x)\big|}{|y-x|^{1+N}}> s\bigg\}\Bigg)\Bigg\}\,,
\end{multline}
with the convention $\|Du\|(\Omega)=+\infty$ if $u\notin
BV(\Omega,\R^m)$.
\end{corollary}
\begin{proof}
Let $\rho_\e\big(|z|\big):\R^N\to[0,+\infty)$ be radial mollifiers,
so that $\int_{\R^N}\rho_\e\big(|z|\big)dz=1$ and for every $r>0$
there exits $\delta:=\delta_r>0$, such that $\supp{(\rho_\e)}\subset
B_r(0)$ for every $\e\in(0,\delta_r)$. Next fix an open subset
$G\subset\subset\Omega$ with Lipschitz boundary. Since, by
Proposition \ref{hjkgjkfhjffgggvggoop} we have
\begin{multline}\label{GMT'3jGHKKkkhjjhgzzZZzzZZzzbvq88nkhhjggjgjkpkjljluytytuguutloklljjgjgjhjklljjjkjkhjkkhkhhkhhjlkkhkjljljkjlkkhkllhjhjhhfyfppiooiououiuiuiuhjh}
\limsup\limits_{\e\to
0^+}\Bigg(\int\limits_{S^{N-1}}\int\limits_{G}\chi_{G}(x+\e\vec
n)\frac{\big|u( x+\e\vec
n)-u(x)\big|^q}{\e^q}dxd\mathcal{H}^{N-1}(\vec n)\Bigg)\\
\leq
(N+q)\limsup\limits_{s\to+\infty}\Bigg\{s\,\mathcal{L}^{2N}\Bigg(\bigg\{(x,y)\in
\Omega\times\Omega\;:\;\frac{
\big|u( y)-u(x)\big|^q}{|y-x|^{q+N}}> s\bigg\}\Bigg)\Bigg\}\,,
\end{multline}
and at the same time by Lemma \ref{gjyfyfuyyfifgyify} we have:
\begin{multline}\label{GMT'3jGHKKkkhjjhgzzZZzzZZzzbvq88nkhhjggjgjkpkjljluytytl;klljkljojkojjo;k;kklklklkiljluikljjhkjhhjjhpoijgggjgljj}
\frac{1}{\mathcal{H}^{N-1}(S^{N-1})}\,\limsup\limits_{\e\to
0^+}\Bigg(\int\limits_{S^{N-1}}\int\limits_{G}\chi_{G}(x+\e\vec
n)\frac{\big|u( x+\e\vec
n)-u(x)\big|^q}{\e^q}dxd\mathcal{H}^{N-1}(\vec n)\Bigg)\\
\geq \limsup\limits_{\e\to
0^+}\int\limits_{G}\int\limits_{G}\rho_\e\Big(|y-x|\Big)\frac{\big|u(
y)-u(x)\big|^q}{|y-x|^q}dydx\,,
\end{multline}
we deduce from
\er{GMT'3jGHKKkkhjjhgzzZZzzZZzzbvq88nkhhjggjgjkpkjljluytytuguutloklljjgjgjhjklljjjkjkhjkkhkhhkhhjlkkhkjljljkjlkkhkllhjhjhhfyfppiooiououiuiuiuhjh}
and
\er{GMT'3jGHKKkkhjjhgzzZZzzZZzzbvq88nkhhjggjgjkpkjljluytytl;klljkljojkojjo;k;kklklklkiljluikljjhkjhhjjhpoijgggjgljj}
that
\begin{multline}\label{GMT'3jGHKKkkhjjhgzzZZzzZZzzbvq88nkhhjggjgjkpkjljluytytuguutloklljjgjgjhjklljjjkjkhjkkhkhhkhhjlkkhkjljljkjlkkhkllhjhjhhfyfppiooiououiuiuiuhjh1}
\limsup\limits_{\e\to
0^+}\int\limits_{G}\int\limits_{G}\rho_\e\Big(|y-x|\Big)\frac{\big|u(
y)-u(x)\big|^q}{|y-x|^q}dydx\\
\leq
\frac{(N+q)}{\mathcal{H}^{N-1}(S^{N-1})}\,\limsup\limits_{s\to+\infty}\Bigg\{s\,\mathcal{L}^{2N}\Bigg(\bigg\{(x,y)\in
\Omega\times\Omega\;:\;\frac{
\big|u( y)-u(x)\big|^q}{|y-x|^{q+N}}> s\bigg\}\Bigg)\Bigg\}.
\end{multline}
On the other hand, since  bounded $G\subset{\mathbb R}^N$ has a
Lipschitz boundary, the so called ``BBM formula'' (see \cite{hhh1}
and \cite{hhh5} for the details) states that for $q>1$ we have
\begin{equation}
\label{eq:1jjj} \lim_{\e\to 0^+} \int_{G}\int_{G}
\frac{|u(x)-u(y)|^q}{|x-y|^q}\,\rho_\e\big(|x-y|\big)\,dx\,dy=K_{q,N}\int_{G}\big|\nabla
u(x)\big|^qdx\,,
\end{equation}
with the convention that $\int_{G}\big|\nabla u(x)\big|^qdx=+\infty$
if $u\notin W^{1,q}(G,\R^m)$ and with $K_{q,N}$ given by
\begin{equation}\label{fgyufghfghjgghgjkhkkGHGHKKggkhhjoozzbvqkk}
K_{q,N}:=\frac{1}{\mathcal{H}^{N-1}(S^{N-1})}\int_{S^{N-1}}|z_1|^qd\mathcal{H}^{N-1}(z)\quad\quad\forall
q\geq 1\,,
\end{equation}
where we denote $z:=(z_1,\ldots, z_N)\in\R^N$. Moreover, for $q=1$
we have
\begin{equation}
\label{eq:1jjj1} \lim_{\e\to 0^+} \int_{G}\int_{G}
\frac{|u(x)-u(y)|}{|x-y|}\,\rho_\e\big(|x-y|\big)\,dx\,dy=K_{1,N}\,\|Du\|(G)\,,
\end{equation}
with the convention $\|Du\|(G)=+\infty$ if $u\notin BV(G,\R^m)$.
Inserting it into
\er{GMT'3jGHKKkkhjjhgzzZZzzZZzzbvq88nkhhjggjgjkpkjljluytytuguutloklljjgjgjhjklljjjkjkhjkkhkhhkhhjlkkhkjljljkjlkkhkllhjhjhhfyfppiooiououiuiuiuhjh1}
gives for $q>1$:
\begin{multline}\label{GMT'3jGHKKkkhjjhgzzZZzzZZzzbvq88nkhhjggjgjkpkjljluytytuguutloklljjgjgjhjklljjjkjkhjkkhkhhkhhjlkkhkjljljkjlkkhkllhjhjhhfyfppiooiououiuiuiuhjhjkhkhk}
\frac{\int_{S^{N-1}}|z_1|^qd\mathcal{H}^{N-1}(z)}{(N+q)}\,\int_{G}\big|\nabla
u(x)\big|^qdx=\frac{K_{q,N}\,\mathcal{H}^{N-1}(S^{N-1})}{(N+q)}\,\int_{G}\big|\nabla
u(x)\big|^qdx\\
\leq
\limsup\limits_{s\to+\infty}\Bigg\{s\,\mathcal{L}^{2N}\Bigg(\bigg\{(x,y)\in
\Omega\times\Omega\;:\;\frac{
\big|u( y)-u(x)\big|^q}{|y-x|^{q+N}}> s\bigg\}\Bigg)\Bigg\}\,,
\end{multline}
and for $q=1$:
\begin{multline}\label{GMT'3jGHKKkkhjjhgzzZZzzZZzzbvq88nkhhjggjgjkpkjljluytytuguutloklljjgjgjhjklljjjkjkhjkkhkhhkhhjlkkhkjljljkjlkkhkllhjhjhhfyfppiooiououiuiuiuhjhjkhkhkjkjjkkhjkhk}
\frac{\int_{S^{N-1}}|z_1|d\mathcal{H}^{N-1}(z)}{(N+1)}\,\|Du\|(G)=
\frac{K_{1,N}\,\mathcal{H}^{N-1}(S^{N-1})}{(N+1)}\,\|Du\|(G)\\
\leq
\limsup\limits_{s\to+\infty}\Bigg\{s\,\mathcal{L}^{2N}\Bigg(\bigg\{(x,y)\in
\Omega\times\Omega\;:\;\frac{
\big|u( y)-u(x)\big|}{|y-x|^{1+N}}> s\bigg\}\Bigg)\Bigg\}\,.
\end{multline}
Thus, taking the supremum of the left hand side of
\er{GMT'3jGHKKkkhjjhgzzZZzzZZzzbvq88nkhhjggjgjkpkjljluytytuguutloklljjgjgjhjklljjjkjkhjkkhkhhkhhjlkkhkjljljkjlkkhkllhjhjhhfyfppiooiououiuiuiuhjhjkhkhk}
over all open $G\subset\subset\Omega$ with Lipschitz boundary, we
deduce
\er{GMT'3jGHKKkkhjjhgzzZZzzZZzzbvq88nkhhjggjgjkpkjljluytytuguutloklljjgjgjhjklljjjkjkhjkkhkhhkhhjlkkhkjljljkjlkkhkllhjhjhhfyfppiooiououiuiuiuhjhjkhkhkjkhhk}
and taking the supremum of the left hand side of
\er{GMT'3jGHKKkkhjjhgzzZZzzZZzzbvq88nkhhjggjgjkpkjljluytytuguutloklljjgjgjhjklljjjkjkhjkkhkhhkhhjlkkhkjljljkjlkkhkllhjhjhhfyfppiooiououiuiuiuhjhjkhkhkjkjjkkhjkhk}
over all open $G\subset\subset\Omega$ with Lipschitz boundary, we
deduce
\er{GMT'3jGHKKkkhjjhgzzZZzzZZzzbvq88nkhhjggjgjkpkjljluytytuguutloklljjgjgjhjklljjjkjkhjkkhkhhkhhjlkkhkjljljkjlkkhkllhjhjhhfyfppiooiououiuiuiuhjhjkhkhkjkjjkkhjkhkhjhjh}.
\end{proof}

\begin{lemma}\label{fhgolghoi}
Let $q\geq 1$ and let $\Omega\subset\R^N$ be a domain with Lipschiz
boundary. Then there exist constants $C:=C_{\Omega}>0$ and
${\widetilde C}_{N}>0$, satisfying $C_{\Omega}=1$ if $\Omega=\R^N$,
such that, in the case $q>1$, for every $u\in W^{1,q}(\Omega,\R^m)$
we have
\begin{multline}\label{GMT'3jGHKKkkhjjhgzzZZzzZZzzbvq88nkhhjggjgjkpkjljluytytuguutloklljjgjgjhjklljjjkjkhjkkhkhhkhhjlkkhkjljljkjlkkhkllhjhjhhfyfppiooiououiuiuiuhjhjkhkhkjkhhkhkkjyukky}
\sup\limits_{s\in(0,+\infty)}\Bigg\{s\,\mathcal{L}^{2N}\Bigg(\bigg\{(x,y)\in
\Omega\times \Omega\;:\;\frac{
\big|u( y)-u(x)\big|^q}{|y-x|^{q+N}}> s\bigg\}\Bigg)\Bigg\} \leq
C^q_{\Omega}{\widetilde C}_{N}\,\int_\Omega\big|\nabla
u(x)\big|^qdx\,,
\end{multline}
and, in the case $q=1$, for every $u\in BV(\Omega,\R^m)$ we have:
\begin{multline}\label{GMT'3jGHKKkkhjjhgzzZZzzZZzzbvq88nkhhjggjgjkpkjljluytytuguutloklljjgjgjhjklljjjkjkhjkkhkhhkhhjlkkhkjljljkjlkkhkllhjhjhhfyfppiooiououiuiuiuhjhjkhkhkjkjjkkhjkhkhjhjhkuy}
\sup\limits_{s\in(0,+\infty)}\Bigg\{s\,\mathcal{L}^{2N}\Bigg(\bigg\{(x,y)\in
\Omega\times \Omega\;:\;\frac{
\big|u( y)-u(x)\big|}{|y-x|^{1+N}}> s\bigg\}\Bigg)\Bigg\} \leq
C_{\Omega}{\widetilde C}_{N}\,\|Du\|(\Omega)\,.
\end{multline}
\end{lemma}
\begin{proof}
By Extension Theorem for Sobolev and $BV$ functions there exist a
constant $C:=C_{\Omega}>0$ such that, in the case $q>1$ for every
$u\in W^{1,q}(\Omega,\R^m)$ there exists its extension $\tilde u\in
W^{1,q}(\R^N,\R^m)$ with the property
\begin{equation}\label{gjfguhjfghfhkjh}
\begin{cases}
\tilde u(x)=u(x)\quad\quad\quad\quad\forall\,x\in\Omega\,,
\\
\int_{\R^N}\big|\nabla \tilde u(x)\big|^qdx\leq
C^q_{\Omega}\,\int_\Omega\big|\nabla u(x)\big|^qdx\,,
\end{cases}
\end{equation}
and in the case $q=1$ for every $u\in BV(\Omega,\R^m)$ there exists
its extension $\tilde u\in BV(\R^N,\R^m)$ with the property
\begin{equation}\label{gjfguhjfghfhkjhyutujgjgh}
\begin{cases}
\tilde u(x)=u(x)\quad\quad\quad\quad\forall\,x\in\Omega\,,
\\
\big\|D\tilde u\big\|(\R^N)\leq C_{\Omega}\,\|Du\|(\Omega)\,.
\end{cases}
\end{equation}
Moreover, in the trivial case $\Omega=\R^N$ we obviously can
consider $C_{\Omega}=1$. Next, by the standard properties of the
Sobolev and the $BV$ functions, there exist a sequence
$\big\{\varphi_n(x)\big\}_{n=1}^{+\infty}\subset
C^\infty_c(\R^N,\R^m)$ such that in the case $q>1$ we have
\begin{equation}\label{gjfguhjfghfhkjhhikh}
\begin{cases}
\varphi_n\to \tilde u\quad\quad\text{strongly in}\quad
L^q(\R^N,\R^m)\,,
\\
\lim\limits_{n\to+\infty}\int_{\R^N}\big|\nabla
\varphi_n(x)\big|^qdx=\int_{\R^N}\big|\nabla \tilde u(x)\big|^qdx\,,
\end{cases}
\end{equation}
and in the case $q=1$ we have
\begin{equation}\label{gjfguhjfghfhkjhhkjhgjgj}
\begin{cases}
\varphi_n\to \tilde u\quad\quad\text{strongly in}\quad
L^1(\R^N,\R^m)\,,
\\
\lim\limits_{n\to+\infty}\int_{\R^N}\big|\nabla
\varphi_n(x)\big|dx=\big\|D \tilde u\big\|(\R^N)\,.
\end{cases}
\end{equation}
On the other hand,  H.~Brezis, J.~Van Schaftingen and Po-Lam~Yung in
\cite{hhh3} or \cite{hhh2} proved that for every $q\geq 1$ there
exists a constant ${\widetilde C}:={\widetilde C}_{N}>0$ such that
for every $\varphi(x)\in C^\infty_c(\R^N,\R^m)$ we have
\begin{multline}\label{GMT'3jGHKKkkhjjhgzzZZzzZZzzbvq88nkhhjggjgjkpkjljluytytuguutloklljjgjgjhjklljjjkjkhjkkhkhhkhhjlkkhkjljljkjlkkhkllhjhjhhfyfppiooiououiuiuiuhjhjkhkhkjkhhkhkkjyukkyjgggjjggjg}
\sup\limits_{s\in(0,+\infty)}\Bigg\{s\,\mathcal{L}^{2N}\Bigg(\bigg\{(x,y)\in
\R^N\times \R^N\;:\;\frac{
\big|\varphi( y)-\varphi(x)\big|^q}{|y-x|^{q+N}}>
s\bigg\}\Bigg)\Bigg\} \leq {\widetilde
C}_{N}\,\int_{\R^N}\big|\nabla \varphi(x)\big|^qdx\,.
\end{multline}
In particular, for every $s>0$ and every $n\in\mathbb{N}$ we have:
\begin{equation}\label{GMT'3jGHKKkkhjjhgzzZZzzZZzzbvq88nkhhjggjgjkpkjljluytytuguutloklljjgjgjhjklljjjkjkhjkkhkhhkhhjlkkhkjljljkjlkkhkllhjhjhhfyfppiooiououiuiuiuhjhjkhkhkjkhhkhkkjyukkyjgggjjggjgkyhj}
s\,\mathcal{L}^{2N}\Bigg(\bigg\{(x,y)\in \R^N\times \R^N\;:\;\frac{
\big|\varphi_n( y)-\varphi_n(x)\big|^q}{|y-x|^{q+N}}> s\bigg\}\Bigg)
\leq {\widetilde C}_{N}\,\int_{\R^N}\big|\nabla
\varphi_n(x)\big|^qdx\,.
\end{equation}
Thus letting $n\to +\infty$ in
\er{GMT'3jGHKKkkhjjhgzzZZzzZZzzbvq88nkhhjggjgjkpkjljluytytuguutloklljjgjgjhjklljjjkjkhjkkhkhhkhhjlkkhkjljljkjlkkhkllhjhjhhfyfppiooiououiuiuiuhjhjkhkhkjkhhkhkkjyukkyjgggjjggjgkyhj}
and using either \er{gjfguhjfghfhkjhhikh} (for $q>1$) or
\er{gjfguhjfghfhkjhhkjhgjgj} (for $q=1$), in the case $q>1$ for
every $s>0$ we deduce
\begin{equation}\label{GMT'3jGHKKkkhjjhgzzZZzzZZzzbvq88nkhhjggjgjkpkjljluytytuguutloklljjgjgjhjklljjjkjkhjkkhkhhkhhjlkkhkjljljkjlkkhkllhjhjhhfyfppiooiououiuiuiuhjhjkhkhkjkhhkhkkjyukkyjgggjjggjgkyhjhkjhjhljkjl}
s\,\mathcal{L}^{2N}\Bigg(\bigg\{(x,y)\in \R^N\times \R^N\;:\;\frac{
\big|\tilde u( y)-\tilde u(x)\big|^q}{|y-x|^{q+N}}> s\bigg\}\Bigg)
\leq {\widetilde C}_{N}\int_{\R^N}\big|\nabla \tilde
u(x)\big|^qdx\,.
\end{equation}
and in the case $q=1$ for every $s>0$ we deduce
\begin{equation}\label{GMT'3jGHKKkkhjjhgzzZZzzZZzzbvq88nkhhjggjgjkpkjljluytytuguutloklljjgjgjhjklljjjkjkhjkkhkhhkhhjlkkhkjljljkjlkkhkllhjhjhhfyfppiooiououiuiuiuhjhjkhkhkjkhhkhkkjyukkyjgggjjggjgkyhjhkjhjhjhkhilkl}
s\,\mathcal{L}^{2N}\Bigg(\bigg\{(x,y)\in \R^N\times \R^N\;:\;\frac{
\big|\tilde u( y)-\tilde u(x)\big|}{|y-x|^{1+N}}> s\bigg\}\Bigg)
\leq {\widetilde C}_{N}\,\big\|D \tilde u\big\|(\R^N)\,.
\end{equation}
Thus, by \er{gjfguhjfghfhkjh} and
\er{GMT'3jGHKKkkhjjhgzzZZzzZZzzbvq88nkhhjggjgjkpkjljluytytuguutloklljjgjgjhjklljjjkjkhjkkhkhhkhhjlkkhkjljljkjlkkhkllhjhjhhfyfppiooiououiuiuiuhjhjkhkhkjkhhkhkkjyukkyjgggjjggjgkyhjhkjhjhljkjl},
in the case $q>1$ for every $s>0$ we deduce
\begin{equation}\label{GMT'3jGHKKkkhjjhgzzZZzzZZzzbvq88nkhhjggjgjkpkjljluytytuguutloklljjgjgjhjklljjjkjkhjkkhkhhkhhjlkkhkjljljkjlkkhkllhjhjhhfyfppiooiououiuiuiuhjhjkhkhkjkhhkhkkjyukkyjgggjjggjgkyhjhkjhjhljkjlhkh}
s\,\mathcal{L}^{2N}\Bigg(\bigg\{(x,y)\in \Omega\times
\Omega\;:\;\frac{ \big|u( y)-u(x)\big|^q}{|y-x|^{q+N}}>
s\bigg\}\Bigg) \leq
{\widetilde C}_{N}\,C^q_{\Omega}\,\int_{\Omega}\big|\nabla
u(x)\big|^qdx\,,
\end{equation}
and by \er{gjfguhjfghfhkjhyutujgjgh} and
\er{GMT'3jGHKKkkhjjhgzzZZzzZZzzbvq88nkhhjggjgjkpkjljluytytuguutloklljjgjgjhjklljjjkjkhjkkhkhhkhhjlkkhkjljljkjlkkhkllhjhjhhfyfppiooiououiuiuiuhjhjkhkhkjkhhkhkkjyukkyjgggjjggjgkyhjhkjhjhjhkhilkl}
in the case $q=1$ for every $s>0$ we deduce
\begin{equation}\label{GMT'3jGHKKkkhjjhgzzZZzzZZzzbvq88nkhhjggjgjkpkjljluytytuguutloklljjgjgjhjklljjjkjkhjkkhkhhkhhjlkkhkjljljkjlkkhkllhjhjhhfyfppiooiououiuiuiuhjhjkhkhkjkhhkhkkjyukkyjgggjjggjgkyhjhkjhjhjhkhilklhkyj}
s\,\mathcal{L}^{2N}\Bigg(\bigg\{(x,y)\in \Omega\times
\Omega\;:\;\frac{ \big|u( y)-u(x)\big|}{|y-x|^{1+N}}> s\bigg\}\Bigg)
\leq {\widetilde C}_{N}\,C_{\Omega}\,\big\|D u\big\|(\Omega)\,.
\end{equation}
Finally, taking the supremum of ether
\er{GMT'3jGHKKkkhjjhgzzZZzzZZzzbvq88nkhhjggjgjkpkjljluytytuguutloklljjgjgjhjklljjjkjkhjkkhkhhkhhjlkkhkjljljkjlkkhkllhjhjhhfyfppiooiououiuiuiuhjhjkhkhkjkhhkhkkjyukkyjgggjjggjgkyhjhkjhjhljkjlhkh}
or
\er{GMT'3jGHKKkkhjjhgzzZZzzZZzzbvq88nkhhjggjgjkpkjljluytytuguutloklljjgjgjhjklljjjkjkhjkkhkhhkhhjlkkhkjljljkjlkkhkllhjhjhhfyfppiooiououiuiuiuhjhjkhkhkjkhhkhkkjyukkyjgggjjggjgkyhjhkjhjhjhkhilklhkyj}
over the set $s\in(0,+\infty)$ completes the proof.
\end{proof}

\begin{proof}[Proof of \rth{hjkgjkfhjffgggvggoopikhhhkjh}]
It suffices to combine Corollary \ref{hjkgjkfhjffgggvggoopikhh} with
Lemma \ref{fhgolghoi}.
\end{proof}

\section{Proof of Theorem
\ref{hjkgjkfhjffgggvggoopikhhhkjhhkjhjgjhhjgggghhdf11}}\label{limtem}
First, we prove the following proposition.
\begin{proposition}\label{fhfhfhfffhkjkj}
Let $\Omega\subset\R^N$ be an open set and let $u\in Lip(\R^N,\R^m)$
be such that there exists $R>0$ satisfying $u(x)=0$ for every
$x\in\R^N\setminus B_R(0)$. Next let $G_0:
\R\times\R^m\times\R^m\times\R^N\times\R^N\to[0,+\infty)$ be a
continuous function, such that $G_0(0,0,0,x,y)=0$ for every
$x,y\in\R^N$ and $G_0$ is bounded on every set of the type
$[\alpha,\beta]\times K\times K\times\R^N\times\R^N$ with any
$K\subset\subset\R^N$ and any $\alpha<\beta\in\R$. Then,
\begin{multline}\label{GMT'3jGHKKkkhjjhgzzZZzzZZzzbvq88nkhhjggjgjkpkjljluytytuguutloklljjgjgjhjklljjjkjkhhiigjhjgjjljklkhhkghfhk;llkljklljjhjbjjbljjhuyuhyulkjjkrr88mhjhhjhjhhhjy1iyyuliyhigugu}
\lim\limits_{s\to +\infty}s\mathcal{L}^{2N}\Bigg(\Bigg\{(x,y)\in
\Omega\times\Omega\;:\;
G_0\bigg(\frac{\big|u(y)-u(x)\big|}{|y-x|},u(y),u(x),y,x\bigg)\,\frac{1}{|y-x|^{N}}>
s\Bigg\}\Bigg)\\
=\frac{1}{N}\int\limits_{\Omega}\Bigg(\int\limits_{S^{N-1}}G_0\bigg(\big|\nabla
u(x)\big||z_1|,u(x),u(x),x,x\bigg)d\mathcal{H}^{N-1}(z)\Bigg)dx.
\end{multline}
\end{proposition}
\begin{proof}
Let $\Omega\subset\R^N$ be an open set and let $u\in Lip(\R^N,\R^m)$
be such that there exists $R>0$ satisfying $u(x)=0$ for every
$x\in\R^N\setminus B_R(0)$. Next let $G:
\R^m\times\R^m\times\R^m\times\R^N\times\R^N\to[0,+\infty)$ be a
continuous function, such that $G(0,0,0,x,y)=0$ for every
$x,y\in\R^N$ and $G$ is bounded on every set of the type $K\times
K\times K\times\R^N\times\R^N$ with any $K\subset\subset\R^N$. Then
for every $s>0$ we have:
\begin{multline}\label{GMT'3jGHKKkkhjjhgzzZZzzZZzzbvq88nkhhjggjgjkpkjljluytytuguutloklljjgjgjhjklljjjkjkhhiigjhjgjjljklkhhkghfhk;llkljklljjhjbjjbljjhuyuhyulkjjkrr88mhjhhjhjhh}
\mathcal{L}^{2N}\Bigg(\Bigg\{(x,y)\in \Omega\times\Omega\;:\;
G\bigg(\frac{u(y)-u(x)}{|y-x|},u(y),u(x),y,x\bigg)\,\frac{1}{|y-x|^{N}}>
s\Bigg\}\Bigg)=\\
\mathcal{L}^{2N}\Bigg(\Bigg\{(x,z)\in \Omega\times\R^N\,:\,
\frac{\chi_{\Omega}(x+z)
}{|z|^{N}} G\bigg(\frac{u(x+z)-u(x)}{|z|},u(x+z),u(x),x+z,x\bigg)>
s\Bigg\}\Bigg)\\= \int\limits_\Omega\mathcal{L}^{N}\Bigg(\Bigg\{z\in
\R^N\,:\, \chi_{\Omega}(x+z)
G\bigg(\frac{u(x+z)-u(x)}{|z|},u(x+z),u(x),x+z,x\bigg)\frac{1}{|z|^{N}}>
s\Bigg\}\Bigg)dx .
\end{multline}
Therefore for every $\e>0$ we have
\begin{multline}\label{GMT'3jGHKKkkhjjhgzzZZzzZZzzbvq88nkhhjggjgjkpkjljluytytuguutloklljjgjgjhjklljjjkjkhhiigjhjgjjljklkhhkghfhk;llkljklljjhjbjjbljjhuyuhyulkjjkrr88mhjhhjhjhhhjy}
\frac{1}{\e^N}\mathcal{L}^{2N}\Bigg(\Bigg\{(x,y)\in
\Omega\times\Omega\;:\;
G\bigg(\frac{u(y)-u(x)}{|y-x|},u(y),u(x),y,x\bigg)\,\frac{1}{|y-x|^{N}}>
\frac{1}{\e^N}\Bigg\}\Bigg)=\\
\int\limits_\Omega\frac{1}{\e^N}\mathcal{L}^{N}\Bigg(\Bigg\{z\in
\R^N\,:\, \chi_{\Omega}(x+z)
G\bigg(\frac{u(x+z)-u(x)}{|z|},u(x+z),u(x),x+z,x\bigg)\frac{1}{|z|^{N}}>
\frac{1}{\e^N}\Bigg\}\Bigg)dx\\=
\int\limits_\Omega\mathcal{L}^{N}\Bigg(\Bigg\{z\in \R^N\,:\,
\chi_{\Omega}(x+\e z) G\bigg(\frac{u(x+\e z)-u(x)}{|\e z|},u(x+\e
z),u(x),x+\e z,x\bigg)\frac{1}{|z|^{N}}> 1\Bigg\}\Bigg)dx .
\end{multline}
Moreover, since $G(0,0,0,x,y)=0$, by the Lipschitz and the compact
support conditions for $u$ we obviously deduce that there exists
$M>0$ such that
\begin{align}\label{GMT'3jGHKKkkhjjhgzzZZzzZZzzbvq88nkhhjggjgjkpkjljluytytuguutloklljjgjgjhjklljjjkjkhhiigjhjgjjljklkhhkghfhk;llkljklljjhjbjjbljjhuyuhyulkjjkrr88mhjhhjhjhhhjyuiyyj}
0\leq G\bigg(\frac{u(y)-u(x)}{|y-x|},u(y),u(x),y,x\bigg)\leq
\Big(\frac{M}{2}\Big)^N\quad\forall
(y,x)\in\R^N\times\R^N\quad\quad\text{and}\quad\quad\\
G\bigg(\frac{u(y)-u(x)}{|y-x|},u(y),u(x),y,x\bigg)=0\quad\quad\text{whenever}\quad\quad
|x|\geq \frac{M}{2}\;\;\text{and}\;\;|y|\geq \frac{M}{2}\,.
\end{align}
In particular, for every $\e\in(0,1)$ we have
\begin{multline}\label{GMT'3jGHKKkkhjjhgzzZZzzZZzzbvq88nkhhjggjgjkpkjljluytytuguutloklljjgjgjhjklljjjkjkhhiigjhjgjjljklkhhkghfhk;llkljklljjhjbjjbljjhuyuhyulkjjkrr88mhjhhjhjhhhjyhkhhhjhh}
\mathcal{L}^{N}\Bigg(\Bigg\{z\in \R^N\,:\, \chi_{\Omega}(x+\e z)
G\bigg(\frac{u(x+\e z)-u(x)}{|\e z|},u(x+\e z),u(x),x+\e
z,x\bigg)\frac{1}{|z|^{N}}> 1\Bigg\}\Bigg)=\\
\mathcal{L}^{N}\Bigg(\Bigg\{z\in B_M(0)\,:\, \chi_{\Omega}(x+\e z)
G\bigg(\frac{u(x+\e z)-u(x)}{|\e z|},u(x+\e z),u(x),x+\e
z,x\bigg)\frac{1}{|z|^{N}}> 1\Bigg\}\Bigg).
\end{multline}
Moreover, for every $\e\in(0,1)$ we have
\begin{multline}\label{GMT'3jGHKKkkhjjhgzzZZzzZZzzbvq88nkhhjggjgjkpkjljluytytuguutloklljjgjgjhjklljjjkjkhhiigjhjgjjljklkhhkghfhk;llkljklljjhjbjjbljjhuyuhyulkjjkrr88mhjhhjhjhhhjyhkhhhjhhyiyyu}
\mathcal{L}^{N}\Bigg(\Bigg\{z\in \R^N\,:\, \chi_{\Omega}(x+\e z)
G\bigg(\frac{u(x+\e z)-u(x)}{|\e z|},u(x+\e z),u(x),x+\e
z,x\bigg)\frac{1}{|z|^{N}}> 1\Bigg\}\Bigg)=0\\
\quad\quad\quad\quad \forall\,x\in\R^N\setminus B_M(0)\,.
\end{multline}
Thus, for every $\e\in(0,1)$ we rewrite
\er{GMT'3jGHKKkkhjjhgzzZZzzZZzzbvq88nkhhjggjgjkpkjljluytytuguutloklljjgjgjhjklljjjkjkhhiigjhjgjjljklkhhkghfhk;llkljklljjhjbjjbljjhuyuhyulkjjkrr88mhjhhjhjhhhjy}
as:
\begin{multline}\label{GMT'3jGHKKkkhjjhgzzZZzzZZzzbvq88nkhhjggjgjkpkjljluytytuguutloklljjgjgjhjklljjjkjkhhiigjhjgjjljklkhhkghfhk;llkljklljjhjbjjbljjhuyuhyulkjjkrr88mhjhhjhjhhhjy1}
\frac{1}{\e^N}\mathcal{L}^{2N}\Bigg(\Bigg\{(x,y)\in
\Omega\times\Omega\;:\;
G\bigg(\frac{u(y)-u(x)}{|y-x|},u(y),u(x),y,x\bigg)\,\frac{1}{|y-x|^{N}}>
\frac{1}{\e^N}\Bigg\}\Bigg)=\\ \int\limits_{\Omega\cap
B_M(0)}\mathcal{L}^{N}\Bigg\{ z\in B_M(0)\,:\,
\frac{\chi_{\Omega}(x+\e z)}{|z|^{N}} G\bigg(\frac{u(x+\e
z)-u(x)}{|\e z|},u(x+\e z),u(x),x+\e z,x\bigg)>1\Bigg\}dx.
\end{multline}
However, since $G$ is continuous and since $u$ is a Lipschitz
function, by Rademacher's Theoreom for a.e. $x\in\R^N$ we have
\begin{multline}\label{GMT'3jGHKKkkhjjhgzzZZzzZZzzbvq88nkhhjggjgjkpkjljluytytuguutloklljjgjgjhjklljjjkjkhhiigjhjgjjljklkhhkghfhk;llkljklljjhjbjjbljjhuyuhyulkjjkrr88mhjhhjhjhhhjyjhjhhjhhjhj}
\lim\limits_{\e\to 0^+}G\bigg(\frac{u(x+\e z)-u(x)}{|\e z|},u(x+\e
z),u(x),x+\e z,x\bigg)\\=G\bigg(\nabla u(x)\cdot\frac{z}{|
z|},u(x),u(x),x,x\bigg)\quad\forall z\in\R^N\setminus\{0\}\,.
\end{multline}
Therefore, since the pointwise convergence implies the convergence
by a measure, we deduce from
\er{GMT'3jGHKKkkhjjhgzzZZzzZZzzbvq88nkhhjggjgjkpkjljluytytuguutloklljjgjgjhjklljjjkjkhhiigjhjgjjljklkhhkghfhk;llkljklljjhjbjjbljjhuyuhyulkjjkrr88mhjhhjhjhhhjyjhjhhjhhjhj}
that
\begin{multline}\label{GMT'3jGHKKkkhjjhgzzZZzzZZzzbvq88nkhhjggjgjkpkjljluytytuguutloklljjgjgjhjklljjjkjkhhiigjhjgjjljklkhhkghfhk;llkljklljjhjbjjbljjhuyuhyulkjjkrr88mhjhhjhjhhhjyjkhjjjy}
\lim\limits_{\e\to 0^+}\mathcal{L}^{N}\Bigg(\Bigg\{z\in B_M(0)\,:\,
\frac{\chi_{\Omega}(x+\e z)}{|z|^{N}} G\bigg(\frac{u(x+\e
z)-u(x)}{|\e z|},u(x+\e z),u(x),x+\e z,x\bigg)> 1\Bigg\}\Bigg)\\=
\mathcal{L}^{N}\Bigg(\Bigg\{z\in B_M(0)\,:\,
\frac{\chi_{\Omega}(x)}{|z|^{N}} G\bigg(\nabla u(x)\cdot\frac{z}{|
z|},u(x),u(x),x,x\bigg)> 1\Bigg\}\Bigg)
\\=
\mathcal{L}^{N}\Bigg(\Bigg\{z\in \R^N\,:\,
\frac{\chi_{\Omega}(x)}{|z|^{N}} G\bigg(\nabla u(x)\cdot\frac{z}{|
z|},u(x),u(x),x,x\bigg)> 1\Bigg\}\Bigg) \quad\text{for
a.e.}\;\;x\,\in\R^N\,.
\end{multline}
Therefore, by a dominated convergence Theorem we deduce from
\er{GMT'3jGHKKkkhjjhgzzZZzzZZzzbvq88nkhhjggjgjkpkjljluytytuguutloklljjgjgjhjklljjjkjkhhiigjhjgjjljklkhhkghfhk;llkljklljjhjbjjbljjhuyuhyulkjjkrr88mhjhhjhjhhhjy1}
and
\er{GMT'3jGHKKkkhjjhgzzZZzzZZzzbvq88nkhhjggjgjkpkjljluytytuguutloklljjgjgjhjklljjjkjkhhiigjhjgjjljklkhhkghfhk;llkljklljjhjbjjbljjhuyuhyulkjjkrr88mhjhhjhjhhhjyjkhjjjy}:
\begin{multline}\label{GMT'3jGHKKkkhjjhgzzZZzzZZzzbvq88nkhhjggjgjkpkjljluytytuguutloklljjgjgjhjklljjjkjkhhiigjhjgjjljklkhhkghfhk;llkljklljjhjbjjbljjhuyuhyulkjjkrr88mhjhhjhjhhhjy1iyyu}
\lim\limits_{s\to +\infty}s\mathcal{L}^{2N}\Bigg(\Bigg\{(x,y)\in
\Omega\times\Omega\;:\;
G\bigg(\frac{u(y)-u(x)}{|y-x|},u(y),u(x),y,x\bigg)\,\frac{1}{|y-x|^{N}}>
s\Bigg\}\Bigg)=\\ \lim\limits_{\e\to
0^+}\frac{1}{\e^N}\mathcal{L}^{2N}\Bigg(\Bigg\{(x,y)\in
\Omega\times\Omega\;:\;
G\bigg(\frac{u(y)-u(x)}{|y-x|},u(y),u(x),y,x\bigg)\,\frac{1}{|y-x|^{N}}>
\frac{1}{\e^N}\Bigg\}\Bigg)=\\ \int\limits_{\Omega\cap
B_M(0)}\mathcal{L}^{N}\Bigg(\Bigg\{z\in \R^N\,:\,
|z|^{N}<G\bigg(\nabla u(x)\cdot\frac{z}{| z|},u(x),u(x),x,x\bigg)
\Bigg\}\Bigg)dx\\
=\int\limits_{\Omega}\mathcal{L}^{N}\Bigg(\Bigg\{z\in \R^N\,:\,
|z|^{N}<G\bigg(\nabla u(x)\cdot\frac{z}{| z|},u(x),u(x),x,x\bigg)
\Bigg\}\Bigg)dx.
\end{multline}
However, in the case
$G\big(a,b,c,y,x\big):=G_0\big(|a|,b,c,y,x\big)$ with
$G_0:\R\times\R^m\times\R^m\times\R^N\times\R^N\to[0,+\infty)$, for
every $x\in\R^N$ we have
\begin{multline}\label{GMT'3jGHKKkkhjjhgzzZZzzZZzzbvq88nkhhjggjgjkpkjljluytytuguutloklljjgjgjhjklljjjkjkhhiigjhjgjjljklkhhkghfhk;llkljklljjhjbjjbljjhuyuhyulkjjkrr88mhjhhjhjhhhjy1iyyuhjggjjgghhfhhhgj}
\mathcal{L}^{N}\Bigg(\Bigg\{z\in \R^N\,:\, |z|^{N}<G\bigg(\nabla
u(x)\cdot\frac{z}{| z|},u(x),u(x),x,x\bigg) \Bigg\}\Bigg)=\\
\mathcal{L}^{N}\Bigg(\Bigg\{z\in \R^N\,:\,
|z|^{N}<G_0\bigg(\big|\nabla u(x)\big|\frac{|z_1|}{|
z|},u(x),u(x),x,x\bigg) \Bigg\}\Bigg)=
\\
\int_{S^{N-1}}\int\limits_{\Big\{t\in (0,+\infty)\;:\;
t^{N}<G_0\big(|\nabla u(x)||z_1|,u(x),u(x),x,x\big)
\Big\}}t^{N-1}dtd\mathcal{H}^{N-1}(z)\\=
\frac{1}{N}\int_{S^{N-1}}G_0\bigg(\big|\nabla
u(x)\big||z_1|,u(x),u(x),x,x\bigg)d\mathcal{H}^{N-1}(z).
\end{multline}
Thus inserting
\er{GMT'3jGHKKkkhjjhgzzZZzzZZzzbvq88nkhhjggjgjkpkjljluytytuguutloklljjgjgjhjklljjjkjkhhiigjhjgjjljklkhhkghfhk;llkljklljjhjbjjbljjhuyuhyulkjjkrr88mhjhhjhjhhhjy1iyyuhjggjjgghhfhhhgj}
into
\er{GMT'3jGHKKkkhjjhgzzZZzzZZzzbvq88nkhhjggjgjkpkjljluytytuguutloklljjgjgjhjklljjjkjkhhiigjhjgjjljklkhhkghfhk;llkljklljjhjbjjbljjhuyuhyulkjjkrr88mhjhhjhjhhhjy1iyyu}
finally gives
\er{GMT'3jGHKKkkhjjhgzzZZzzZZzzbvq88nkhhjggjgjkpkjljluytytuguutloklljjgjgjhjklljjjkjkhhiigjhjgjjljklkhhkghfhk;llkljklljjhjbjjbljjhuyuhyulkjjkrr88mhjhhjhjhhhjy1iyyuliyhigugu}.
\end{proof}

\begin{proof}
[Proof of Theorem
\ref{hjkgjkfhjffgggvggoopikhhhkjhhkjhjgjhhjgggghhdf11}]
Let
$\big\{u_n\big\}_{n=1}^{+\infty}\subset C^\infty_c(\R^N,\R^m)$ be
such that
\begin{equation}\label{GMT'3jGHKKkkhjjhgzzZZzzZZzzbvq88nkhhjggjgjkpkjljluytytuguutloklljjgjgjhj188hkjjggjhhjjh}
\lim\limits_{n\to+\infty}\int_{\R^N}\bigg(\Big|\nabla u_n(x)-\nabla
u(x)\Big|^q+\big| u_n(x)- u(x)\big|^q\bigg)\,=\,0.
\end{equation}
Then,
by the particular case of Proposition
\ref{fhfhfhfffhkjkj} with $G_0(a,b,c,y,x)=F(\sigma a,y,x)$, for
every fixed $n\in\mathbb{N}$ and every $\sigma>0$ we have
\begin{multline}\label{GMT'3jGHKKkkhjjhgzzZZzzZZzzbvq88nkhhjggjgjkpkjljluytytuguutloklljjgjgjhjklljjjkjkhjkkhkhhkhhjlkkhkjljljkjlkkhkllhjhjhhfyfppiooiououiuiuiuhjhjkhkhkjkhhkkhjjyhjggjuyyjghfhjhjghhh}
\lim\limits_{s\to +\infty}s\mathcal{L}^{2N}\Bigg(\Bigg\{(x,y)\in
\Omega\times\Omega\;:\;
F\bigg(\sigma\,\frac{\big|u_n(y)-u_n(x)\big|}{|y-x|},y,x\bigg)\,\frac{1}{|y-x|^{N}}>
s\Bigg\}\Bigg)\\
=\frac{1}{N}\int\limits_{\Omega}\Bigg(\int\limits_{S^{N-1}}F\bigg(\sigma\,\big|\nabla
u_n(x)\big||z_1|,x,x\bigg)d\mathcal{H}^{N-1}(z)\Bigg)dx\,.
\end{multline}
On the other hand, since $F(a,y,x)$ is non-decreasing on
$a\in[0,+\infty)$, given arbitrary $v\in W^{1,q}(\R^N,\R^m)$, $w\in
W^{1,q}(\R^N,\R^m)$, $y\neq x\in\R^N$, $s>0$ and $\alpha>1$, by
triangle inequality we obviously have
\begin{multline}\label{GMT'3jGHKKkkhjjhgzzZZzzZZzzbvq88nkhhjggjgjkpkjljluytytuguutloklljjgjgjhjklljjjkjkhjkkhkhhkhhjlkkhkjljljkjlkkhkllhjhjhhfyfppiooiououiuiuiuhjhjkhkhkjkhhkkhjjyhjggjuyyjghfhjhjghhhhjgjghkhjh}
F\bigg(\frac{\Big|\big(v(y)+w(y)\big)-\big(v(x)+w(x)\big)\Big|}{|y-x|},y,x\bigg)\leq
F\bigg(\frac{\big|v(y)-v(x)\big|}{|y-x|}+\frac{
\big|w(y)-w(x)\big|}{|y-x|},y,x\bigg)\\
=F\Bigg(\frac{1}{\alpha}\bigg(\alpha\frac{\big|v(y)-v(x)\big|}{|y-x|}\bigg)+\frac{\alpha-1}{\alpha}\bigg(\frac{\alpha}{\alpha-1}\frac{
\big|w(y)-w(x)\big|}{|y-x|}\bigg),y,x\Bigg)\\
\leq
F\Bigg(\max\bigg\{\alpha\frac{\big|v(y)-v(x)\big|}{|y-x|},\frac{\alpha}{\alpha-1}\frac{
\big|w(y)-w(x)\big|}{|y-x|}\bigg\},y,x\Bigg)\\=
\max\Bigg\{F\bigg(\alpha\frac{\big|v(y)-v(x)\big|}{|y-x|},y,x\bigg),F\bigg(\frac{\alpha}{\alpha-1}\frac{
\big|w(y)-w(x)\big|}{|y-x|},y,x\bigg)\Bigg\}\,.
\end{multline}
Therefore, by
\er{GMT'3jGHKKkkhjjhgzzZZzzZZzzbvq88nkhhjggjgjkpkjljluytytuguutloklljjgjgjhjklljjjkjkhjkkhkhhkhhjlkkhkjljljkjlkkhkllhjhjhhfyfppiooiououiuiuiuhjhjkhkhkjkhhkkhjjyhjggjuyyjghfhjhjghhhhjgjghkhjh}
we have
\begin{multline}\label{GMT'3jGHKKkkhjjhgzzZZzzZZzzbvq88nkhhjggjgjkpkjljluytytuguutloklljjgjgjhjklljjjkjkhjkkhkhhkhhjlkkhkjljljkjlkkhkllhjhjhhfyfppiooiououiuiuiuhjhjkhkhkjkhhkkhjjyhjggjuyyjghfhjhjghhhhjgjghkhjhhgjg}
F\bigg(\frac{\Big|\big(v(y)+w(y)\big)-\big(v(x)+w(x)\big)\Big|}{|y-x|},y,x\bigg)\,\frac{1}{|y-x|^{N}}>s\quad\quad\text{implies}\quad\quad\\
\quad\quad\text{either}\quad\quad
F\bigg(\alpha\frac{\big|v(y)-v(x)\big|}{|y-x|},y,x\bigg)\,\frac{1}{|y-x|^{N}}>s\\
\quad\quad\text{or}\quad\quad F\bigg(\frac{\alpha}{\alpha-1}\frac{
\big|w(y)-w(x)\big|}{|y-x|},y,x\bigg)\,\frac{1}{|y-x|^{N}}>s\,.
\end{multline}
Thus, denoting the sets:
\begin{align*}
A:=\Bigg\{(x,y)\in \Omega\times \Omega\;:\;F\bigg(\frac{\Big|\big(v(y)+w(y)\big)-\big(v(x)+w(x)\big)\Big|}{|y-x|},y,x\bigg)\,\frac{1}{|y-x|^{N}}>s\Bigg\}\\
B_1:=\Bigg\{(x,y)\in \Omega\times \Omega\;:\;F\bigg(\alpha\frac{\big|v(y)-v(x)\big|}{|y-x|},y,x\bigg)\,\frac{1}{|y-x|^{N}}>s\Bigg\}\\
B_2:= \Bigg\{(x,y)\in \Omega\times
\Omega\;:\;F\bigg(\frac{\alpha}{\alpha-1}\frac{
\big|w(y)-w(x)\big|}{|y-x|},y,x\bigg)\,\frac{1}{|y-x|^{N}}>s\Bigg\}\,,
\end{align*}
by
\er{GMT'3jGHKKkkhjjhgzzZZzzZZzzbvq88nkhhjggjgjkpkjljluytytuguutloklljjgjgjhjklljjjkjkhjkkhkhhkhhjlkkhkjljljkjlkkhkllhjhjhhfyfppiooiououiuiuiuhjhjkhkhkjkhhkkhjjyhjggjuyyjghfhjhjghhhhjgjghkhjhhgjg}
we obviously deduce:
\begin{equation}\label{GMT'3jGHKKkkhjjhgzzZZzzZZzzbvq88nkhhjggjgjkpkjljluytytuguutloklljjgjgjhj188hkjjggjhhjjhkojojjk}
A\subset B_1\cup B_2\,,
\end{equation}
and so,
\begin{equation}\label{GMT'3jGHKKkkhjjhgzzZZzzZZzzbvq88nkhhjggjgjkpkjljluytytuguutloklljjgjgjhj188hkjjggjhhjjhkojojjkhkhh}
\mathcal{L}^{2N}(A)\leq\mathcal{L}^{2N}(B_1)+\mathcal{L}^{2N}(B_2)\,,
\end{equation}
i.e. for every $v\in W^{1,q}(\R^N,\R^m)$ and $w\in
W^{1,q}(\R^N,\R^m)$, every $s>0$ and every $\alpha>1$ we have
\begin{multline}\label{GMT'3jGHKKkkhjjhgzzZZzzZZzzbvq88nkhhjggjgjkpkjljluytytuguutloklljjgjgjhjklljjjkjkhjkkhkhhkhhjlkkhkjljljkjlkkhkllhjhjhhfyfppiooiououiuiuiuhjhjkhkhkjkhhkkhjjyhjggjuyyjghfhjhjghhhjkhhkkjhkh}
\mathcal{L}^{2N}\Bigg(\Bigg\{(x,y)\in \Omega\times
\Omega\;:\;F\bigg(\frac{\Big|\big(v(y)+w(y)\big)-\big(v(x)+w(x)\big)\Big|}{|y-x|},y,x\bigg)\,\frac{1}{|y-x|^{N}}>s\Bigg\}\Bigg)\leq\\
\mathcal{L}^{2N}\Bigg\{(x,y)\in \Omega\times \Omega\;:\;F\bigg(\alpha\frac{\big|v(y)-v(x)\big|}{|y-x|},y,x\bigg)\,\frac{1}{|y-x|^{N}}>s\Bigg\}\Bigg)+\\
\mathcal{L}^{2N}\Bigg(\Bigg\{(x,y)\in \Omega\times
\Omega\;:\;F\bigg(\frac{\alpha}{\alpha-1}\frac{
\big|w(y)-w(x)\big|}{|y-x|},y,x\bigg)\,\frac{1}{|y-x|^{N}}>s\Bigg\}\Bigg)\,.
\end{multline}
Then, by
\er{GMT'3jGHKKkkhjjhgzzZZzzZZzzbvq88nkhhjggjgjkpkjljluytytuguutloklljjgjgjhjklljjjkjkhjkkhkhhkhhjlkkhkjljljkjlkkhkllhjhjhhfyfppiooiououiuiuiuhjhjkhkhkjkhhkkhjjyhjggjuyyjghfhjhjghhhjkhhkkjhkh}
we deduce that for every $v\in W^{1,q}(\R^N,\R^m)$, $w\in
W^{1,q}(\R^N,\R^m)$ and every $\alpha>1$ we have
\begin{multline}\label{GMT'3jGHKKkkhjjhgzzZZzzZZzzbvq88nkhhjggjgjkpkjljluytytuguutloklljjgjgjhjklljjjkjkhjkkhkhhkhhjlkkhkjljljkjlkkhkllhjhjhhfyfppiooiououiuiuiuhjhjkhkhkjkhhkkhjjyhjggjuyyjghfhjhjghhhjull33}
\limsup\limits_{s\to+\infty}s\mathcal{L}^{2N}\Bigg(\Bigg\{(x,y)\in \Omega\times \Omega\;:\;F\bigg(\frac{\big|v(y)-v(x)\big|}{|y-x|},y,x\bigg)\,\frac{1}{|y-x|^{N}}>s\Bigg\}\Bigg)\leq\\
\limsup\limits_{s\to+\infty}s\mathcal{L}^{2N}\Bigg(\Bigg\{(x,y)\in \Omega\times \Omega\;:\;F\bigg(\alpha\frac{\big|w(y)-w(x)\big|}{|y-x|},y,x\bigg)\,\frac{1}{|y-x|^{N}}>s\Bigg\}\Bigg)+\\
\limsup\limits_{s\to+\infty}s\,\mathcal{L}^{2N}\Bigg(\Bigg\{(x,y)\in
\Omega\times \Omega\;:\;\\
F\bigg(\frac{\alpha}{\alpha-1}\frac{\Big|\big(v(y)-w(y)\big)-\big(v(x)-w(x)\big)\Big|}{|y-x|},y,x\bigg)\,\frac{1}{|y-x|^{N}}>s\Bigg\}\Bigg).
\end{multline}
and
\begin{multline}\label{GMT'3jGHKKkkhjjhgzzZZzzZZzzbvq88nkhhjggjgjkpkjljluytytuguutloklljjgjgjhjklljjjkjkhjkkhkhhkhhjlkkhkjljljkjlkkhkllhjhjhhfyfppiooiououiuiuiuhjhjkhkhkjkhhkkhjjyhjggjuyyjghfhjhjghhhjull33mod}
\liminf\limits_{s\to+\infty}s\mathcal{L}^{2N}\Bigg(\Bigg\{(x,y)\in \Omega\times \Omega\;:\;F\bigg(\frac{\big|v(y)-v(x)\big|}{|y-x|},y,x\bigg)\,\frac{1}{|y-x|^{N}}>s\Bigg\}\Bigg)\leq\\
\liminf\limits_{s\to+\infty}s\mathcal{L}^{2N}\Bigg(\Bigg\{(x,y)\in \Omega\times \Omega\;:\;F\bigg(\alpha\frac{\big|w(y)-w(x)\big|}{|y-x|},y,x\bigg)\,\frac{1}{|y-x|^{N}}>s\Bigg\}\Bigg)+\\
\limsup\limits_{s\to+\infty}s\,\mathcal{L}^{2N}\Bigg(\Bigg\{(x,y)\in
\Omega\times \Omega\;:\;\\
F\bigg(\frac{\alpha}{\alpha-1}\frac{\Big|\big(v(y)-w(y)\big)-\big(v(x)-w(x)\big)\Big|}{|y-x|},y,x\bigg)\,\frac{1}{|y-x|^{N}}>s\Bigg\}\Bigg).
\end{multline}
Therefore, since $F(a,y,x)\leq C|a|^q$ for every $a\in\R$ and every
$x,y\in\R^N$, by
\er{GMT'3jGHKKkkhjjhgzzZZzzZZzzbvq88nkhhjggjgjkpkjljluytytuguutloklljjgjgjhjklljjjkjkhjkkhkhhkhhjlkkhkjljljkjlkkhkllhjhjhhfyfppiooiououiuiuiuhjhjkhkhkjkhhkkhjjyhjggjuyyjghfhjhjghhhjull33},
for every $v\in W^{1,q}(\R^N,\R^m)$, $w\in W^{1,q}(\R^N,\R^m)$ and
every $\alpha>1$ we deduce
\begin{multline}\label{GMT'3jGHKKkkhjjhgzzZZzzZZzzbvq88nkhhjggjgjkpkjljluytytuguutloklljjgjgjhjklljjjkjkhjkkhkhhkhhjlkkhkjljljkjlkkhkllhjhjhhfyfppiooiououiuiuiuhjhjkhkhkjkhhkkhjjyhjggjuyyjghfhjhjghhhjullyy}
\limsup\limits_{s\to+\infty}s\mathcal{L}^{2N}\Bigg(\Bigg\{(x,y)\in \Omega\times \Omega\;:\;F\bigg(\frac{\big|v(y)-v(x)\big|}{|y-x|},y,x\bigg)\,\frac{1}{|y-x|^{N}}>s\Bigg\}\Bigg)\leq\\
\limsup\limits_{s\to+\infty}s\mathcal{L}^{2N}\Bigg(\Bigg\{(x,y)\in \Omega\times \Omega\;:\;F\bigg(\alpha\frac{\big|w(y)-w(x)\big|}{|y-x|},y,x\bigg)\,\frac{1}{|y-x|^{N}}>s\Bigg\}\Bigg)+\\
\limsup\limits_{s\to+\infty}s\,\mathcal{L}^{2N}\Bigg(\Bigg\{(x,y)\in
\Omega\times
\Omega\;:\;C\frac{\alpha^q}{(\alpha-1)^q}\frac{\Big|\big(v(y)-w(y)\big)-\big(v(x)-w(x)\big)\Big|^q}{|y-x|^{q+N}}>s\Bigg\}\Bigg)
\\=\limsup\limits_{s\to+\infty}s\mathcal{L}^{2N}\Bigg(\Bigg\{(x,y)\in \Omega\times \Omega\;:\;F\bigg(\alpha\frac{\big|w(y)-w(x)\big|}{|y-x|},y,x\bigg)\,\frac{1}{|y-x|^{N}}>s\Bigg\}\Bigg)+\\C\frac{\alpha^q}{(\alpha-1)^q}\limsup\limits_{s\to+\infty}s\,\mathcal{L}^{2N}\Bigg(\Bigg\{(x,y)\in
\Omega\times
\Omega\;:\;\frac{\Big|\big(v(y)-w(y)\big)-\big(v(x)-w(x)\big)\Big|^q}{|y-x|^{q+N}}>s\Bigg\}\Bigg),
\end{multline}
and similarly by
\er{GMT'3jGHKKkkhjjhgzzZZzzZZzzbvq88nkhhjggjgjkpkjljluytytuguutloklljjgjgjhjklljjjkjkhjkkhkhhkhhjlkkhkjljljkjlkkhkllhjhjhhfyfppiooiououiuiuiuhjhjkhkhkjkhhkkhjjyhjggjuyyjghfhjhjghhhjull33mod}
we obtain
\begin{multline}\label{GMT'3jGHKKkkhjjhgzzZZzzZZzzbvq88nkhhjggjgjkpkjljluytytuguutloklljjgjgjhjklljjjkjkhjkkhkhhkhhjlkkhkjljljkjlkkhkllhjhjhhfyfppiooiououiuiuiuhjhjkhkhkjkhhkkhjjyhjggjuyyjghfhjhjghhhjullyymod}
\liminf\limits_{s\to+\infty}s\mathcal{L}^{2N}\Bigg(\Bigg\{(x,y)\in \Omega\times \Omega\;:\;F\bigg(\frac{\big|v(y)-v(x)\big|}{|y-x|},y,x\bigg)\,\frac{1}{|y-x|^{N}}>s\Bigg\}\Bigg)\leq\\
\liminf\limits_{s\to+\infty}s\mathcal{L}^{2N}\Bigg(\Bigg\{(x,y)\in
\Omega\times
\Omega\;:\;F\bigg(\alpha\frac{\big|w(y)-w(x)\big|}{|y-x|},y,x\bigg)\,\frac{1}{|y-x|^{N}}>s\Bigg\}\Bigg)+\\C\frac{\alpha^q}{(\alpha-1)^q}\limsup\limits_{s\to+\infty}s\,\mathcal{L}^{2N}\Bigg(\Bigg\{(x,y)\in
\Omega\times
\Omega\;:\;\frac{\Big|\big(v(y)-w(y)\big)-\big(v(x)-w(x)\big)\Big|^q}{|y-x|^{q+N}}>s\Bigg\}\Bigg),
\end{multline}
Therefore, using Theorem \ref{hjkgjkfhjffgggvggoopikhhhkjh}, by
\er{GMT'3jGHKKkkhjjhgzzZZzzZZzzbvq88nkhhjggjgjkpkjljluytytuguutloklljjgjgjhjklljjjkjkhjkkhkhhkhhjlkkhkjljljkjlkkhkllhjhjhhfyfppiooiououiuiuiuhjhjkhkhkjkhhkkhjjyhjggjuyyjghfhjhjghhhjullyy}
and
\er{GMT'3jGHKKkkhjjhgzzZZzzZZzzbvq88nkhhjggjgjkpkjljluytytuguutloklljjgjgjhjklljjjkjkhjkkhkhhkhhjlkkhkjljljkjlkkhkllhjhjhhfyfppiooiououiuiuiuhjhjkhkhkjkhhkkhjjyhjggjuyyjghfhjhjghhhjullyymod},
for every given  $v\in W^{1,q}(\R^N,\R^m)$, $w\in
W^{1,q}(\R^N,\R^m)$ and every $\alpha>1$ we infer
\begin{multline}\label{GMT'3jGHKKkkhjjhgzzZZzzZZzzbvq88nkhhjggjgjkpkjljluytytuguutloklljjgjgjhjklljjjkjkhjkkhkhhkhhjlkkhkjljljkjlkkhkllhjhjhhfyfppiooiououiuiuiuhjhjkhkhkjkhhkkhjjyhjggjuyyjghfhjhjghhhjull}
\limsup\limits_{s\to+\infty}s\mathcal{L}^{2N}\Bigg(\Bigg\{(x,y)\in \Omega\times \Omega\;:\;F\bigg(\frac{\big|v(y)-v(x)\big|}{|y-x|},y,x\bigg)\,\frac{1}{|y-x|^{N}}>s\Bigg\}\Bigg)\leq\\
\limsup\limits_{s\to+\infty}s\mathcal{L}^{2N}\Bigg(\Bigg\{(x,y)\in
\Omega\times
\Omega\;:\;F\bigg(\alpha\frac{\big|w(y)-w(x)\big|}{|y-x|},y,x\bigg)\,\frac{1}{|y-x|^{N}}>s\Bigg\}\Bigg)\\+
C\,\widetilde C_{N}\,\frac{\alpha^q}{(\alpha-1)^q}
\int_{\R^N}\Big|\nabla v(x)-\nabla w(x)\Big|^qdx,
\end{multline}
and
\begin{multline}\label{GMT'3jGHKKkkhjjhgzzZZzzZZzzbvq88nkhhjggjgjkpkjljluytytuguutloklljjgjgjhjklljjjkjkhjkkhkhhkhhjlkkhkjljljkjlkkhkllhjhjhhfyfppiooiououiuiuiuhjhjkhkhkjkhhkkhjjyhjggjuyyjghfhjhjghhhjullmod}
\liminf\limits_{s\to+\infty}s\mathcal{L}^{2N}\Bigg(\Bigg\{(x,y)\in \Omega\times \Omega\;:\;F\bigg(\frac{\big|v(y)-v(x)\big|}{|y-x|},y,x\bigg)\,\frac{1}{|y-x|^{N}}>s\Bigg\}\Bigg)\leq\\
\liminf\limits_{s\to+\infty}s\mathcal{L}^{2N}\Bigg(\Bigg\{(x,y)\in
\Omega\times
\Omega\;:\;F\bigg(\alpha\frac{\big|w(y)-w(x)\big|}{|y-x|},y,x\bigg)\,\frac{1}{|y-x|^{N}}>s\Bigg\}\Bigg)\\+
C\,\widetilde C_{N}\,\frac{\alpha^q}{(\alpha-1)^q}
\int_{\R^N}\Big|\nabla v(x)-\nabla w(x)\Big|^qdx.
\end{multline}
In particular, taking firstly $v=u$ and $w=u_n$ in
\er{GMT'3jGHKKkkhjjhgzzZZzzZZzzbvq88nkhhjggjgjkpkjljluytytuguutloklljjgjgjhjklljjjkjkhjkkhkhhkhhjlkkhkjljljkjlkkhkllhjhjhhfyfppiooiououiuiuiuhjhjkhkhkjkhhkkhjjyhjggjuyyjghfhjhjghhhjull}
and secondly $v=u_n$ and $w=u$ in
\er{GMT'3jGHKKkkhjjhgzzZZzzZZzzbvq88nkhhjggjgjkpkjljluytytuguutloklljjgjgjhjklljjjkjkhjkkhkhhkhhjlkkhkjljljkjlkkhkllhjhjhhfyfppiooiououiuiuiuhjhjkhkhkjkhhkkhjjyhjggjuyyjghfhjhjghhhjullmod},
for every $\alpha>1$ we deduce:
\begin{multline}\label{GMT'3jGHKKkkhjjhgzzZZzzZZzzbvq88nkhhjggjgjkpkjljluytytuguutloklljjgjgjhjklljjjkjkhjkkhkhhkhhjlkkhkjljljkjlkkhkllhjhjhhfyfppiooiououiuiuiuhjhjkhkhkjkhhkkhjjyhjggjuyyjghfhjhjghhhjulld1}
\limsup\limits_{s\to+\infty}s\mathcal{L}^{2N}\Bigg(\Bigg\{(x,y)\in \Omega\times \Omega\;:\;F\bigg(\frac{\big|u(y)-u(x)\big|}{|y-x|},y,x\bigg)\,\frac{1}{|y-x|^{N}}>s\Bigg\}\Bigg)\leq\\
\limsup\limits_{s\to+\infty}s\mathcal{L}^{2N}\Bigg(\Bigg\{(x,y)\in
\Omega\times
\Omega\;:\;F\bigg(\alpha\frac{\big|u_n(y)-u_n(x)\big|}{|y-x|},y,x\bigg)\,\frac{1}{|y-x|^{N}}>s\Bigg\}\Bigg)\\+
C\,\widetilde C_{N}\,\frac{\alpha^q}{(\alpha-1)^q}
\int_{\R^N}\Big|\nabla u_n(x)-\nabla u(x)\Big|^qdx,
\end{multline}
and
\begin{multline}\label{GMT'3jGHKKkkhjjhgzzZZzzZZzzbvq88nkhhjggjgjkpkjljluytytuguutloklljjgjgjhjklljjjkjkhjkkhkhhkhhjlkkhkjljljkjlkkhkllhjhjhhfyfppiooiououiuiuiuhjhjkhkhkjkhhkkhjjyhjggjuyyjghfhjhjghhhjulld2}
\liminf\limits_{s\to+\infty}s\mathcal{L}^{2N}\Bigg(\Bigg\{(x,y)\in \Omega\times \Omega\;:\;F\bigg(\frac{\big|u_n(y)-u_n(x)\big|}{|y-x|},y,x\bigg)\,\frac{1}{|y-x|^{N}}>s\Bigg\}\Bigg)\leq\\
\liminf\limits_{s\to+\infty}s\mathcal{L}^{2N}\Bigg(\Bigg\{(x,y)\in
\Omega\times
\Omega\;:\;F\bigg(\alpha\frac{\big|u(y)-u(x)\big|}{|y-x|},y,x\bigg)\,\frac{1}{|y-x|^{N}}>s\Bigg\}\Bigg)\\+
C\,\widetilde C_{N}\,\frac{\alpha^q}{(\alpha-1)^q}
\int_{\R^N}\Big|\nabla u_n(x)-\nabla u(x)\Big|^qdx.
\end{multline}
Thus by combining
\er{GMT'3jGHKKkkhjjhgzzZZzzZZzzbvq88nkhhjggjgjkpkjljluytytuguutloklljjgjgjhjklljjjkjkhjkkhkhhkhhjlkkhkjljljkjlkkhkllhjhjhhfyfppiooiououiuiuiuhjhjkhkhkjkhhkkhjjyhjggjuyyjghfhjhjghhh}
with
\er{GMT'3jGHKKkkhjjhgzzZZzzZZzzbvq88nkhhjggjgjkpkjljluytytuguutloklljjgjgjhjklljjjkjkhjkkhkhhkhhjlkkhkjljljkjlkkhkllhjhjhhfyfppiooiououiuiuiuhjhjkhkhkjkhhkkhjjyhjggjuyyjghfhjhjghhhjulld1}
we deduce
\begin{multline}\label{GMT'3jGHKKkkhjjhgzzZZzzZZzzbvq88nkhhjggjgjkpkjljluytytuguutloklljjgjgjhjklljjjkjkhjkkhkhhkhhjlkkhkjljljkjlkkhkllhjhjhhfyfppiooiououiuiuiuhjhjkhkhkjkhhkkhjjyhjggjuyyjghfhjhjghhhhkhj13}
\limsup\limits_{s\to+\infty}s\mathcal{L}^{2N}\Bigg\{(x,y)\in \Omega\times \Omega\;:\;F\bigg(\frac{\big|u(y)-u(x)\big|}{|y-x|},y,x\bigg)\,\frac{1}{|y-x|^{N}}>s\Bigg\}\Bigg)\leq\\
\frac{1}{N}\int\limits_{\Omega}\Bigg(\int\limits_{S^{N-1}}F\bigg(\alpha\,\big|\nabla
u_n(x)\big||z_1|,x,x\bigg)d\mathcal{H}^{N-1}(z)\Bigg)dx+
C\,\widetilde C_{N}\,\frac{\alpha^q}{(\alpha-1)^q}
\int_{\R^N}\Big|\nabla u_n(x)-\nabla u(x)\Big|^qdx,
\end{multline}
and by combining
\er{GMT'3jGHKKkkhjjhgzzZZzzZZzzbvq88nkhhjggjgjkpkjljluytytuguutloklljjgjgjhjklljjjkjkhjkkhkhhkhhjlkkhkjljljkjlkkhkllhjhjhhfyfppiooiououiuiuiuhjhjkhkhkjkhhkkhjjyhjggjuyyjghfhjhjghhh}
with
\er{GMT'3jGHKKkkhjjhgzzZZzzZZzzbvq88nkhhjggjgjkpkjljluytytuguutloklljjgjgjhjklljjjkjkhjkkhkhhkhhjlkkhkjljljkjlkkhkllhjhjhhfyfppiooiououiuiuiuhjhjkhkhkjkhhkkhjjyhjggjuyyjghfhjhjghhhjulld2}
we deduce
\begin{multline}\label{GMT'3jGHKKkkhjjhgzzZZzzZZzzbvq88nkhhjggjgjkpkjljluytytuguutloklljjgjgjhjklljjjkjkhjkkhkhhkhhjlkkhkjljljkjlkkhkllhjhjhhfyfppiooiououiuiuiuhjhjkhkhkjkhhkkhjjyhjggjuyyjghfhjhjghhhhkhj24}
\frac{1}{N}\int\limits_{\Omega}\Bigg(\int\limits_{S^{N-1}}F\bigg(\big|\nabla
u_n(x)\big||z_1|,x,x\bigg)d\mathcal{H}^{N-1}(z)\Bigg)dx\leq
C\,\widetilde C_{N}\,\frac{\alpha^q}{(\alpha-1)^q}
\int_{\R^N}\Big|\nabla u_n(x)-\nabla u(x)\Big|^qdx\\+
\liminf\limits_{s\to+\infty}s\mathcal{L}^{2N}\Bigg(\Bigg\{(x,y)\in
\Omega\times
\Omega\;:\;F\bigg(\alpha\frac{\big|u(y)-u(x)\big|}{|y-x|},y,x\bigg)\,\frac{1}{|y-x|^{N}}>s\Bigg\}\Bigg).
\end{multline}
Therefore, letting $n\to+\infty$ in
\er{GMT'3jGHKKkkhjjhgzzZZzzZZzzbvq88nkhhjggjgjkpkjljluytytuguutloklljjgjgjhjklljjjkjkhjkkhkhhkhhjlkkhkjljljkjlkkhkllhjhjhhfyfppiooiououiuiuiuhjhjkhkhkjkhhkkhjjyhjggjuyyjghfhjhjghhhhkhj13}
and
\er{GMT'3jGHKKkkhjjhgzzZZzzZZzzbvq88nkhhjggjgjkpkjljluytytuguutloklljjgjgjhjklljjjkjkhjkkhkhhkhhjlkkhkjljljkjlkkhkllhjhjhhfyfppiooiououiuiuiuhjhjkhkhkjkhhkkhjjyhjggjuyyjghfhjhjghhhhkhj24}
and using
\er{GMT'3jGHKKkkhjjhgzzZZzzZZzzbvq88nkhhjggjgjkpkjljluytytuguutloklljjgjgjhj188hkjjggjhhjjh}
together with the Dominated Convergence Theorem, we deduce:
\begin{multline}\label{GMT'3jGHKKkkhjjhgzzZZzzZZzzbvq88nkhhjggjgjkpkjljluytytuguutloklljjgjgjhjklljjjkjkhjkkhkhhkhhjlkkhkjljljkjlkkhkllhjhjhhfyfppiooiououiuiuiuhjhjkhkhkjkhhkkhjjyhjggjuyyjghfhjhjghhhhkhj135}
\limsup\limits_{s\to+\infty}s\mathcal{L}^{2N}\Bigg\{(x,y)\in \Omega\times \Omega\;:\;F\bigg(\frac{\big|u(y)-u(x)\big|}{|y-x|},y,x\bigg)\,\frac{1}{|y-x|^{N}}>s\Bigg\}\Bigg)\leq\\
\frac{1}{N}\int\limits_{\Omega}\Bigg(\int\limits_{S^{N-1}}F\bigg(\alpha\,\big|\nabla
u(x)\big||z_1|,x,x\bigg)d\mathcal{H}^{N-1}(z)\Bigg)dx\,,
\end{multline}
and
\begin{multline}\label{GMT'3jGHKKkkhjjhgzzZZzzZZzzbvq88nkhhjggjgjkpkjljluytytuguutloklljjgjgjhjklljjjkjkhjkkhkhhkhhjlkkhkjljljkjlkkhkllhjhjhhfyfppiooiououiuiuiuhjhjkhkhkjkhhkkhjjyhjggjuyyjghfhjhjghhhhkhj247kk}
\frac{1}{N}\int\limits_{\Omega}\Bigg(\int\limits_{S^{N-1}}F\bigg(\big|\nabla
u(x)\big||z_1|,x,x\bigg)d\mathcal{H}^{N-1}(z)\Bigg)dx\leq
\\
\liminf\limits_{s\to+\infty}s\mathcal{L}^{2N}\Bigg(\Bigg\{(x,y)\in
\Omega\times
\Omega\;:\;F\bigg(\alpha\frac{\big|u(y)-u(x)\big|}{|y-x|},y,x\bigg)\,\frac{1}{|y-x|^{N}}>s\Bigg\}\Bigg)\,.
\end{multline}
In particular, taking
\er{GMT'3jGHKKkkhjjhgzzZZzzZZzzbvq88nkhhjggjgjkpkjljluytytuguutloklljjgjgjhjklljjjkjkhjkkhkhhkhhjlkkhkjljljkjlkkhkllhjhjhhfyfppiooiououiuiuiuhjhjkhkhkjkhhkkhjjyhjggjuyyjghfhjhjghhhhkhj247kk}
for $\frac{1}{\alpha}\,u$ instead of $u$ we deduce:
\begin{multline}\label{GMT'3jGHKKkkhjjhgzzZZzzZZzzbvq88nkhhjggjgjkpkjljluytytuguutloklljjgjgjhjklljjjkjkhjkkhkhhkhhjlkkhkjljljkjlkkhkllhjhjhhfyfppiooiououiuiuiuhjhjkhkhkjkhhkkhjjyhjggjuyyjghfhjhjghhhhkhj247}
\frac{1}{N}\int\limits_{\Omega}\Bigg(\int\limits_{S^{N-1}}F\bigg(\frac{1}{\alpha}\big|\nabla
u(x)\big||z_1|,x,x\bigg)d\mathcal{H}^{N-1}(z)\Bigg)dx\leq
\\
\liminf\limits_{s\to+\infty}s\mathcal{L}^{2N}\Bigg(\Bigg\{(x,y)\in
\Omega\times
\Omega\;:\;F\bigg(\frac{\big|u(y)-u(x)\big|}{|y-x|},y,x\bigg)\,\frac{1}{|y-x|^{N}}>s\Bigg\}\Bigg)\,.
\end{multline}
Finally, letting $\alpha\to 1^+$ in
\er{GMT'3jGHKKkkhjjhgzzZZzzZZzzbvq88nkhhjggjgjkpkjljluytytuguutloklljjgjgjhjklljjjkjkhjkkhkhhkhhjlkkhkjljljkjlkkhkllhjhjhhfyfppiooiououiuiuiuhjhjkhkhkjkhhkkhjjyhjggjuyyjghfhjhjghhhhkhj135}
and
\er{GMT'3jGHKKkkhjjhgzzZZzzZZzzbvq88nkhhjggjgjkpkjljluytytuguutloklljjgjgjhjklljjjkjkhjkkhkhhkhhjlkkhkjljljkjlkkhkllhjhjhhfyfppiooiououiuiuiuhjhjkhkhkjkhhkkhjjyhjggjuyyjghfhjhjghhhhkhj247}
and using again the Dominated Convergence Theorem, we infer
\begin{multline}\label{GMT'3jGHKKkkhjjhgzzZZzzZZzzbvq88nkhhjggjgjkpkjljluytytuguutloklljjgjgjhjklljjjkjkhjkkhkhhkhhjlkkhkjljljkjlkkhkllhjhjhhfyfppiooiououiuiuiuhjhjkhkhkjkhhkkhjjyhjggjuyyjghfhjhjghhhhkhjhjggjkgkf}
\frac{1}{N}\int\limits_{\Omega}\Bigg(\int\limits_{S^{N-1}}F\bigg(\big|\nabla
u(x)\big||z_1|,x,x\bigg)d\mathcal{H}^{N-1}(z)\Bigg)dx\leq\\
\liminf\limits_{s\to+\infty}s\mathcal{L}^{2N}\Bigg(\Bigg\{(x,y)\in \Omega\times \Omega\;:\;F\bigg(\frac{\big|u(y)-u(x)\big|}{|y-x|},y,x\bigg)\,\frac{1}{|y-x|^{N}}>s\Bigg\}\Bigg)\leq\\
\limsup\limits_{s\to+\infty}s\mathcal{L}^{2N}\Bigg(\Bigg\{(x,y)\in \Omega\times \Omega\;:\;F\bigg(\frac{\big|u(y)-u(x)\big|}{|y-x|},y,x\bigg)\,\frac{1}{|y-x|^{N}}>s\Bigg\}\Bigg)\\
\leq\frac{1}{N}\int\limits_{\Omega}\Bigg(\int\limits_{S^{N-1}}F\bigg(\big|\nabla
u(x)\big||z_1|,x,x\bigg)d\mathcal{H}^{N-1}(z)\Bigg)dx \,,
\end{multline}
and we obtain
\er{GMT'3jGHKKkkhjjhgzzZZzzZZzzbvq88nkhhjggjgjkpkjljluytytuguutloklljjgjgjhjklljjjkjkhjkkhkhhkhhjlkkhkjljljkjlkkhkllhjhjhhfyfppiooiououiuiuiuhjhjkhkhkjkhhkkhjjyhjggjuyyjghfhjhjghhhzz11}.
\end{proof}

\section{Proof of Theorem \ref{hjkgjkfhjffjhmgg7}}
The next Proposition is proved exactly as a similar statement
in \cite{jmp}; the proof is postponed to  the Appendix; in both cases the key
ingredient is Proposition \ref{hgugghghhffhfhKKzzbvq} which is part of \cite[Proposition 2.4]{jmp}).
\begin{proposition}\label{hgugghghhffhfhKKzzbvqhkjjgg}
Let $\Omega$ be an open set with bounded Lipschitz boundary, $q>1$
and $u\in BV(\Omega,\R^m)\cap L^\infty(\Omega,\R^m)$. Then,
\begin{multline}\label{gghgjhfgggjfgfhughGHGHKKzzjkjkyuyuybvqjhgfhfhgjgjjlhkluyikhhkhkhkjgjhhh}
\lim\limits_{\e\to
0^+}\Bigg\{\int_{S^{N-1}}\int_{\Omega}\frac{\big|u( x+\e\vec
n)-u(x)\big|^q}{\e}\chi_\Omega( x+\e\vec
n)dxd\mathcal{H}^{N-1}(\vec n)\Bigg\}=\\
\bigg(\int_{S^{N-1}}|z_1|d\mathcal{H}^{N-1}(z)\bigg)\Bigg(\int_{J_u\cap
\Omega}\Big|u^+(x)-u^-(x)\Big|^qd\mathcal{H}^{N-1}(x)\Bigg)\,.
\end{multline}
\end{proposition}
\begin{proof}[Proof of Theorem \ref{hjkgjkfhjffjhmgg7}]
The result is a direct consequence of Theorem \ref{hjkgjkfhjff} and
Proposition \ref{hgugghghhffhfhKKzzbvqhkjjgg}.
\end{proof}

\appendix
\section{Appendix}

\begin{lemma}\label{gjyfyfuyyfifgyify}
Let $\Omega\subset\R^N$ be a domain $q\geq 1$, $r\geq 0$ and $u\in
L^p(\Omega,\R^m)$. Next let
$\rho_\e\big(|z|\big):\R^N\to[0,+\infty)$ be radial mollifiers so
that $\int_{\R^N}\rho_\e\big(|z|\big)dz=1$ and for every $r>0$ there
exits $\delta:=\delta_r>0$, such that $\supp{(\rho_\e)}\subset
B_r(0)$ for every $\e\in(0,\delta_r)$. Then
\begin{multline}\label{GMT'3jGHKKkkhjjhgzzZZzzZZzzbvq88nkhhjggjgjkpkjljluytytl;klljkljojkojjo;k;kklklklkiljluikljjhkjh}
\frac{1}{\mathcal{H}^{N-1}(S^{N-1})}\,\liminf\limits_{\e\to
0^+}\Bigg(\int\limits_{S^{N-1}}\int\limits_{\Omega}\chi_{\Omega}(x+\e\vec
n)\frac{\big|u( x+\e\vec
n)-u(x)\big|^q}{\e^r}dxd\mathcal{H}^{N-1}(\vec n)\Bigg)\\
\leq\liminf\limits_{\e\to
0^+}\int\limits_{\Omega}\int\limits_{\Omega}\rho_\e\Big(|y-x|\Big)\frac{\big|u(
y)-u(x)\big|^q}{|y-x|^r}dydx \leq\limsup\limits_{\e\to
0^+}\int\limits_{\Omega}\int\limits_{\Omega}\rho_\e\Big(|y-x|\Big)\frac{\big|u(
y)-u(x)\big|^q}{|y-x|^r}dydx
\\
\leq\frac{1}{\mathcal{H}^{N-1}(S^{N-1})}\,\limsup\limits_{\e\to
0^+}\Bigg(\int\limits_{S^{N-1}}\int\limits_{\Omega}\chi_{\Omega}(x+\e\vec
n)\frac{\big|u( x+\e\vec
n)-u(x)\big|^q}{\e^r}dxd\mathcal{H}^{N-1}(\vec n)\Bigg).
\end{multline}
\end{lemma}
\begin{proof}
Obviously, we have
\begin{multline}\label{GMT'3jGHKKkkhjjhgzzZZzzZZzzbvq88nkhhjggjgjkpkjljluytytl;klljkljojkojjoikokhh11}
\int\limits_{\Omega}\int\limits_{\Omega}\rho_\e\Big(|y-x|\Big)\frac{\big|u(
y)-u(x)\big|^q}{|y-x|^r}dydx=\int\limits_{\mathbb{R}^N}\int\limits_{\Omega}\rho_\e\Big(|z|\Big)\chi_{\Omega}(x+z)\frac{\big|u(
x+z)-u(x)\big|^q}{|z|^r}dxdz\\=
\int\limits_{S^{N-1}}\int\limits_{\mathbb{R}^+}\int\limits_{\Omega}t^{N-1}\rho_\e\big(t\big)\chi_{\Omega}(x+t\vec
n)\frac{\big|u( x+t\vec
n)-u(x)\big|^q}{t^r}dxdtd\mathcal{H}^{N-1}(\vec n)
\\=
\int\limits_{0}^{+\infty}\Bigg\{t^{N-1}\rho_\e\big(t\big)\bigg(\int\limits_{S^{N-1}}\int\limits_{\Omega}\chi_{\Omega}(x+t\vec
n)\frac{\big|u( x+t\vec
n)-u(x)\big|^q}{t^r}dxd\mathcal{H}^{N-1}(\vec n)\bigg)\Bigg\}dt.
\end{multline}
Therefore, if $\e\in(0,\delta_r)$ by
\er{GMT'3jGHKKkkhjjhgzzZZzzZZzzbvq88nkhhjggjgjkpkjljluytytl;klljkljojkojjoikokhh11}
we deduce
\begin{multline}\label{GMT'3jGHKKkkhjjhgzzZZzzZZzzbvq88nkhhjggjgjkpkjljluytytl;klljkljojkojjoikokhhkhgjgj}
\frac{1}{\mathcal{H}^{N-1}(S^{N-1})}\,\inf\limits_{t\in(0,r)}\Bigg\{\int\limits_{S^{N-1}}\int\limits_{\Omega}\chi_{\Omega}(x+t\vec
n)\frac{\big|u( x+t\vec
n)-u(x)\big|^q}{t^r}dxd\mathcal{H}^{N-1}(\vec
n)\Bigg\}\\
=\bigg(\int\limits_{0}^{+\infty}t^{N-1}\rho_\e\big(t\big)dt\bigg)\inf\limits_{t\in(0,r)}\Bigg\{\int\limits_{S^{N-1}}\int\limits_{\Omega}\chi_{\Omega}(x+t\vec
n)\frac{\big|u( x+t\vec
n)-u(x)\big|^q}{t^r}dxd\mathcal{H}^{N-1}(\vec n)\Bigg\}\\ \leq
\int\limits_{\Omega}\int\limits_{\Omega}\rho_\e\Big(|y-x|\Big)\frac{\big|u(
y)-u(x)\big|^q}{|y-x|^r}dydx\\
\leq
\bigg(\int\limits_{0}^{+\infty}t^{N-1}\rho_\e\big(t\big)dt\bigg)\sup\limits_{t\in(0,r)}\Bigg\{\int\limits_{S^{N-1}}\int\limits_{\Omega}\chi_{\Omega}(x+t\vec
n)\frac{\big|u( x+t\vec
n)-u(x)\big|^q}{t^r}dxd\mathcal{H}^{N-1}(\vec n)\Bigg\}\\=
\frac{1}{\mathcal{H}^{N-1}(S^{N-1})}\,\sup\limits_{t\in(0,r)}\Bigg\{\int\limits_{S^{N-1}}\int\limits_{\Omega}\chi_{\Omega}(x+t\vec
n)\frac{\big|u( x+t\vec
n)-u(x)\big|^q}{t^r}dxd\mathcal{H}^{N-1}(\vec n)\Bigg\}\,.
\end{multline}
Thus, letting $r\to 0^+$ in
\er{GMT'3jGHKKkkhjjhgzzZZzzZZzzbvq88nkhhjggjgjkpkjljluytytl;klljkljojkojjoikokhhkhgjgj}
we easily deduce
\er{GMT'3jGHKKkkhjjhgzzZZzzZZzzbvq88nkhhjggjgjkpkjljluytytl;klljkljojkojjo;k;kklklklkiljluikljjhkjh}.
\end{proof}

\subsection{The case r=q}

\begin{lemma}\label{gughfgfhfgdgddffddfKKzzbvq}
Let $\Omega\subset\R^N$ be an open set, $q\geq 1$, $u\in
L^q(\Omega,\R^m)$ and $t_1,t_2>0$. Furthermore, let $G\subset\Omega$
be an open subset such that either $G$ is convex or
$t_1<\dist(G,\R^N\setminus\Omega)$. Then, we have
\begin{multline}
\label{gghgjhfgggjfgfhughGHGHGHKKjhggjhggjjhjgghjhhjhkoioji}
\int\limits_{S^{N-1}}\int\limits_{G}\frac{1}{(t_1+t_2)^q}\Big|u\big(x+(t_1+t_2)\vec
n\big)-u(x)\Big|^q\chi_G\big(x+(t_1+t_2)\vec n\big)dxd\mathcal{H}^{N-1}(\vec n) \leq\\
\frac{t_2}{(t_1+t_2)}\int\limits_{S^{N-1}}\int\limits_{\Omega}\frac{\Big|u\big(x+t_2\vec n\big)-u\big(x\big)\Big|^q}{t_2^q}\chi_\Omega\big(x+t_2\vec n\big)dxd\mathcal{H}^{N-1}(\vec n)\\
+
\frac{t_1}{(t_1+t_2)}\int\limits_{S^{N-1}}\int\limits_{\Omega}\frac{\Big|u(x+t_1\vec
n)-u(x)\Big|^q}{t_1^q}\chi_\Omega\big(x+t_1\vec
n\big)dxd\mathcal{H}^{N-1}(\vec n).
\end{multline}
In particular,
\begin{multline}
\label{gghgjhfgggjfgfhughGHGHGHKKjhggjhggjjhjgghjhhjhkoiojioijuoui}
\int\limits_{S^{N-1}}\int\limits_{G}\frac{1}{(t_1+t_2)^q}\Big|u\big(x+(t_1+t_2)\vec
n\big)-u(x)\Big|^q\chi_G\big(x+(t_1+t_2)\vec n\big)dxd\mathcal{H}^{N-1}(\vec n) \leq\\
\max\Bigg\{\int\limits_{S^{N-1}}\int\limits_{\Omega}\frac{\Big|u\big(x+t_2\vec
n\big)-u\big(x\big)\Big|^q}{t_2^q}\chi_\Omega\big(x+t_2\vec
n\big)dxd\mathcal{H}^{N-1}(\vec n)\,,\\
\,\int\limits_{S^{N-1}}\int\limits_{\Omega}\frac{\Big|u(x+t_1\vec
n)-u(x)\Big|^q}{t_1^q}\chi_\Omega\big(x+t_1\vec
n\big)dxd\mathcal{H}^{N-1}(\vec n)\Bigg\}.
\end{multline}
\end{lemma}
\begin{proof}
By the triangle inequality, for every $\vec n\in S^{N-1}$ and every
$h_1,h_2>0$ we have
\begin{multline}
\label{gghgjhfgggjfgfhughGHGHGHKKjhggzz}
\int_{G}\frac{1}{(h_1+h_2)^q}\Big|u\big(x+(h_1+h_2)\vec n\big)-u(x)\Big|^q\chi_G\big(x+(h_1+h_2)\vec n\big)dx=\\
\int_{\R^N}\frac{1}{(h_1+h_2)^q}\Big|u\big(x+(h_1+h_2)\vec
n\big)-u(x)\Big|^q\chi_G\big(x\big)\chi_G\big(x+(h_1+h_2)\vec
n\big)dx=
\\ \int_{\R^N}\frac{1}{(h_1+h_2)^q}\Big|u\big(x+(h_1+h_2)\vec n\big)-u(x+h_1\vec n)+u(x+h_1\vec n)-u(x)\Big|^q\chi_G\big(x\big)\chi_G\big(x+(h_1+h_2)\vec n\big)dx\leq\\
\int_{\R^N}\frac{1}{(h_1+h_2)^q}\bigg(\Big|u\big(x+(h_1+h_2)\vec
n\big)-u(x+h_1\vec n)\Big|+\Big|u(x+h_1\vec
n)-u(x)\Big|\bigg)^q\chi_G\big(x\big)\chi_G\big(x+(h_1+h_2)\vec
n\big)dx\\=
\int_{\R^N}\Bigg(\frac{h_2}{(h_1+h_2)}\frac{\Big|u\big(x+(h_1+h_2)\vec
n\big)-u(x+h_1\vec
n)\Big|}{h_2}\\+\frac{h_1}{(h_1+h_2)}\frac{\Big|u(x+h_1\vec
n)-u(x)\Big|}{h_1}\Bigg)^q\chi_G\big(x\big)\chi_G\big(x+(h_1+h_2)\vec
n\big)dx.
\end{multline}
Thus, by \er{gghgjhfgggjfgfhughGHGHGHKKjhggzz} and convexity of
$g(s):=|s|^q$, for every $\vec n\in S^{N-1}$ and every $h_1,h_2>0$
we deduce
\begin{multline}
\label{gghgjhfgggjfgfhughGHGHGHKKjhgg}
\int_{G}\frac{1}{(h_1+h_2)^q}\Big|u\big(x+(h_1+h_2)\vec
n\big)-u(x)\Big|^q\chi_G\big(x+(h_1+h_2)\vec n\big)dx\leq
\\
\int_{\R^N}\Bigg(\frac{h_2}{h_1+h_2}\bigg(\frac{\Big|u\big(x+(h_1+h_2)\vec
n\big)-u(x+h_1\vec
n)\Big|}{h_2}\bigg)^q\\+\frac{h_1}{h_1+h_2}\bigg(\frac{\Big|u(x+h_1\vec
n)-u(x)\Big|}{h_1}\bigg)^q\Bigg)\chi_G\big(x\big)
\chi_G\big(x+(h_1+h_2)\vec n\big)dx=\\
\frac{h_2}{h_1+h_2}\int_{\R^N}\frac{\Big|u\big(x+(h_1+h_2)\vec n\big)-u\big(x+h_1\vec n\big)\Big|^q}{h_2^q}\chi_G\big(x\big)\chi_G\big(x+(h_1+h_2)\vec n\big)dx\\
+ \frac{h_1}{h_1+h_2}\int_{\R^N}\frac{\Big|u(x+\vec n
h_1)-u(x)\Big|^q}{h_1^q}\chi_G\big(x\big)\chi_G\big(x+(h_1+h_2)\vec
n\big)dx .
\end{multline}
However, if $G\subset\Omega$ is convex then $x\in G$ and
$x+(h_1+h_2)\vec n\in G$ implies $x+h_1\vec n\in G$ and
then
\begin{multline}\label{hkjkgkfjfkkkghjggfhfhf}
\chi_G\big(x\big)\chi_G\big(x+(h_1+h_2)\vec
n\big)=\chi_G\big(x\big)\chi_G\big(x+h_1\vec
n\big)\chi_G\big(x+(h_1+h_2)\vec n\big)\\ \leq
\chi_\Omega\big(x\big)\chi_\Omega\big(x+h_1\vec
n\big)\chi_\Omega\big(x+(h_1+h_2)\vec n\big).
\end{multline}
On the other hand, if $h_1<\dist(G,\R^N\setminus\Omega)$, then $x\in
G$ implies $x+h_1\vec n\in \Omega$ and so we also deduce
\er{hkjkgkfjfkkkghjggfhfhf} in that case.
Thus, inserting \er{hkjkgkfjfkkkghjggfhfhf}
into
\er{gghgjhfgggjfgfhughGHGHGHKKjhgg}, in both cases we have
\begin{multline}
\label{gghgjhfgggjfgfhughGHGHGHKKjhggjhggjzz}
\int_{G}\frac{1}{(h_1+h_2)^q}\Big|u\big(x+(h_1+h_2)\vec
n\big)-u(x)\Big|^q\chi_G\big(x+(h_1+h_2)\vec n\big)dx
\leq\\
\frac{h_2}{(h_1+h_2)}\int_{\R^N}\frac{\Big|u\big(x+(h_1+h_2)\vec n\big)-u\big(x+h_1\vec n\big)\Big|^q}{h_2^q}\chi_\Omega\big(x\big)\chi_\Omega\big(x+h_1\vec n\big)\chi_\Omega\big(x+(h_1+h_2)\vec n\big)dx\\
+ \frac{h_1}{(h_1+h_2)}\int_{\R^N}\frac{\Big|u(x+h_1\vec
n)-u(x)\Big|^q}{h_1^q}\chi_\Omega\big(x\big)\chi_\Omega\big(x+h_1\vec
n\big)\chi_\Omega\big(x+(h_1+h_2)\vec n\big)dx.
\end{multline}
Therefore, since $\chi_\Omega\leq 1$ by
\er{gghgjhfgggjfgfhughGHGHGHKKjhggjhggjzz} we infer
\begin{multline}
\label{gghgjhfgggjfgfhughGHGHGHKKjhggjhggj}
\int_{G}\frac{1}{(h_1+h_2)^q}\Big|u\big(x+(h_1+h_2)\vec
n\big)-u(x)\Big|^q\chi_G\big(x+(h_1+h_2)\vec n\big)dx
\leq\\
\frac{h_2}{(h_1+h_2)}\int_{\R^N}\frac{\Big|u\big(x+(h_1+h_2)\vec n\big)-u\big(x+h_1\vec n\big)\Big|^q}{h_2^q}\chi_\Omega\big(x+h_1\vec n\big)\chi_\Omega\big(x+(h_1+h_2)\vec n\big)dx\\
+
\frac{h_1}{(h_1+h_2)}\int_{\R^N}\frac{\Big|u(x+h_1\vec n)-u(x)\Big|^q}{h_1^q}\chi_\Omega\big(x\big)\chi_\Omega\big(x+h_1\vec n\big)dx\\
=\frac{h_2}{(h_1+h_2)}\int_{\R^N}\frac{\Big|u\big(x+h_2\vec n\big)-u\big(x\big)\Big|^q}{h_2^q}\chi_\Omega\big(x\big)\chi_\Omega\big(x+h_2\vec n\big)dx\\
+ \frac{h_1}{(h_1+h_2)}\int_{\R^N}\frac{\Big|u(x+h_1\vec
n)-u(x)\Big|^q}{h_1^q}\chi_\Omega\big(x\big)\chi_\Omega\big(x+h_1\vec
n\big)dx.
\end{multline}
So, 
we deduce
\er{gghgjhfgggjfgfhughGHGHGHKKjhggjhggjjhjgghjhhjhkoioji}. In
particular, by
\er{gghgjhfgggjfgfhughGHGHGHKKjhggjhggjjhjgghjhhjhkoioji}
we finally obtain
\er{gghgjhfgggjfgfhughGHGHGHKKjhggjhggjjhjgghjhhjhkoiojioijuoui}.
%
%
%
%
%
%
\end{proof}

\begin{lemma}\label{gughfgfhfgdgddffddfKKzzbvqhig}
Let $\Omega\subset\R^N$ be an open set, $q\geq 1$, $u\in
L^q(\Omega,\R^m)$ and $t>0$. Furthermore, let $G\subset\Omega$ be an
open subset, such that $\mathcal{L}^N(\partial G)=0$ and, either
$t<\dist(G,\R^N\setminus\Omega)$ or $G$ is convex. Then, we have
\begin{multline}
\label{gghgjhfgggjfgfhughGHGHGHKKjhggjhggjjhjgghjhhjhkoiojiiuyyihjjg}
\int\limits_{S^{N-1}}\int\limits_{G}\frac{\big|u\big(x+t\vec
n\big)-u(x)\big|^q}{t^q}\chi_G\big(x+t\vec
n\big)dxd\mathcal{H}^{N-1}(\vec n)\\ \leq\liminf\limits_{\e\to
0^+}\Bigg(\int\limits_{S^{N-1}}\int\limits_{\Omega}\frac{\big|u(
x+\e\vec n)-u(x)\big|^q}{\e^q}\chi_{\Omega}(x+\e\vec
n)dxd\mathcal{H}^{N-1}(\vec n)\Bigg).
\end{multline}
\end{lemma}
\begin{proof}
First of all, in the case of convex $G$, for every $s>0$ and every
$j=1,2,\ldots$, taking $t_1=s$ and $t_2=js$ in
\er{gghgjhfgggjfgfhughGHGHGHKKjhggjhggjjhjgghjhhjhkoiojioijuoui},
with $G$ instead of $\Omega$, gives
\begin{multline}
\label{gghgjhfgggjfgfhughGHGHGHKKjhggjhggjjhjgghjhhjhkoiojioijuouiljjjhkkh1}
\int\limits_{S^{N-1}}\int\limits_{G}\frac{1}{\big((j+1)s\big)^q}\Bigg|u\Big(x+\big((j+1)s\big)\vec
n\Big)-u(x)\Bigg|^q\chi_G\Big(x+\big((j+1)s\big)\vec n\Big)dxd\mathcal{H}^{N-1}(\vec n) \leq\\
\max\Bigg\{\int\limits_{S^{N-1}}\int\limits_{G}\frac{\Big|u\big(x+s\vec
n\big)-u\big(x\big)\Big|^q}{s^q}\chi_G\big(x+s\vec
n\big)dxd\mathcal{H}^{N-1}(\vec n)\,,\\ \,
\int\limits_{S^{N-1}}\int\limits_{G}\frac{\Big|u\big(x+js\vec
n\big)-u\big(x\big)\Big|^q}{(js)^q}\chi_G\big(x+js\vec
n\big)dxd\mathcal{H}^{N-1}(\vec n) \Bigg\}.
\end{multline}
Therefore, using
\er{gghgjhfgggjfgfhughGHGHGHKKjhggjhggjjhjgghjhhjhkoiojioijuouiljjjhkkh1},
by induction, in the case of convex $G$ we prove
\begin{multline}
\label{gghgjhfgggjfgfhughGHGHGHKKjhggjhggjjhjgghjhhjhkoiojioijuouiljjjhkkhkljjjiouioui1}
\int\limits_{S^{N-1}}\int\limits_{G}\frac{1}{\big(js\big)^q}\Bigg|u\Big(x+\big(js\big)\vec
n\Big)-u(x)\Bigg|^q\chi_G\Big(x+\big(js\big)\vec n\Big)dxd\mathcal{H}^{N-1}(\vec n) \leq\\
\int\limits_{S^{N-1}}\int\limits_{G}\frac{\Big|u\big(x+s\vec
n\big)-u\big(x\big)\Big|^q}{s^q}\chi_G\big(x+s\vec
n\big)dxd\mathcal{H}^{N-1}(\vec
n)\leq\\
\int\limits_{S^{N-1}}\int\limits_{\Omega}\frac{\Big|u\big(x+s\vec
n\big)-u\big(x\big)\Big|^q}{s^q}\chi_\Omega\big(x+s\vec
n\big)dxd\mathcal{H}^{N-1}(\vec
n)\quad\quad\quad\quad\forall\,j=1,2,\ldots.
\end{multline}
On the other hand, for every $s>0$, every $j=1,2,\ldots$ and every
open $G_{j+1}\subset G_j\subset\Omega$, such that
$s<\dist(G_{j+1},\R^N\setminus G_j)$, taking $t_1=s$ and $t_2=js$ in
\er{gghgjhfgggjfgfhughGHGHGHKKjhggjhggjjhjgghjhhjhkoiojioijuoui}
gives
\begin{multline}
\label{gghgjhfgggjfgfhughGHGHGHKKjhggjhggjjhjgghjhhjhkoiojioijuouiljjjhkkh}
\int\limits_{S^{N-1}}\int\limits_{G_{j+1}}\frac{1}{\big((j+1)s\big)^q}\Bigg|u\Big(x+\big((j+1)s\big)\vec
n\Big)-u(x)\Bigg|^q\chi_{G_{j+1}}\Big(x+\big((j+1)s\big)\vec n\Big)dxd\mathcal{H}^{N-1}(\vec n) \leq\\
\max\Bigg\{\int\limits_{S^{N-1}}\int\limits_{G_j}\frac{\Big|u\big(x+s\vec
n\big)-u\big(x\big)\Big|^q}{s^q}\chi_{G_j}\big(x+s\vec
n\big)dxd\mathcal{H}^{N-1}(\vec n)\,,\\ \,
\int\limits_{S^{N-1}}\int\limits_{G_j}\frac{\Big|u\big(x+js\vec
n\big)-u\big(x\big)\Big|^q}{(js)^q}\chi_{G_j}\big(x+js\vec
n\big)dxd\mathcal{H}^{N-1}(\vec n) \Bigg\}
\leq\\
\max\Bigg\{\int\limits_{S^{N-1}}\int\limits_{\Omega}\frac{\Big|u\big(x+s\vec
n\big)-u\big(x\big)\Big|^q}{s^q}\chi_{\Omega}\big(x+s\vec
n\big)dxd\mathcal{H}^{N-1}(\vec n)\,,\\ \,
\int\limits_{S^{N-1}}\int\limits_{G_j}\frac{\Big|u\big(x+js\vec
n\big)-u\big(x\big)\Big|^q}{(js)^q}\chi_{G_j}\big(x+js\vec
n\big)dxd\mathcal{H}^{N-1}(\vec n) \Bigg\}.
\end{multline}
Therefore, using
\er{gghgjhfgggjfgfhughGHGHGHKKjhggjhggjjhjgghjhhjhkoiojioijuouiljjjhkkh},
by induction we prove that, for every $s>0$, every $j=1,2,\ldots$
and every $G_j\subset\Omega$, such that
$js<\dist(G_{j},\R^N\setminus \Omega)$, we have
\begin{multline}
\label{gghgjhfgggjfgfhughGHGHGHKKjhggjhggjjhjgghjhhjhkoiojioijuouiljjjhkkhkljjjiouioui}
\int\limits_{S^{N-1}}\int\limits_{G_j}\frac{1}{\big(js\big)^q}\Bigg|u\Big(x+\big(js\big)\vec
n\Big)-u(x)\Bigg|^q\chi_{G_j}\Big(x+\big(js\big)\vec n\Big)dxd\mathcal{H}^{N-1}(\vec n) \leq\\
\int\limits_{S^{N-1}}\int\limits_{\Omega}\frac{\Big|u\big(x+s\vec
n\big)-u\big(x\big)\Big|^q}{s^q}\chi_\Omega\big(x+s\vec
n\big)dxd\mathcal{H}^{N-1}(\vec
n)\quad\quad\quad\quad\forall\,j=1,2,\ldots.
\end{multline}

Next, assume that a sequence $\{\e_k\}_{k=1}^{+\infty}$ satisfies
$\e_k\downarrow 0$ and
\begin{multline}
\label{gghgjhfgggjfgfhughGHGHGHKKjhggjhggjjhjgghjhhjhkoiojiiuyyihjjgjj}
\lim\limits_{k\to
+\infty}\Bigg(\int\limits_{S^{N-1}}\int\limits_{\Omega}\frac{\big|u(
x+\e_k\vec n)-u(x)\big|^q}{\e^q_k}\chi_{\Omega}(x+\e_k\vec
n)dxd\mathcal{H}^{N-1}(\vec n)\Bigg)=\\
\liminf\limits_{\e\to
0^+}\Bigg(\int\limits_{S^{N-1}}\int\limits_{\Omega}\frac{\big|u(
x+\e\vec n)-u(x)\big|^q}{\e^q}\chi_{\Omega}(x+\e\vec
n)dxd\mathcal{H}^{N-1}(\vec n)\Bigg).
\end{multline}
Then, given $t>0$, for every $k\in \mathbb{N}$ consider $j_k\in
\mathbb{N}$ and $r_k\in[0,1)$ such that
\begin{equation}\label{hkjkgkfjfkkkghjggfhfhfhjgjghlkhigiukpoi}
\frac{t}{\e_k}=j_k+r_k\quad\quad\quad\quad\forall\,k\in\mathbb{N}\,,
\end{equation}
so that
\begin{equation}\label{hkjkgkfjfkkkghjggfhfhfhjgjghlkhigiuyigfhjljjjhjhjhhiyipiip}
t=(j_k+r_k)\e_k=j_k\e_k+r_k\e_k\quad\quad\quad\quad\forall\,k\in\mathbb{N}\,.
\end{equation}
In particular, since $\e_k\downarrow 0$ and $r_k\in[0,1)$ we deduce
\begin{equation}\label{hkjkgkfjfkkkghjggfhfhfhjgjghlkhigiuyigfhjljjjhjhjhhiyipiiphuyyu}
\lim_{k\to+\infty}j_k\e_k=t\,.
\end{equation}
Then, by
\er{hkjkgkfjfkkkghjggfhfhfhjgjghlkhigiuyigfhjljjjhjhjhhiyipiiphuyyu},
obviously we have
\begin{multline}
\label{gghgjhfgggjfgfhughGHGHGHKKjhggjhggjjhjgghjhhjhkoiojikljjlhjghgiojjj}
\lim\limits_{k\to+\infty}\int\limits_{S^{N-1}}\int\limits_{G}\frac{1}{(j_k\e_k)^q}\Bigg|u\bigg(x+\big(j_k\e_k\big)\vec
n\bigg)-u(x)\Bigg|^q\chi_G\bigg(x+\big(j_k\e_k\big)\vec
n\bigg)dxd\mathcal{H}^{N-1}(\vec n)=\\
\lim\limits_{k\to+\infty}\int\limits_{S^{N-1}}\int\limits_{G}\frac{1}{t^q}\Bigg|u\bigg(x+\big(j_k\e_k\big)\vec
n\bigg)-u(x)\Bigg|^q\chi_G\bigg(x+\big(j_k\e_k\big)\vec
n\bigg)dxd\mathcal{H}^{N-1}(\vec n)
=\\
\int\limits_{S^{N-1}}\int\limits_{G}\frac{1}{t^q}\Big|u\big(x+t\vec
n\big)-u(x)\Big|^q\chi_G\big(x+t\vec n\big)dxd\mathcal{H}^{N-1}(\vec
n),
\end{multline}
where we include the fact $\mathcal{L}^N(\partial G)=0$ in the proof
of the last equation. However, by either
\er{gghgjhfgggjfgfhughGHGHGHKKjhggjhggjjhjgghjhhjhkoiojioijuouiljjjhkkhkljjjiouioui1}
with $s=\e_k$ and $j=j_k$,
in the case of convex $G$, together with
\er{gghgjhfgggjfgfhughGHGHGHKKjhggjhggjjhjgghjhhjhkoiojiiuyyihjjgjj},
or by
\er{gghgjhfgggjfgfhughGHGHGHKKjhggjhggjjhjgghjhhjhkoiojioijuouiljjjhkkhkljjjiouioui},
with $s=\e_k$, $j=j_k$ and $G_j=G$, in the case of
$t<\dist(G,\R^N\setminus\Omega)$, together with
\er{gghgjhfgggjfgfhughGHGHGHKKjhggjhggjjhjgghjhhjhkoiojiiuyyihjjgjj}
and
\er{hkjkgkfjfkkkghjggfhfhfhjgjghlkhigiuyigfhjljjjhjhjhhiyipiiphuyyu},
we infer,
\begin{multline}
\label{gghgjhfgggjfgfhughGHGHGHKKjhggjhggjjhjgghjhhjhkoiojioijuouiljjjhkkhkljjjiouiouiiggguiou}
\lim\limits_{k\to+\infty}\int\limits_{S^{N-1}}\int\limits_{G}\frac{1}{(j_k\e_k)^q}\Bigg|u\bigg(x+\big(j_k\e_k\big)\vec
n\bigg)-u(x)\Bigg|^q\chi_G\bigg(x+\big(j_k\e_k\big)\vec
n\bigg)dxd\mathcal{H}^{N-1}(\vec n)\\ \leq\lim\limits_{k\to
+\infty}\Bigg(\int\limits_{S^{N-1}}\int\limits_{\Omega}\frac{\big|u(
x+\e_k\vec n)-u(x)\big|^q}{\e^q_k}\chi_{\Omega}(x+\e_k\vec
n)dxd\mathcal{H}^{N-1}(\vec n)\Bigg)=\\
\liminf\limits_{\e\to
0^+}\Bigg(\int\limits_{S^{N-1}}\int\limits_{\Omega}\frac{\big|u(
x+\e\vec n)-u(x)\big|^q}{\e^q}\chi_{\Omega}(x+\e\vec
n)dxd\mathcal{H}^{N-1}(\vec n)\Bigg).
\end{multline}
Therefore, by inserting
\er{gghgjhfgggjfgfhughGHGHGHKKjhggjhggjjhjgghjhhjhkoiojioijuouiljjjhkkhkljjjiouiouiiggguiou}
into
\er{gghgjhfgggjfgfhughGHGHGHKKjhggjhggjjhjgghjhhjhkoiojikljjlhjghgiojjj}
we finally obtain
\er{gghgjhfgggjfgfhughGHGHGHKKjhggjhggjjhjgghjhhjhkoiojiiuyyihjjg}.
\end{proof}

\begin{corollary}\label{gughfgfhfgdgddffddfKKzzbvqhigygygtyuu2}
Let $\Omega\subset\R^N$ be an open set, $q\geq 1$ and $u\in
L^q(\Omega,\R^m)$. Furthermore, let $G\subset\Omega$ be an open
subset, such that $\mathcal{L}^N(\partial G)=0$ and
$h:=\dist(G,\R^N\setminus\Omega)>0$. Then, we have
\begin{multline}
\label{gghgjhfgggjfgfhughGHGHGHKKjhggjhggjjhjgghjhhjhkoiojiiuyyihjjgy7yyujhij2}
\sup\limits_{\e\in(0,h)}\Bigg(\int\limits_{S^{N-1}}\int\limits_{G}\frac{\big|u(
x+\e\vec n)-u(x)\big|^q}{\e^q}\chi_{G}(x+\e\vec
n)dxd\mathcal{H}^{N-1}(\vec n)\Bigg)\\ \leq\liminf\limits_{\e\to
0^+}\Bigg(\int\limits_{S^{N-1}}\int\limits_{\Omega}\frac{\big|u(
x+\e\vec n)-u(x)\big|^q}{\e^q}\chi_{\Omega}(x+\e\vec
n)dxd\mathcal{H}^{N-1}(\vec n)\Bigg).
\end{multline}
\end{corollary}

\begin{corollary}\label{gughfgfhfgdgddffddfKKzzbvqhigygygtyuu}
Let $\Omega\subset\R^N$ be a convex open domain, such that
$\mathcal{L}^N(\partial\Omega)=0$, $q\geq 1$ and $u\in
L^q(\Omega,\R^m)$. Then, we have
\begin{multline}
\label{gghgjhfgggjfgfhughGHGHGHKKjhggjhggjjhjgghjhhjhkoiojiiuyyihjjgy7yyujhij}
\sup\limits_{\e\in(0,+\infty)}\Bigg(\int\limits_{S^{N-1}}\int\limits_{\Omega}\frac{\big|u(
x+\e\vec n)-u(x)\big|^q}{\e^q}\chi_{\Omega}(x+\e\vec
n)dxd\mathcal{H}^{N-1}(\vec n)\Bigg)\\=\limsup\limits_{\e\to
0^+}\Bigg(\int\limits_{S^{N-1}}\int\limits_{\Omega}\frac{\big|u(
x+\e\vec n)-u(x)\big|^q}{\e^q}\chi_{\Omega}(x+\e\vec
n)dxd\mathcal{H}^{N-1}(\vec n)\Bigg) \\=\liminf\limits_{\e\to
0^+}\Bigg(\int\limits_{S^{N-1}}\int\limits_{\Omega}\frac{\big|u(
x+\e\vec n)-u(x)\big|^q}{\e^q}\chi_{\Omega}(x+\e\vec
n)dxd\mathcal{H}^{N-1}(\vec n)\Bigg).
\end{multline}
In particular, for that case, if
$\rho_\e\big(|z|\big):\R^N\to[0,+\infty)$ are radial mollifiers, so
that $\int_{\R^N}\rho_\e\big(|z|\big)dz=1$ and for every $r>0$ there
exits $\delta:=\delta_r>0$, such that $\supp{(\rho_\e)}\subset
B_r(0)$ for every $\e\in(0,\delta_r)$, then by Lemma
\ref{gjyfyfuyyfifgyify} we have:
\begin{multline}\label{GMT'3jGHKKkkhjjhgzzZZzzZZzzbvq88nkhhjggjgjkpkjljluytytl;klljkljojkojjo;k;kklklklkiljluikljjhkjhhjjhpoi}
\frac{1}{\mathcal{H}^{N-1}(S^{N-1})}\,\lim\limits_{\e\to
0^+}\Bigg(\int\limits_{S^{N-1}}\int\limits_{\Omega}\chi_{\Omega}(x+\e\vec
n)\frac{\big|u( x+\e\vec
n)-u(x)\big|^q}{\e^q}dxd\mathcal{H}^{N-1}(\vec n)\Bigg)\\
=\lim\limits_{\e\to
0^+}\int\limits_{\Omega}\int\limits_{\Omega}\rho_\e\Big(|y-x|\Big)\frac{\big|u(
y)-u(x)\big|^q}{|y-x|^q}dydx
\\
=\frac{1}{\mathcal{H}^{N-1}(S^{N-1})}\,\sup\limits_{\e\in(0,+\infty)}\Bigg(\int\limits_{S^{N-1}}\int\limits_{\Omega}\chi_{\Omega}(x+\e\vec
n)\frac{\big|u( x+\e\vec
n)-u(x)\big|^q}{\e^q}dxd\mathcal{H}^{N-1}(\vec n)\Bigg).
\end{multline}
\end{corollary}

\subsection{The case $r\in(0,q)$}

The following Lemma is a part of the statement, that was proven in
\cite{jmp}:
\begin{lemma}\label{hjhhjKKzzbvq}
Let $\Omega\subset\R^N$ be an open set, $q\geq 1$ and let   $u\in
L^q_{loc}(\Omega,\R^m)$. Then,  for every open
$\Omega_1\subset\subset \Omega_2\subset\subset\Omega$, $\vec k\in
S^{N-1}$ and $\e$ satisfying
\begin{equation}
\label{eq:epsil} 0<\e<
\min{\big\{\dist(\Omega_1,\R^N\setminus\Omega_2),\dist(\Omega_2,\R^N\setminus\Omega)\big\}}\,,
\end{equation}
we have
\begin{equation}\label{gghgjhfgggjfgfhughGHGHKKzzjkjkyuyuybvq}
\int_{\ov\Omega_1}\frac{1}{\e}\Big|u(x+\e\vec k)-u(x)\Big|^qdx \leq
\frac{2^{N+q}}{\mathcal{L}^N({B_1(0)})}\int_{{B_1(0)}}\int_{\ov\Omega_2}\frac{1}{\e|z|}\Big|u(x+\e
z)-u\big(x)\Big|^qdxdz.
\end{equation}
\end{lemma}
\begin{corollary}\label{gjggjfhhffhfh}
Let $\Omega\subset\R^N$ be an open set, $q\geq 1$, $r>0$ and $u\in
L^q_{loc}(\Omega,\R^m)$. Next assume that open $\Omega_1\subset
\Omega_2\subset\Omega$ satisfy either $\Omega_1\subset\subset
\Omega_2\subset\subset\Omega$ or $\Omega_1=\Omega_2=\Omega=\R^N$.
Then we have
\begin{multline}\label{gghgjhfgggjfgfhughGHGHKKzzjkjkyuyuybvqjhgfhfhgjgj}
\limsup\limits_{\e\to 0^+}\Bigg(\sup\limits_{\vec k\in
S^{N-1}}\bigg\{\int_{\ov\Omega_1}\frac{1}{\e^r}\Big|u(x+\e\vec
k)-u(x)\Big|^qdx\bigg\}\Bigg) \leq\\
\frac{2^{N+q}}{(N-1+r)\mathcal{L}^N({B_1(0)})}\Bigg(\limsup\limits_{\e\to
0^+}\bigg\{\int_{S^{N-1}}\int_{\ov\Omega_2}\frac{\big|u( x+\e\vec
n)-u(x)\big|^q}{\e^r}dxd\mathcal{H}^{N-1}(\vec n)\bigg\}\Bigg).
\end{multline}
\end{corollary}
\begin{proof}
In the case $\Omega=\R^N$, we easily deduce from
\er{gghgjhfgggjfgfhughGHGHKKzzjkjkyuyuybvq} that for every $\e>0$ we
have
\begin{equation}\label{gghgjhfgggjfgfhughGHGHKKzzjkjkyuyuybvq33}
\int_{\R^N}\frac{1}{\e}\Big|u(x+\e\vec k)-u(x)\Big|^qdx \leq
\frac{2^{N+q}}{\mathcal{L}^N({B_1(0)})}\int_{{B_1(0)}}\int_{\R^N}\frac{1}{\e|z|}\Big|u(x+\e
z)-u\big(x)\Big|^qdxdz.
\end{equation}
Then, by either \er{gghgjhfgggjfgfhughGHGHKKzzjkjkyuyuybvq} for
every $\e$ satisfying \er{eq:epsil} or by
\er{gghgjhfgggjfgfhughGHGHKKzzjkjkyuyuybvq33} for every $\e>0$ we
have
\begin{multline}\label{gghgjhfgggjfgfhughGHGHKKzzjkjkyuyuybvqjhgfhfh}
\sup\limits_{\vec k\in
S^{N-1}}\Bigg\{\int_{\ov\Omega_1}\frac{1}{\e^r}\Big|u(x+\e\vec
k)-u(x)\Big|^qdx\Bigg\} \leq
\frac{2^{N+q}}{\mathcal{L}^N({B_1(0)})}\int_{{B_1(0)}}\int_{\ov\Omega_2}\frac{1}{\e^r|z|}\Big|u(x+\e
z)-u\big(x)\Big|^qdxdz\\
=
\frac{2^{N+q}}{\mathcal{L}^N({B_1(0)})}\int\limits_{S^{N-1}}\int\limits_{0}^{1}\int\limits_{\ov\Omega_2}s^{N-1}\frac{\big|u(
x+\e s\vec n)-u(x)\big|^q}{\e^rs}dxdsd\mathcal{H}^{N-1}(\vec n) \\
=
\frac{2^{N+q}}{\mathcal{L}^N({B_1(0)})}\int\limits_{S^{N-1}}\int\limits_{0}^{\e}\int\limits_{\ov\Omega_2}\frac{t^{N-1}}{\e^{N-1}}\frac{\big|u(
x+t\vec n)-u(x)\big|^q}{\e^rt}dxdtd\mathcal{H}^{N-1}(\vec n) \\
\leq
\frac{2^{N+q}}{\mathcal{L}^N({B_1(0)})}\Bigg(\sup\limits_{t\in(0,\e)}\bigg\{\int_{S^{N-1}}\int_{\ov\Omega_2}\frac{\big|u(
x+t\vec n)-u(x)\big|^q}{t^r}dxd\mathcal{H}^{N-1}(\vec
n)\bigg\}\Bigg)\int\limits_{0}^{\e}\frac{\tau^{N-2+r}}{\e^{N-1+r}}d\tau
\\
=
\frac{2^{N+q}}{(N-1+r)\mathcal{L}^N({B_1(0)})}\Bigg(\sup\limits_{t\in(0,\e)}\bigg\{\int_{S^{N-1}}\int_{\ov\Omega_2}\frac{\big|u(
x+t\vec n)-u(x)\big|^q}{t^r}dxd\mathcal{H}^{N-1}(\vec
n)\bigg\}\Bigg)\,.
\end{multline}
In particular, by \er{gghgjhfgggjfgfhughGHGHKKzzjkjkyuyuybvqjhgfhfh}
we deduce \er{gghgjhfgggjfgfhughGHGHKKzzjkjkyuyuybvqjhgfhfhgjgj}.
\end{proof}

\begin{definition}\label{gjghghghjgghGHKKzzbvqKKkk}
Given a compact set $ \ov U\subset\subset\O$ let
\begin{equation}
\begin{aligned}
B_{u,q,r}\big(\ov U\big):=\limsup\limits_{\e\to
0^+}\sup\limits_{\vec k\in S^{N-1}}\int_{\ov
U}\frac{1}{\e^r}\Big|u(x+\e\vec k)-u(x)\Big|^qdx
\label{GMT'1jGHKKzzbvq}.
\end{aligned}
\end{equation}
Next, given an open set $\Omega\subset\R^N$ define
\begin{equation}
\label{hjhjhghgjkghggghGHKKjjzzbvqjj}
B_{u,q,r}\big(\Omega\big):=\sup\limits_{K\subset\subset\Omega}B_{u,q,r}\big(K\big).
\end{equation}
Finally, set
\begin{equation}
\label{GMT'3jGHKKkkhjjhgzzZZzzbvqkk} \hat
B_{u,q,r}\big(\R^N\big):=\limsup\limits_{\e\to 0^+}\sup\limits_{\vec
k\in S^{N-1}}\int_{\R^N}\frac{1}{\e^r}\Big|u(x+\e\vec
k)-u(x)\Big|^qdx.
\end{equation}
\end{definition}
The following result is known; for the convenience of a reader we
will give its proof.
\begin{lemma}\label{hjgjg}
For any $q\geq 1$ and $r\in(0,q)$, a function $u\in L^q(\R^N,\R^m)$
belongs to $B_{q,\infty}^{r/q}(\R^N,\R^m)$ if and only if $\hat
B_{u,q,r}\big(\R^N\big)<+\infty$. Moreover, for any  open
$\Omega\subset\R^N$, a function $u\in L^q_{loc}(\Omega,\R^m)$
belongs to $\big(B_{q,\infty}^{r/q}\big)_{loc}(\Omega,\R^m)$ if and
only if for every compact $K\subset\subset\Omega$ we have
$B_{u,q,r}\big(K\big)<+\infty$.
\end{lemma}
\begin{proof}
We have,
\begin{multline}\label{gjhgjhfghffhnmmbbbkhjhhkljkjk}
\sup\limits_{\rho\in(0,\infty)}\Bigg(\sup_{|h|=\rho}\int_{\mathbb{R}^N}\bigg(\frac{1}{\rho^r}\big|u(x+h)-u(x)\big|\bigg)^qdx\Bigg)\leq
\sup\limits_{\rho\in(0,\infty)}\Bigg(\sup_{|h|\leq\rho}\int_{\mathbb{R}^N}\bigg(\frac{1}{\rho^r}\big|u(x+h)-u(x)\big|\bigg)^qdx\Bigg)\\=
\sup\limits_{\rho\in(0,\infty)}\sup_{t\in(0,\rho]}\Bigg(\sup_{|h|=t}\int_{\mathbb{R}^N}\bigg(\frac{1}{\rho^r}\big|u(x+h)-u(x)\big|\bigg)^qdx\Bigg)\\
\leq
\sup\limits_{\rho\in(0,\infty)}\Bigg(\sup_{t\in(0,\rho]}\bigg(\sup_{|h|=t}\int_{\mathbb{R}^N}\Big(\frac{1}{t^r}\big|u(x+h)-u(x)\big|\Big)^qdx\bigg)\Bigg)
=\\
\sup\limits_{\rho\in(0,\infty)}\Bigg(\sup_{|h|=\rho}\int_{\mathbb{R}^N}\bigg(\frac{1}{\rho^r}\big|u(x+h)-u(x)\big|\bigg)^qdx\Bigg)
=\\
\sup\limits_{\rho\in(0,\infty)}\Bigg(\sup_{\vec k\in
S^{N-1}}\int_{\mathbb{R}^N}\bigg(\frac{1}{\rho^r}\big|u(x+\rho\vec
k)-u(x)\big|\bigg)^qdx\Bigg).
\end{multline}
On the other hand by the triangle inequality and the convexity of
$g(s):=|s|^q$ for every $\delta>0$ we have,
\begin{multline}\label{gjhgjhfghffhj}
\sup\limits_{\rho\in(0,\delta)}\bigg(\sup_{\vec k\in
S^{N-1}}\int_{\mathbb{R}^N}\frac{1}{\rho^r}\big|u(x+\rho\vec
k)-u(x)\big|^qdx\bigg)\leq\sup\limits_{\rho\in(0,\infty)}\bigg(\sup_{\vec
k\in S^{N-1}}\int_{\mathbb{R}^N}\frac{1}{\rho^r}\big|u(x+\rho\vec
k)-u(x)\big|^qdx\bigg)\leq\\
\sup\limits_{\rho\in(0,\delta)}\bigg(\sup_{\vec k\in
S^{N-1}}\int_{\mathbb{R}^N}\frac{1}{\rho^r}\big|u(x+\rho\vec
k)-u(x)\big|^qdx\bigg)+\sup\limits_{\rho\in[\delta,\infty)}\bigg(\sup_{\vec
k\in S^{N-1}}\int_{\mathbb{R}^N}\frac{1}{\rho^r}\big|u(x+\rho\vec
k)-u(x)\big|^qdx\bigg)\\ \leq
\sup\limits_{\rho\in(0,\delta)}\bigg(\sup_{\vec k\in
S^{N-1}}\int_{\mathbb{R}^N}\frac{1}{\rho^r}\big|u(x+\rho\vec
k)-u(x)\big|^qdx\bigg)\\+\frac{2^{q-1}}{\delta^r}\sup\limits_{\rho\in[\delta,\infty)}\bigg(\sup_{\vec
k\in S^{N-1}}\int_{\mathbb{R}^N}\Big(\big|u(x+\rho\vec
k)\big|^q+\big|u(x)\big|^q\Big)dx\bigg)\\=
\sup\limits_{\rho\in(0,\delta)}\bigg(\sup_{\vec k\in
S^{N-1}}\int_{\mathbb{R}^N}\frac{1}{\rho^r}\big|u(x+\rho\vec
k)-u(x)\big|^qdx\bigg)+\frac{2^{q}}{\delta^r}\big\|u\big\|^q_{L^q(\R^N,\R^d)}.
\end{multline}
Therefore, by \er{gjhgjhfghffhnmmbbbkhjhhkljkjk} and
\er{gjhgjhfghffhj} we have:
\begin{multline}\label{gjhgjhfghffhjj}
\sup\limits_{\e\in(0,\delta)}\bigg(\sup_{\vec k\in
S^{N-1}}\int_{\mathbb{R}^N}\frac{1}{\e^r}\big|u(x+\e\vec
k)-u(x)\big|^qdx\bigg)\leq
\sup\limits_{\rho\in(0,\infty)}\Bigg(\sup_{|h|\leq\rho}\int_{\mathbb{R}^N}\bigg(\frac{1}{\rho^{(r/q)}}\big|u(x+h)-u(x)\big|\bigg)^qdx\Bigg)
\\
\leq \sup\limits_{\e\in(0,\delta)}\bigg(\sup_{\vec k\in
S^{N-1}}\int_{\mathbb{R}^N}\frac{1}{\e^r}\big|u(x+\e\vec
k)-u(x)\big|^qdx\bigg)+\frac{2^{q}}{\delta^r}\big\|u\big\|^q_{L^q(\R^N,\R^d)}.
\end{multline}
Thus by \er{gjhgjhfghffhjj} we clearly obtain that if $u\in
L^q(\mathbb{R}^N,\R^m)$ then
\begin{multline}\label{gjhgjhfghffhhgjgjjk}
\sup\limits_{\rho\in(0,\infty)}\Bigg(\sup_{|h|\leq\rho}\int_{\mathbb{R}^N}\bigg(\frac{1}{\rho^{(r/q)}}\big|u(x+h)-u(x)\big|\bigg)^qdx\Bigg)
<+\infty\quad\text{if and only if}\\ \quad \limsup\limits_{\e\to
0^+}\bigg(\sup_{\vec k\in
S^{N-1}}\int_{\mathbb{R}^N}\frac{1}{\e^r}\big|u(x+\e\vec
k)-u(x)\big|^qdx\bigg)<+\infty.
\end{multline}
So we proved that $u\in L^q(\R^N,\R^m)$ belongs to
$B_{q,\infty}^{r/q}(\R^N,\R^m)$ if and only if we have $\hat
B_{u,q,r}\big(\R^N\big)<+\infty$.

Next, given open $\Omega\subset\R^N$ let $u\in
L^q_{loc}(\Omega,\R^m)$ and $K\subset\subset\Omega$ be a compact
set. Moreover, consider an open set $U\subset\R^N$ such that we have
the following compact embedding: $$K\subset\subset U\subset\ov
U\subset\subset\Omega.$$ Then, assuming
$u\in\big(B_{q,\infty}^{r/q}\big)_{loc}(\Omega,\R^m)$ implies
existence of $\hat u\in B_{q,\infty}^{r/q}(\mathbb{R}^N,\R^m)$ such
that $\hat u(x)= u(x)$ for every $x\in \bar U$, that gives
$$B_{u,q,r}\big(K\big)=B_{\hat u,q,r}\big(K\big)\leq \hat B_{\hat u,q,r}\big(\R^N\big)<+\infty.$$
On the other hand, if we assume
\begin{equation}\label{jjhhkghjgjfhfhfh}
B_{u,q,r}\big(\ov U\big)<+\infty\,,
\end{equation}
then define
\begin{equation}\label{jjhhkghjgjfhfhfhn,nn}
\hat u(x)=\begin{cases}\eta(x)u(x)\quad\forall x\in U\\ 0
\quad\forall x\in \R^N\setminus U,
\end{cases}
\end{equation}
where $\eta(x)\in C^\infty_c\big(U,[0,1]\big)$ is some cut-off
function such that $\eta(x)=1$ for every $x\in K$. Thus in
particular $\hat u(x)= u(x)$ for every $x\in K$ and so, in order to
complete the proof, we need just to show that $\hat u\in
B_{q,\infty}^{r/q}(\mathbb{R}^N,\R^m)$. Thus by
\er{gjhgjhfghffhhgjgjjk} it is sufficient to show:
\begin{equation}\label{gjhgjhfghffhhgjgjjklkhjhhj}
\limsup\limits_{\e\to 0^+}\bigg(\sup_{\vec k\in
S^{N-1}}\int_{\mathbb{R}^N}\frac{1}{\e^r}\big|\hat u(x+\e\vec
k)-\hat u(x)\big|^qdx\bigg)<+\infty.
\end{equation}
However,
since $|\eta|\leq 1$, $\supp\eta\subset\subset U$ and $\eta$ is
smooth we have:
\begin{multline}\label{gjhgjhfghffhhgjgjjklkhjhhjl;klkl;}
\limsup\limits_{\e\to 0^+}\bigg(\sup_{\vec k\in
S^{N-1}}\int_{\mathbb{R}^N}\frac{1}{\e^r}\big|\hat u(x+\e\vec
k)-\hat
u(x)\big|^qdx\bigg)=\\
\limsup\limits_{\e\to 0^+}\bigg(\sup_{\vec k\in
S^{N-1}}\int_{U}\frac{1}{\e^r}\big|\eta(x+\e\vec k)u(x+\e\vec
k)-\eta(x) u(x)\big|^qdx\bigg)
=\\
\limsup\limits_{\e\to 0^+}\bigg(\sup_{\vec k\in S^{N-1}}\int_{
U}\frac{1}{\e^r}\Big|\eta(x+\e\vec k)\big(u(x+\e\vec
k)-u(x)\big)+\big(\eta(x+\e\vec k)-\eta(x)\big) u(x)\Big|^qdx\bigg)
\\
\leq \limsup\limits_{\e\to 0^+}\Bigg(\sup_{\vec k\in S^{N-1}}\int_{
U}\frac{2^{q-1}}{\e^r}\bigg(\Big|\eta(x+\e\vec k)\big(u(x+\e\vec
k)-u(x)\big)\Big|^q+\Big|\big(\eta(x+\e\vec
k)-\eta(x)\big) u(x)\Big|^q\bigg)dx\Bigg)\\
\leq 2^{q-1}\limsup\limits_{\e\to 0^+}\Bigg(\sup_{\vec k\in
S^{N-1}}\int_{U}\frac{1}{\e^r}\bigg(\Big|u(x+\e\vec
k)-u(x)\Big|^q\bigg)dx\Bigg)+\\
 2^{q-1}\limsup\limits_{\e\to 0^+}\Bigg(\e^{q-r}\sup_{\vec k\in
S^{N-1}}\int_{U}\bigg|\frac{\big(\eta(x+\e\vec
k)-\eta(x)\big)}{\e}\bigg|^q\big| u(x)\big|^qdx\Bigg)\\
\leq 2^{q-1}B_{u,q,r}\big(\ov U\big)+2^{q-1}\bigg(\int_{U}\big|
u(x)\big|^qdx\bigg)\big\|\nabla\eta\big\|^q_{L^\infty}\Big(\limsup_{\e\to
0^+}\e^{q-r}\Big)=2^{q-1}B_{u,q,r}\big(\ov U\big)<+\infty.
\end{multline}
\end{proof}

The next Proposition is part of  \cite[Proposition 2.4]{jmp}, see \cite{jmp} for the proof.
\begin{proposition}\label{hgugghghhffhfhKKzzbvq}
Let $\Omega\subset\R^N$ be an open set and let $u\in
BV_{loc}(\Omega,\R^m)\cap L^\infty_{loc}(\Omega,\R^m)$. Then, for
every compact set $K\subset\subset\Omega$
such that $\|Du\|(\partial K)=0$, any $q>1$ and any vector $\vec
k\subset S^{N-1}$ we have
\begin{equation}\label{fgyufghfghjgghgjkhkkGHGHKKzzbvq}
\lim_{\e\to0^+}\int_K\frac{1}{\e}\Big|u(x+\e\vec
k)-u(x)\Big|^qdx=\int_{J_u\cap K}\Big|u^+(x)-u^-(x)\Big|^q\big|\vec
k\cdot\vec\nu(x)\big|d\mathcal{H}^{N-1}(x).
\end{equation}
\end{proposition}

\begin{proof}[Proof of Proposition \ref{hgugghghhffhfhKKzzbvqhkjjgg}]
For any compact set $K\subset\subset\Omega$, we can choose
$\Omega_1\subset\subset\Omega$, such that $K\subset\subset\Omega_1$
and then, for every small $\e>0$ and any vector $\vec k\subset
S^{N-1}$ we clearly have
\begin{multline}\label{hgjhkjgfgjffzzbvq}
0\leq\frac{1}{\e}\int_K\Big|u(x+\e\vec
k)-u(x)\Big|^qdx\leq2^{q-1}\|u\|^{q-1}_{L^\infty(\ov
\Omega_1)}\int_K\frac{1}{\e}\Big|u(x+\e\vec k)-u(x)\Big|dx\\ \leq
2^{q-1}\|u\|^{q-1}_{L^\infty(\ov \Omega_1)}\|u\|_{BV(\Omega_1)}.
\end{multline}
Thus by dominated convergence, by
\er{fgyufghfghjgghgjkhkkGHGHKKzzbvq} in Proposition
\ref{hgugghghhffhfhKKzzbvq} we get
\begin{multline}\label{fgyufghfghjgghgjkhkkGHGHKKjjjjkjkkjjhjkzzbvq}
\lim\limits_{\e\to 0^+}\Bigg\{\int_{S^{N-1}}\int_{K}\frac{\big|u(
x+\e\vec n)-u(x)\big|^q}{\e}dxd\mathcal{H}^{N-1}(\vec
n)\Bigg\}\\=\int_{S^{N-1}}\Bigg(\int_{J_u\cap
K}\Big|u^+(x)-u^-(x)\Big|^q\big|\vec
n\cdot\vec\nu(x)\big|d\mathcal{H}^{N-1}(x)\Bigg)d\mathcal{H}^{N-1}(\vec
n)=\\
\Bigg(\int_{J_u\cap
K}\Big|u^+(x)-u^-(x)\Big|^q\bigg(\int_{S^{N-1}}\big|\vec
n\cdot\vec\nu(x)\big|d\mathcal{H}^{N-1}(\vec
n)\bigg)d\mathcal{H}^{N-1}(x)\Bigg)=\\
\bigg(\int_{S^{N-1}}|z_1|d\mathcal{H}^{N-1}(z)\bigg)\Bigg(\int_{J_u\cap
K}\Big|u^+(x)-u^-(x)\Big|^qd\mathcal{H}^{N-1}(x)\Bigg)\,.
\end{multline}

 Next, if $\Omega$ is
an open set with bounded Lipschitz boundary and  $u\in
BV(\Omega,\R^m)\cap L^\infty(\Omega,\R^m)$, then
we can extend the function $u(x)$ to all of $\R^N$ in such a way
that $u\in BV(\R^N,\R^m)\cap L^\infty(\R^N,\R^m)$ and $\|D
u\|(\partial\Omega)=0$. Next, consider an increasing sequence of
compact sets $K_n\subset\subset\Omega$ such that $K_n\subset
K_{n+1}$, $\dist(K_n,\R^N\setminus K_{n+1})>0$ and
$\bigcup_{n=1}^{+\infty}K_n=\Omega$. Moreover, consider a decreasing
sequence of open sets $V_n$ such that $\ov\Omega\subset V_n$,
 $V_{n+1}\subset V_{n}$, $\dist(V_{n+1},\R^N\setminus V_{n})>0$ and
$\bigcap_{n=1}^{+\infty}V_n=\ov\Omega$, so that, denoting an open
set $U_n:=V_n\setminus K_n$, we have $U_{n+1}\subset U_{n}$
$\dist(U_{n+1},\R^N\setminus U_{n})>0$ and
$\bigcap_{n=1}^{+\infty}U_n=\partial\Omega$. Then, by
\er{fgyufghfghjgghgjkhkkGHGHKKjjjjkjkkjjhjkzzbvq} for every $n$ we
have
\begin{multline}\label{gghgjhfgggjfgfhughGHGHKKzzjkjkyuyuybvqjhgfhfhgjgjjlhkluyi}
\bigg(\int_{S^{N-1}}|z_1|d\mathcal{H}^{N-1}(z)\bigg)\Bigg(\int_{J_u\cap
K_n}\Big|u^+(x)-u^-(x)\Big|^qd\mathcal{H}^{N-1}(x)\Bigg)=\\
\lim\limits_{\e\to 0^+}\Bigg\{\int_{S^{N-1}}\int_{K_n}\frac{\big|u(
x+\e\vec
n)-u(x)\big|^q}{\e}dxd\mathcal{H}^{N-1}(\vec n)\Bigg\}\leq\\
\limsup\limits_{\e\to
0^+}\Bigg\{\int_{S^{N-1}}\int_{\Omega}\chi_\Omega( x+\e\vec
n)\frac{\big|u( x+\e\vec
n)-u(x)\big|^q}{\e}dxd\mathcal{H}^{N-1}(\vec n)\Bigg\}\leq\\
\limsup\limits_{\e\to
0^+}\Bigg\{\int_{S^{N-1}}\int_{\Omega}\frac{\big|u( x+\e\vec
n)-u(x)\big|^q}{\e}dxd\mathcal{H}^{N-1}(\vec n)\Bigg\}\,.
\end{multline}
On the other hand, as in \er{hgjhkjgfgjffzzbvq},  for every small
$\e>0$ and any vector $\vec k\subset S^{N-1}$ we clearly have
\begin{multline}\label{hgjhkjgfgjffzzbvqhkjghkhh}
0\leq\frac{1}{\e}\int_{\Omega\setminus K_n}\Big|u(x+\e\vec
k)-u(x)\Big|^qdx\leq2^{q-1}\|u\|^{q-1}_{L^\infty(\R^N)}\int_{\Omega\setminus
K_n}\frac{1}{\e}\Big|u(x+\e\vec k)-u(x)\Big|dx\\ \leq
2^{q-1}\|u\|^{q-1}_{L^\infty(\R^N)}\|u\|_{BV(U_n)}.
\end{multline}
In particular,
\begin{multline}\label{hgjhkjgfgjffzzbvqhkjghkhhjhjgjg}
0\leq \limsup\limits_{\e\to
0^+}\Bigg\{\int_{S^{N-1}}\int_{\Omega\setminus K_n}\frac{\big|u(
x+\e\vec n)-u(x)\big|^q}{\e}dxd\mathcal{H}^{N-1}(\vec
n)\Bigg\}\leq\\
2^{q-1}\Big(\mathcal{H}^{N-1}(S^{N-1})\Big)\|u\|^{q-1}_{L^\infty(\R^N)}\|u\|_{BV(U_n)}.
\end{multline}
Thus, since $\bigcap_{n=1}^{+\infty}U_n=\partial\Omega$, letting
$n\to+\infty$ in \er{hgjhkjgfgjffzzbvqhkjghkhhjhjgjg} and using the
fact $\|Du\|(\partial\Omega)=0$ gives
\begin{multline}\label{hgjhkjgfgjffzzbvqhkjghkhhjhjgjgjhghghhjjggj}
0\leq \limsup_{n\to+\infty}\Bigg(\limsup\limits_{\e\to
0^+}\bigg\{\int_{S^{N-1}}\int_{\Omega\setminus K_n}\frac{\big|u(
x+\e\vec n)-u(x)\big|^q}{\e}dxd\mathcal{H}^{N-1}(\vec
n)\bigg\}\Bigg)\leq\\
2^{q-1}\Big(\mathcal{H}^{N-1}(S^{N-1})\Big)\|u\|^{q-1}_{L^\infty(\R^N)}\|Du\|(\partial\Omega)=0.
\end{multline}
Thus, inserting \er{hgjhkjgfgjffzzbvqhkjghkhhjhjgjgjhghghhjjggj}
into \er{gghgjhfgggjfgfhughGHGHKKzzjkjkyuyuybvqjhgfhfhgjgjjlhkluyi}
we finally deduce
\er{gghgjhfgggjfgfhughGHGHKKzzjkjkyuyuybvqjhgfhfhgjgjjlhkluyikhhkhkhkjgjhhh}.
\end{proof}

\end{document}